\documentclass[11 pt,fullpage]{article}
\usepackage{amsfonts,amssymb}
\usepackage{amsmath,amsthm}
\usepackage{color} 
\usepackage{hyperref}
\usepackage{mathtools}
\usepackage[right = 2.5cm, left=2.5cm, top = 2.5cm, bottom =2.5cm]{geometry}
\theoremstyle{plain}
\newtheorem{theorem}{Theorem}

\newtheorem{proposition}{Proposition}[section]
\newtheorem{lemma}[proposition]{Lemma}

\theoremstyle{definition}
\newtheorem{remark}{Remark}

\allowdisplaybreaks

\numberwithin{equation}{section}

\newcommand\Torus{{\mathbb T}}

\newcommand{\cC}{\mathcal C}
\newcommand{\cD}{\mathcal D}
\newcommand{\cE}{\mathcal E}

\newcommand{\cH}{\mathcal H}
\newcommand{\cI}{\mathcal I}
\newcommand{\cJ}{\mathcal J}
\newcommand{\cK}{\mathcal K}
\newcommand{\cL}{\mathcal L}
\newcommand{\cM}{\mathcal M}
\newcommand{\cN}{\mathcal N}

\def\eps{{\varepsilon}}

\renewcommand\eps{\epsilon }

\newcommand{\Real}{\mathbb R}
\newcommand{\Complex}{\mathbb C}
\newcommand{\Integer}{\mathbb Z}
\newcommand{\Integers}{\mathbb Z}
\newcommand{\norm}[1]{\left\lVert#1\right\rVert}
\newcommand{\abs}[1]{\left\vert#1\right\vert}
\newcommand{\set}[1]{\left\{#1\right\}}

\newcommand{\grad}{\nabla}

\newcommand{\Naturals}{\mathbb N}
\newcommand{\jap}[1]{\langle #1 \rangle} 
\newcommand{\brak}[1]{\langle #1 \rangle} 

\newcommand{\ddv}{\dd v}

\newcommand{\dss}{\displaystyle}

 \newcommand{\dd}{{\, \mathrm d}}

\makeatletter
\def\blfootnote{\xdef\@thefnmark{}\@footnotetext}
\makeatother

\begin{document}
 
\title{Suppression of plasma echoes and Landau damping in Sobolev spaces by weak collisions in a Vlasov-Fokker-Planck equation} 
\author{Jacob Bedrossian\footnote{\textit{jacob@cscamm.umd.edu}, University of Maryland, College Park. The author was partially supported by NSF CAREER grant DMS-1552826, NSF DMS-1413177, and a Sloan research fellowship. Additionally, the research was supported in part by NSF RNMS \#1107444 (Ki-Net).}}

\date{\today}
\maketitle

\begin{abstract} 
In this paper, we study Landau damping in the weakly collisional limit of a Vlasov-Fokker-Planck equation with nonlinear collisions in the phase-space $(x,v) \in \Torus_x^n \times \Real^n_v$. 
The goal is four-fold: (A) to understand how collisions suppress plasma echoes and enable Landau damping in agreement with linearized theory in Sobolev spaces,
(B) to understand how phase mixing accelerates collisional relaxation, (C) to understand better how the plasma returns to global equilibrium during Landau damping, and (D) to rule out that collision-driven nonlinear instabilities dominate. 
 We give an estimate for the scaling law between Knudsen number and the maximal size of the perturbation necessary for linear theory to be accurate in Sobolev regularity. 
We conjecture this scaling to be sharp (up to logarithmic corrections) due to  potential nonlinear echoes in the collisionless model. 
\end{abstract}

\setcounter{tocdepth}{1}
{\small\tableofcontents}

\section{Introduction}
In this paper we consider the weakly collisional limit of a single-species Vlasov-Poisson with nonlinear Fokker-Planck collisions near a global Maxwellian in the phase-space $(x,v) \in \Torus_x^n \times \Real_v^n$ (with $\Torus_x$ normalized to length $2\pi$).
Specifically, if $F(t,x,v)$ is the distribution function, then we consider the following single-species model with a neutralizing background: 
\begin{equation}\label{def:VPE}
  \left\{
\begin{array}{l} \dss 
\partial_t F + v\cdot \grad_x F + E(t,x)\cdot \grad_v F  = \nu \bar{\rho} (\bar{T}\Delta_v F + \grad_v \cdot ((v-u)F) ) \\% [3mm]
E(t,x) = -\grad_x (-\Delta)^{-1}_x(\bar{\rho}-1) \\ 
\bar{\rho}(t,x) = \int F(t,x,v) \dd v , \\
\bar{\rho} u(t,x)  = \int F(t,x,v) v \dd v \\ 
\frac{1}{2}\bar{\rho} \abs{u}^2 + \frac{1}{2} n \bar{\rho} \bar{T}  = \frac{1}{2}\int F \abs{v}^2 \dd v \\  
F(t=0,x,v) = F_{\mbox{{\scriptsize in}}}(x,v). 
\end{array}
\right.
\end{equation}
Here, $\nu > 0$ is the inverse Knudsen number and $\bar{\rho}, u,$ and $\bar{T}$ respectively denote the associated hydrodynamic fields of density, mean velocity, and temperature.
We study the long-time ($t \rightarrow \infty$), weak collisionality limit ($\nu \rightarrow 0$) of \eqref{def:VPE} near global equilibrium, e.g. $F(t,x,v) = \mu(v) + h(t,x,v)$ for $h$ initially small in Sobolev spaces.
 
In many plasma physics settings, the collisions are very weak, that is $\nu \ll 1$, and there are many situations when physicists neglect them entirely (see e.g. the classical books \cite{GoldstonRutherford95,BoydSanderson}). 
However, exactly how weak collisions and collisionless effects such as Landau damping interact, especially when nonlinearity is accounted for, has been the subject of study and debate in the physics community for almost 60 years (see e.g. \cite{LenardBernstein1958,SuOberman1968,ONeil1968,Johnston1971,NgBhattacharjee1999,NgBhattacharjee2006,ShortSimon2002,Callen2014} for discussions). 
The purpose of our work is to provide a mathematical study of this interaction and to understand the role it plays in suppressing nonlinear instabilities.  
The only other mathematical work devoted to weakly collisional limits known to the author is the recent work of Tristani \cite{Tristani2016}, which studies a linear analogue of \eqref{def:VPE} but does not quantify the interaction between the collisions and the collisionless effects. 

In the case $\nu = 0$, \eqref{def:VPE} reduces to the Vlasov equations. 
Of particular note is the behavior known as \emph{Landau damping} -- the rapid decay of the electric field despite the lack of dissipative mechanisms \cite{Landau46,Ryutov99,BoydSanderson,MouhotVillani11}.
Landau damping was discovered first for the linearized equations by Landau in 1946 \cite{Landau46} and was later observed in experiments \cite{MalmbergWharton64,MalmbergWharton68}, and is now a fundamentally important property of collisionless plasmas (see e.g. \cite{Ryutov99,BoydSanderson,Stix,MouhotVillani11}). 
In \cite{MouhotVillani11}, Landau's linearized analysis was confirmed, mathematically rigorously, to be accurate uniformly in time for the nonlinear Vlasov equations for initial data which was analytic or Gevrey class of sufficiently low index (see also earlier work of \cite{CagliotiMaffei98,HwangVelazquez09} and the more recent \cite{BMM13,Young14}). 
Landau damping is connected to the mixing in phase-space due to the transport operator, and can be considered a variant of velocity averaging (see e.g. \cite{GolseEtAl1985,PerthameSouganidis1998,GolseEtAl1988,JabinVega2004}). 

The famous experiments \cite{MalmbergWharton68}, showed that weakly collisional plasmas near equilibrium can display nonlinear oscillations known as plasma echoes.
These oscillations are caused by nonlinear effects exciting modes which are un-mixing in phase space, causing a transient growth of the electric field, which can then, in turn, excite further oscillations and create  a cascade. 
This transient growth is related to the Orr mechanism in fluid mechanics \cite{Orr07}; see \cite{Bedrossian16} and the references therein for more discussion on the relevance of the Orr mechanism to Landau damping. See also \cite{YuDriscollONeil,YuDriscoll02,Vanneste02,VMW98,BM13,BVW16} for discussions regarding the relevance of nonlinear echoes to fluid mechanics.
 Mouhot and Villani \cite{MouhotVillani11} isolated the plasma echo `resonances' as one of the primary potential barriers to proving that the linearized theory extends to the nonlinear Vlasov equations on $\Torus^n \times \Real^n$, and the authors could only control the plasma echoes in a sufficiently regular Gevrey class (see also \cite{BMM13}). 
It was later proved by the author in \cite{Bedrossian16} that on $\Torus \times \Real$, there are indeed settings in which nonlinear plasma echoes can dominate the dynamics if one only assumes the data to be small in Sobolev spaces, justifying the need to work in such high regularity spaces in general. 

The intuition that a small amount of collisions should play a role in suppressing nonlinear effects is classical in plasma physics \cite{SuOberman1968,ONeil1968,Stix}.
We are specifically interested in estimating the \emph{minimal} collisions necessary to suppress the plasma echoes on $\Torus^n \times \Real^n$.  
This question is analogous to a well-known set of problems in hydrodynamic stability: that of estimating the ``transition threshold'' of an equilibrium, such as a shear flow (see e.g. \cite{TTRD93,BaggettEtAl,ReddySchmidEtAl98,BGM15I,BGM15II,BGM15III,BVW16} and the references therein).
The goal there is to estimate the basin of nonlinear stability of a shear flow (in a certain sense) as a scaling law in high Reynolds numbers.  
In \cite{BGM15I,BGM15II,BGM15III}, it was shown that the enhancement of the viscous dissipation due to the mixing driven by the shear is crucial for understanding this problem. 
This mixing-enhanced dissipation effect was derived first by Kelvin \cite{Kelvin87}, and is sometimes called the `shear diffuse mechanism' or `relaxation enhancement' in the fluid mechanics community. It is caused by the interaction of the transport term and dissipation: conservative transport transfers information to high frequencies where the dissipation is increasingly dominant. 
The effect has been studied by many authors in the mathematics community  \cite{CKRZ08,Zlatos2010,BeckWayne11,VukadinovicEtAl2015,BCZ15} and in the fluid mechanics community \cite{RhinesYoung83,DubrulleNazarenko94,LatiniBernoff01}. 
That a similar effect is predicted to happen in plasmas due to the second-order smoothing of Coulomb collisions is classical in the physics community \cite{LenardBernstein1958,SuOberman1968,ONeil1968}. 
Indeed, the time-scale for collisional relaxation of $x$-dependent modes is $O(\abs{k}^{-2/3}\nu^{-1/3})$ (where $k$ is spatial frequency), at least for the linearized problem (see Lemma \ref{lem:propS} below). 
This also suggests a hypoelliptic smoothing effect in $x$, indeed, such hypoellipticity has been observed in the context of collisional kinetic theory for a variety of collision operators \cite{AlexandreEtAl2010, Chen2009320, ChenDesLing2009}. 
Both the enhanced decay and the hypoellipticity play important roles in our work. 

The PDE \eqref{def:VPE} conserves mass, momentum, and energy: 
\begin{align}
\begin{array}{l} %\dss
\dss \frac{d}{dt}\int \bar{\rho}(t,x) \dd x = 0, \\
\dss \frac{d}{dt} \int \bar{\rho} u(t,x) \dd x = 0, \\
\dss \frac{d}{dt}\left(\frac{1}{2}\int \int F(t,x,v) \abs{v}^2 \dd v \dd x + \frac{1}{2}\norm{E(t)}_{L^2}^2\right) := \frac{d}{dt}\cE = 0. 
\end{array}
 \end{align}
One also has an analogue of Boltzmann's $H$-theorem (which plays a role in our work in its linearized guise as a spectral gap; see \S\ref{sec:Zeromode}).  
By conservation of energy, when the electric field Landau damps, the kinetic energy must increase. 
Hence, we are interested in understanding more precisely how the energy in the electric field is converted into heat and how the plasma returns to global thermodynamic equilibrium.
This partially motivates the considerable additional effort in the proof below to use the nonlinear collision operator in \eqref{def:VPE} -- as opposed to the linear $\nu(\Delta_v F + \grad_v \cdot( Fv))$ Fokker-Planck commonly used in studies on weak collisions \cite{LenardBernstein1958,SuOberman1968,ONeil1968,Johnston1971,NgBhattacharjee1999,NgBhattacharjee2006,ShortSimon2002} and \cite{Tristani2016}. 
The other motivation is to rule out that collision-driven nonlinear instabilities dominate (as might arise for example, if the plasma were close to a `hydrodynamic' regime for too long; see e.g. \cite{DesvillettesVillani2005} for a related discussion in the context of the Boltzmann equation). 

For all $\nu > 0$, the eventual state as $t \rightarrow \infty$ is a global Maxwellian with final temperature $T_\infty$:  
\begin{align}
\mu(v) = \frac{1}{(2\pi T_\infty)^{n/2}}e^{-\abs{v}^2/2T_\infty}, \label{def:mu}
\end{align}
where $T_\infty$ is determined by the conservation laws via the energy per unit volume: 
\begin{align}
n T_\infty & = \frac{2}{(2\pi)^n}\cE. \label{def:Tinft}
\end{align}
We have not developed the techniques to be uniform in $T_\infty$ and hence we will need $T_\infty = O(1)$; it is technically simplest to assume $T_\infty = 1$ for convenience, though any fixed $T_\infty > 0$ could be treated.
Denote the following projections for an arbitrary function $g$, 
\begin{align*}
g_0(t,v) & := \frac{1}{(2\pi)^n} \int_{\Real^n} g(t,x,v) \dd x \\
g_{\neq}(t,x,v) & := g(t,x,v) - g_0(t,v),
\end{align*}
and the weighted Sobolev norms (see \S\ref{sec:Notation} for conventions):
\begin{align*}
\norm{f}_{H^j_\ell} = \norm{\brak{v}^{\ell} \brak{\grad_{x,v}}^j f}_{L^2}. 
\end{align*}
The theorem we prove is the following 
\begin{theorem} \label{thm:main}
Suppose $F_{in}$ satisfies $\int_{\Torus^n \times \Real^n} F_{in} \dd x \dd v = (2\pi)^n$, $\int_{\Real^n} vF_{in} \dd x = 0$, and $\cE = \frac{n}{2}(2\pi)^n$. 
Write $F_{in} = \mu + h_{in}$ (with $\mu$ defined as \eqref{def:mu} with $T_\infty = 1$) and consider \eqref{def:VPE} 
with initial data $F_{in}$. 
Then, $\exists\, \sigma_0 = \sigma_0(n)$ such that for all $\sigma > \sigma_0$ chosen so that $\sigma + 1/2 \in \Naturals$ and integers $m',m$ with $m > n/2+3$ and $m' > m +\sigma+1/2$, there exists constants $\nu_0 = \nu_0(\sigma,n,m,m')$ and $c_0 = c_0(\sigma,n,m,m') > 0$ such that if $\nu \in (0,\nu_0)$ and
\begin{align}
\norm{h_{in}}_{H^{\sigma+1/2}_{m'}} = \eps < c_0 \nu^{1/3}, \label{ineq:threshold}
\end{align}
then for some universal $\delta > 0$, and $\beta = \sigma-6$, there holds the following with implicit constants taken independent of $\nu, t$, and $\eps$: 
\begin{subequations} \label{ineq:finalests}
\begin{align}
\norm{F_0(t) - \mu}_{H^{\beta}_{m}}& \lesssim \eps e^{-\nu t}, \\ 
\norm{F_{\neq}(t)}_{L^2_{m}} & \lesssim \eps e^{-\delta \nu^{1/3}t}, \\ 
\abs{\hat{\bar{\rho}}(t,k)} + \abs{\hat{u}(t,k)} + \abs{\widehat{\bar{T}}(t,k)}  & \lesssim \eps e^{-\delta \nu^{1/3}t}\frac{1}{\jap{k,k t}^{\beta}} \quad\quad \forall\, k \neq 0, \\ 
\abs{u_0(t)} + \abs{\bar{T}_0(t) - 1} & \lesssim \eps e^{-2\delta \nu^{1/3}t}\frac{1}{\jap{t}^{2\beta}}. 
\end{align} 
\end{subequations}
\end{theorem} 
\begin{remark} 
See \cite{BMM16} for results on $\Real^3 \times \Real^3$, which show that for models with Debye shielding, dispersive effects, rather than collisions, can suppress nonlinear effects in Sobolev regularity. The dispersion effect may be too weak in $n \leq 2$ (note that the case $n = 1$ can be physically relevant for motions along externally imposed magnetic field lines).  
\end{remark} 
\begin{remark} 
Significantly more precise estimates are available; see \S\ref{sec:Energy} below. 
\end{remark} 
\begin{remark} 
In general, we write $E = -\grad_x W \ast (\bar{\rho}-1)$. For Coulomb electrostatic interactions we have $\widehat{W}(k) = \abs{k}^{-2}$. 
Our proof applies as is to any $W$ satisfying $0 \leq \widehat{W}(k) \lesssim \abs{k}^{-1}$, in particular, the regularizing effects of the collisions can be used to deal with more singular interactions than are treated in \cite{MouhotVillani11,BMM13} (note that more singular interactions can arise in certain limits; see e.g. \cite{Bardos2012,HanKwanRousset15} and the references therein). 
We further remark that if one supplements the hypotheses of Theorem \ref{thm:main} with a Penrose stability criterion such as that used in \cite{MouhotVillani11,BMM13,BMM16} one can extend our theorem to cover some non-repulsive kernels as well. 
\end{remark} 

Theorem \ref{thm:main} simultaneously describes several things: the suppression of echoes via collisions, the enhancement of collisional relaxation via mixing on time-scales like $\nu^{-1/3}$, the Landau damping of the hydrodynamic quantities, the associated heating of the plasma, and explicit control on the return to global thermodynamic equilibrium.  

The $O(\nu^{1/3})$ threshold \eqref{ineq:threshold} is very natural when one considers the relevant time-scales involved.
Using the methods of \cite{BMM13}, it is relatively straightforward to prove that the linearized Vlasov equations are an accurate approximation to the nonlinear Vlasov equations on time-scales like $O(\eps^{-1})$ if $\norm{h_{in}}_{H^\sigma_M} = \eps$ for any $\sigma$ sufficiently large. 
In \cite{Bedrossian16}, it was shown that in certain settings at least, this estimate is sharp up to logarithmic corrections for Sobolev data (that is, nonlinear effects can dominate after time-scales longer than $\eps^{-1}$). 
On the other hand, the time-scale at which collisions begin to dominate is predicted to be $\nu^{-1/3}$.
Hence, \eqref{ineq:threshold} is precisely that the collisional time-scale is shorter than the nonlinear time-scale. 
Accordingly, we conjecture that the $\nu^{1/3}$ threshold is sharp (up to logarithmic corrections) and we remark that the proof is simpler if one takes significantly smaller data in \eqref{ineq:threshold} (e.g. $\eps \lesssim \nu^{1/2}$ or $\eps \lesssim \nu$).
Note that using a Boltzmann collision operator will change the threshold and the collisional time-scale (the Boltzmann operator without cutoff behaves like a fractional dissipation operator \cite{AlexandreEtAl2010}) -- for plasmas, this means making incorrect predictions. 

\begin{remark}[Landau collision operator] \label{rmk:CollOps}
The most well-known collision operator for charged particle interactions is usually referred to as the Landau collision operator \cite{GoldstonRutherford95}; see the mathematical works \cite{Guo02,Guo12,ChenDesLing2009} and the references therein.
Like the Landau operator, the nonlinear Fokker-Planck is a second-order elliptic operator, it has the same conservation laws, and it satisfies an analogue of the $H$-theorem.  
We feel that the nonlinear Fokker-Planck is hence the logical place to start the study of weakly collisional limits in kinetic theory. 
However, it would be very interesting to study the Landau operator, e.g. taking the weakly collisional limit of \cite{Guo02,Guo12}.
Weakly collisional limits of non-cutoff Boltzmann equations would also be interesting (see \cite{AlexandreEtAl2010} for hypoelliptic smoothing effects). 
\end{remark}

\subsection{Notation and conventions} \label{sec:Notation}
We denote $\Naturals = \set{0,1,2,\dots}$ (including zero) and
$\Integer_\ast = \Integer \setminus \set{0}$.  For $\xi \in \Complex$
we use $\bar{\xi}$ to denote the complex conjugate.  
For a vectors $x = (x_1,x_2,...x_n)$ we use $\abs{x}$  to denote the $\ell^1$ norm,  
We denote $\jap{v} = \left( 1 + \abs{v}^2 \right)^{1/2}$  
and furthermore use the shorthand
\begin{align*}
\abs{k,\ell} = \abs{(k,\ell)}, \quad\quad \jap{k,\ell} = \jap{(k,\ell)}.  
\end{align*}
We will use similar notation for the velocity weighted $L^2$ inner product: 
\begin{align*}
\brak{f,g}_{L^2_q} & := \int_{\Torus^n \times \Real^n} \brak{v}^q \overline{f} g \, \dd x \dd v \\
\norm{f}_{L^2_q} & := \sqrt{\brak{f,f}_{L^2_q}} \\ 
\norm{f}_2 & := \norm{f}_{L^2_0}.  
\end{align*}
For $f(x,v)$ we define the Fourier transform $\hat{f}(k,\eta)$, where $(k,\eta) \in \Integer^n \times \Real^n$, and the inverse Fourier transform via
\begin{align*} 
\hat{f}(k,\eta)  = \frac{1}{(2\pi)^{n}}\int_{\Torus^n \times \Real^n} e^{-i x k - iy\eta} f(x,v) \dd x \dd v, \quad f(x,v) = \frac{1}{(2\pi)^{n}}\sum_{k \in \Integer^n} \int_{\Real^n} e^{i x k + iy\eta} \hat{f}(k,\eta) \dd \eta. 
\end{align*} 
For any locally bounded function $m(t,k,\eta)$ we denote the Fourier multiplier: 
\begin{align*}
mf = m(t,\partial_x,\partial_v)f = m(t,\grad)f := \left(m(t,k,\eta)\hat{f}(t,k,\eta)\right)^{\vee}.
\end{align*}
Weighted Sobolev norms are given as $\norm{f}^2_{H^\sigma_q} = \norm{\jap{\grad}^\sigma f}_{L^2_q}$.
To deal with moments, we will often abuse notation and write 
\begin{align*}
\norm{\jap{v}^q f}_{H^\sigma}^2 = \int_{\Torus^n \times \Real^n} \abs{\brak{\grad}^{\sigma}(\brak{v}^q f)}^2 \dd x \dd v \approx_{\sigma,q} \norm{f}_{H^\sigma_q}^2 . 
\end{align*}
We use the notation $f \lesssim g$ when there exists a constant
$C > 0$ such that $f \leq Cg$ (we analogously define $f \gtrsim g$).  Similarly, we use
the notation $f \approx g$ when there exists $C > 0$ such that
$C^{-1}g \leq f \leq Cg$.  We sometimes use the notation
$f \lesssim_{\alpha} g$ if we want to emphasize that the implicit
constant depends on some parameter $\alpha$.

\section{Outline of the proof}

\subsection{Reformulation of the problem}
In this section we reformulate the problem in ways which reveal some of the structure and make the problem more amenable
to Fourier multiplier methods such as those in \cite{BGM15III,BVW16}. 

First, we write the following
\begin{subequations} \label{def:M1M2T}
\begin{align}
\bar{T}(t,x) & = 1 + T(t,x), \\ 
M_1(t,x) & = \int h v \ddv, \\
M_2(t,x) & = \int h \abs{v}^2 \ddv.  
\end{align}
\end{subequations}
By conservation of mass, momentum, and energy, together with the definition of $\cE$, 
\begin{subequations} 
\begin{align}
\int \rho \dd x & = 0,  \\ 
\int M_1 \dd x & = 0, \label{eq:M1ctrl} \\ 
\int M_2 \dd x & = -\norm{E(t)}_{L^2}^2. \label{eq:M2ctrl}
\end{align}   
\end{subequations}
From \eqref{def:VPE}, 
\begin{subequations} \label{def:uT}
\begin{align}
(1 + \rho) u & = M_1 \\ 
(1+\rho) \abs{u}^2 + n (1+\rho)(1 + T) & = \int \mu \abs{v}^2 \ddv + M_2. 
\end{align}
\end{subequations} 
Note that by $T_\infty = 1$ and \eqref{def:Tinft}, this becomes 
\begin{align*}
n(1+\rho) T & = M_2 - (1+\rho) \abs{u}^2. 
\end{align*}
For future notational convenience, define
\begin{align}
M_T := (1+\rho)T. \label{def:MT}
\end{align}
Since the density perturbation $\rho$ will remain small, when necessary, we write $u$ and $T$ in terms of a geometric series: 
\begin{subequations} \label{eq:uTexpan}
\begin{align} 
u & = M_1 \sum_{j=0}^\infty (-1)^j \rho^j \\ 
nT & = \left(M_2 - (1+\rho) \abs{u}^2\right) \sum_{j=0}^\infty (-1)^j \rho^j. 
\end{align}
\end{subequations}
Next, we expand the collision operator. From \eqref{def:VPE}, we write
\begin{align}
\partial_t h + v\cdot \grad_x h + E\cdot \grad_v \mu + E \cdot \grad_v h = \nu Lh + \nu \cC_\mu + \nu \cC_{h}, \label{def:VPEreform}
\end{align}
with 
\begin{subequations} \label{def:cC}
\begin{align}
Lh & = \Delta_v h + \grad_v \cdot (h v) \\ 
\cC_\mu & = M_T \Delta_v \mu - M_1 \cdot \grad_v\mu \label{def:cCmu} \\
\cC_h & = \rho \left(\Delta_v h + \grad_v \cdot (h v) \right) + M_T \Delta_v h - M_1 \cdot \grad_vh. \label{def:cCh}
\end{align}
\end{subequations}
The two leading terms $Lh$ and $\cC_h$ comprise the effects of collisions between $h$ with the background $\mu$; the $Lh$ term will be dominant for long times, whereas the $\cC_\mu$ term will rapidly decay due to the Landau damping effects. The last term, $\cC_h$, comprises the fully nonlinear collision effects. 

Next, we expand \eqref{def:VPE} on the Fourier side to deal with the phase-mixing by adapting ideas developed in \cite{BMM13,BGM15III,BVW16},    
\begin{align*}
\partial_t \hat{h} - k \cdot \grad_{\eta} \hat{h} + \nu \eta \cdot \grad_\eta \hat{h} + \widehat{E}(t,k) \cdot i \eta \widehat{\mu}(\eta) = -\nu \eta^2 \hat{h} + \cN(t,k,\eta),
\end{align*}
where $\cN = -\widehat{E \cdot \grad_v h} + \nu \widehat{\cC_\mu} + \nu \widehat{\cC_h}$. 
We will use the method of characteristics to deal with the transport in frequency, hence, we write 
\begin{subequations}
\begin{align}
\bar{\eta}(t;k,\eta) & = e^{\nu t} \eta - k\left(\frac{e^{\nu t} - 1}{\nu}\right) \\
\hat{h}(t,k,\eta) & = \hat{f}\left(t,k,e^{-\nu t} \eta + k\left(\frac{1 - e^{-\nu t}}{\nu}\right) \right), \label{def:hfrelat}
\end{align}
\end{subequations} 
which implies 
\begin{align}
\partial_t \hat{f} + \widehat{E}(t,k)\cdot i \bar{\eta}(t;k,\eta) \widehat{\mu}(\bar{\eta}(t;k,\eta)) = -\nu \bar{\eta}^2 \hat{f} + \mathcal{N},  \label{eq:fhat}
\end{align}
where 
\begin{align*}
\mathcal{N}(t,k,\bar{\eta}(t;k,\eta)) & = 
\sum_{\ell \in \Integers^n_\ast} \widehat{E}(t,\ell) \cdot i\left[\eta - k\left(\frac{1-e^{-\nu t}}{\nu}\right) \right] e^{\nu t} \widehat{f}\left(t,k-\ell,\eta - \ell\left(\frac{1-e^{-\nu t}}{\nu} \right) \right) \\
& \quad + \nu \widehat{\cC_{\mu}}(t,k,\bar{\eta}(t;k,\eta)) + \nu \widehat{\cC_h}(t,k,\bar{\eta}(t;k,\eta)).  
\end{align*}
We will use \eqref{eq:fhat} to derive suitable energy estimates on the distribution function using both some new ideas as well as a combination of techniques from \cite{BMM13,BVW16,BMM16}. 
Note the relationship between the moments and the distribution function: 
\begin{subequations} \label{def:Hydromoments}
\begin{align}
\hat{\rho}(t,k) & = \hat{f}\left(t,k,k\left(\frac{1 - e^{-\nu t}}{\nu}\right) \right) \\ 
\widehat{M_1}(t,k) & = e^{-\nu t}\grad_\eta \hat{f}\left(t,k,k\left(\frac{1 - e^{-\nu t}}{\nu}\right) \right) \\ 
\widehat{M_2}(t,k) & = e^{-2\nu t}\Delta_\eta \hat{f}\left(t,k,k\left(\frac{1 - e^{-\nu t}}{\nu}\right) \right). 
\end{align}
\end{subequations}
Note that the analogue of the Orr critical time \cite{Orr07} is no longer $\eta = kt$, but instead now $\eta=\eta_{CT}(t,k) := k\left(\frac{1 - e^{-\nu t}}{\nu}\right)$ (see e.g. \cite{Bedrossian16} for more discussion). 

One of the easiest ways to treat the linearized (collisionless) Vlasov equations is to derive a Volterra equation for the density; see e.g. \cite{MouhotVillani11,BMM13,BMM16}. 
In order to handle collisions in a manner consistent with these treatments, we first apply Duhamel's principle to \eqref{eq:fhat}. 
Hence, denoting 
\begin{align*}
S(t,\tau;k,\eta) = \exp\left[-\nu\int_\tau^t \left[e^{\nu s} \eta - k\left(\frac{e^{\nu s} - 1}{\nu}\right)\right]^2 \dd s\right], 
\end{align*}
we have 
\begin{align}
\hat{f}(t,k,\eta) & = S(t,0;k,\eta) \widehat{f_{in}}(k,\eta) - \int_0^t S(t,\tau;k,\eta)\widehat{E}(\tau,k) \cdot i\bar{\eta}(\tau;k,\eta) \widehat{\mu}(\bar{\eta}(\tau;k,\eta)) \dd \tau \nonumber  \\ & \quad + \int_0^t S(t,\tau;k,\eta) \mathcal{N}(\tau,k,\bar{\eta}(\tau;k,\eta)) \dd \tau.  \label{eq:hatfDuhamel}  
\end{align}
We next restrict \eqref{eq:hatfDuhamel} to $\eta = \eta_{CT}(t,k) = k\left(\frac{1 - e^{-\nu t}}{\nu}\right)$ in order to get an expression for the density. 
Hence, if we define 
\begin{align}
S(t-\tau,k) & = S\left(t,\tau;k,k\frac{1-e^{-\nu t}}{\nu}\right) = \exp\left[-\nu k^2\int_\tau^t \left[\frac{1 - e^{\nu(s-t)}}{\nu}\right]^{2} \dd s\right] \nonumber \\
%& = \exp\left[-\nu^{-1} k^2\int_\tau^t 1 - 2e^{\nu(s-t)} + e^{2\nu(s-t)} ds\right] \\
& = \exp\left[-\nu^{-1} k^2 \left((t-\tau) + 2\frac{e^{\nu(\tau-t)}-1}{\nu} - \frac{e^{2\nu(\tau-t)}-1}{2\nu}\right)\right], \label{def:S}
\end{align}
we can write the density evolution as a Volterra equation (with nonlinear contributions): 
\begin{align}
\hat{\rho}(t,k) & = S(t,k) \widehat{f_{in}}(k,\eta_{CT}(t,k)) - \int_0^t S(t-\tau,k)\widehat{\rho}(\tau,k) \widehat{W}(k) \abs{k}^2\left(\frac{1 - e^{-\nu(t-\tau)}}{\nu}\right) \widehat{\mu}\left(k \left(\frac{1 - e^{-\nu(t-\tau)}}{\nu}\right)\right) \dd \tau \nonumber \\
& \quad + \int_0^t S(t-\tau,k) \mathcal{N}(\tau,k,\bar{\eta}(\tau;k,\eta_{CT}(t,k)))  \dd \tau. \label{eq:rhoDef}
\end{align}
A similar calculation is done to deduce Volterra-like equations for $M_1$ and $M_2$, but they are more technical; see \S\ref{sec:highMoments} for the details.

\subsection{Energy estimates} \label{sec:Energy}
As has been done in several previous works on Landau damping and mixing in fluid mechanics \cite{BMM13,BM13,BMV14,Young14,BMM16,BGM15I,BGM15II,BGM15III}, we apply a bootstrap argument. 
Bootstrap arguments typically require a local well-posedness and propagation of regularity result. 
The following theorem is sufficient for our purposes and can be proved with standard techniques; the proof is omitted for brevity. 
Note that the requirement on $m$ is so make the Fourier restriction necessary to define $\widehat{T}$ pointwise in frequency.  
%The propagation of regularity can be proved by a variant of arguments in, e.g. \cite{LevermoreOliver97}. 
%inequality
%\begin{align}
%  \norm{B(t,\grad_x,\grad_x\frac{1-e^{-\nu t}}{\nu})\rho(t)}_2 \lesssim \sum_{\abs{\alpha} \leq M} \norm{v^\alpha B(t,\grad_x,\grad_v)h(t)}_2,  
%\end{align} 
%for all Fourier multipliers $B$ and all integers $m > n/2$ (and similar restriction inequalities applied for the higher moments).  
\begin{lemma}[Local well-posedness] \label{lem:loctheory} 
Let $F_{in} \in H^{\alpha}_m$ for $\alpha > n+2$ with $m$ an integer with $m > n/2+2$ such that $\bar{T}(0,x) > \zeta$ and $\bar{\rho}(0,x) > \zeta$ for all $x$ for some fixed $\zeta \in (0,1)$. 
Then, there exists a time $t^\ast = t^\ast(m,\alpha,n,\zeta)$ and a unique solution $F(t) \in C([0,\tau];H^{\alpha}_m)$ to \eqref{def:VPE} for all $\tau < t^\ast$ 
additionally satisfying $\bar{T}(t,x) > 0$ and $\bar{\rho}(t,x) > 0$ for all $t < t^\ast$.
Moreover, if for some $t \leq t^\ast$ and $\alpha \geq \alpha' > n+2$ there holds 
$\limsup_{\tau \nearrow t} \norm{F(\tau)}_{H^{\alpha'}_m} < \infty$, $\sup_{\tau < t} \norm{\bar{T}^{-1}(\tau)}_{L^\infty} < \infty$, and $\sup_{\tau < t} \norm{\bar{\rho}^{-1}(\tau)}_{L^\infty} < \infty$, 
then we may take $t < t^\ast$ (that is, we may extend the unique solution further in time). 
\end{lemma}

Next, we are interested in obtaining uniform-in-time higher regularity estimates on the solution to \eqref{eq:fhat}, as well as the mixing-enhanced collisional relaxation.  
We will make use of Fourier multiplier norm ideas from \cite{BGM15III,BVW16} (with kinetic theory-specific adaptations) as well as a variation of the bootstrap arguments used previously in studies of nonlinear Landau damping \cite{BMM13,BMM16,Bedrossian16}. 
The main technical novelties in the proof arise due to the nonlinearity in the collision operator, in particular, the challenge in dealing with large potential regularity losses ($x$-dependence in the coefficients creates a kind of `cross-diffusion' between Fourier modes, which can be problematic for the specific kinds of delicate energy estimates we need to make in order to deduce Landau damping and enhanced collisional decay) and in the controlling of higher velocity moments of the solutions (in the collisionless works, controlling the moments is more or less a triviality, but this difficulty plays a centrol role here and dictates many of our choices below). 
These issues are discussed in more precise detail in Remarks \ref{rmk:12deriv}--\ref{rmk:collisvsnoncollis} below and in the main body of the paper.      
 
In order to obtain mixing-enhanced collisional decay, we adapt an idea from \cite{BGM15III}. 
Specifically, the following multiplier is variant of one used therein, but here designed with the $\nu$-dependent critical times accounted for: 
\begin{align*}
\partial_t M(t,k,\eta) & = -\frac{\nu^{1/3}}{1 + \nu^{2/3} e^{2\nu t}\abs{\eta - k \frac{1-e^{-\nu t}}{\nu}}^2} \frac{e^{2\nu t} \abs{\nu \eta - k}^2}{\brak{e^{\nu t}(\nu \eta - k)}^2}  M(t,k,\eta) \\
M(0,k,\eta) & = 1. 
\end{align*}
The relevant properties of this multiplier are outlined in Appendix \S\ref{sec:propDissM}. 
The norms we apply are then built from the multiplier 
\begin{subequations} \label{def:Asc} 
\begin{align}
A_{s,c}(t,k,\eta) & = e^{c \nu^{1/3} t} \jap{k,\eta}^s M(t,k,\eta), & \textup{for } k\neq 0 \\
A_{s,c}(t,k,\eta) & = \jap{\eta}^s, &  \textup{for } k = 0. 
\end{align}
\end{subequations}
For convenience, denote $\partial_v^t := e^{\nu t}\left(\grad_v - \grad_x \frac{1-e^{-\nu t}}{\nu}\right)$ and note, 
\begin{align*}
\widehat{\partial_v^t f} = i\left(\eta e^{\nu t} - k\frac{e^{\nu t} - 1}{\nu}\right)\hat{f}(t,k,\eta) = i\bar{\eta} \hat{f}(t,k,\eta). 
\end{align*}
Define the following norms on a distribution function $f(t,x,v)$ or $f(t,v)$: for constants $K_j$ determined by the proof chosen such that $1 \leq K_j < K_{j+1} <...$,
\begin{subequations} \label{def:FcD}
\begin{align}
\norm{f(t)}_{F^{s,c}_m} & := \left(\sum_{\alpha \in \Naturals^n:\abs{\alpha} \leq m} \frac{e^{-2\abs{\alpha}\nu t}}{K_{\abs{\alpha}}}\norm{A_{s,c}(t,\grad) (v^\alpha f(t))}_{L^2}^2\right)^{1/2} \\ 
\norm{f(t)}_{\cD^{s,c}_m} & := \left(\sum_{\alpha \in \Naturals^n:\abs{\alpha} \leq m} \frac{e^{-2\abs{\alpha}\nu t}}{K_{\abs{\alpha}}}\norm{\partial_v^t A_{s,c}(t,\grad) (v^\alpha f(t))}_{L^2}^2\right)^{1/2}. 
\end{align}
\end{subequations}
Note the growth/loss of decay for higher order moments, which is natural in view of \eqref{def:hfrelat}. 
For $x$-dependent modes, this loss is very slight compared to the overall $e^{-c\nu^{1/3}t}$ decay. 
Notice also that, by Lemma \ref{lem:propM}, $M \approx 1$, and hence the presence of $M$ does not change the norm defined by the multiplier $A_{s,c}$:
\begin{align*}
\norm{A_{s,c} f}_{L^2} \approx e^{c \nu^{1/3}t}\norm{\brak{\grad}^{s} f}_{L^2}.
\end{align*}
In this way, this particular use of the multiplier methods from \cite{BGM15III,BVW16} is reminiscent of a Fourier-side variant of the ghost energy method \cite{Alinhac01}. 
Finally, observe also that the ordering of $A_{s,c}$ and $v^\alpha$ will normally be irrelevant:
\begin{align}
\norm{f}_{F^{s,c}_m} & \approx \left(\sum_{j = 0}^m e^{-2j\nu t} \norm{\brak{v}^{j} A_{s,c}(t,\grad) f(t)}_{L^2}^2 \right)^{1/2} \approx \sum_{j=0}^m e^{-j\nu t} e^{c\nu^{1/3}t} \norm{f(t)}_{H^s_j}. \label{ineq:FHsj}
\end{align}
This follows from Liebniz's rule on the Fourier side and noting that the commutators $[v^\alpha,A_s]f$ involve lower order moments of $f$, which have slower growth than moments of order $\alpha$ (using also \eqref{ineq:detaM} to control the derivatives of $M$).  
Similarly, for $m \geq 1$, 
\begin{align}
\left(\sum_{j=0}^m e^{-2 j \nu t} e^{2c \nu^{1/3}t}\norm{\partial_v^t f}^2_{H^{s}_j}\right)^{1/2} \lesssim \norm{f}_{\cD^{s,c}_m} + \norm{f}_{F^{s,c}_{m-1}}. \label{ineq:cDHsj}
\end{align}
For the hydrodynamic quantities, the norms are defined as follows for a function $Q=Q(t,x)$: 
\begin{align}
\norm{Q}_{\cM^{s,c}} := \left(\sum_{k \in \Integers^d} A^2_{s,c}(t,k,kt) \abs{\hat{Q}(t,k)}^2\right)^{1/2}. 
\end{align}
To measure $h$ with high moments and low derivatives, we fix an arbitrary $\theta \in (1,2)$ and define: 
\begin{align}
\norm{h(t)}_{\cH} = \left(\sum_{\abs{\alpha} + \abs{\gamma} \leq n+2} \brak{t}^{-2\theta\abs{\gamma}}\norm{D_x^\alpha D_v^\gamma h(t)}^2_{L^2_{m'-\abs{\alpha}}} \right)^{1/2}.
\end{align}
It is clear that the $\theta\abs{\gamma}$ growth is connected to the phase-mixing, however, the relation between $x$ derivatives and velocity localization is a more subtle detail connected to the nonlinear collision operator. 
 
Let $\beta < \sigma$, $0 < \delta_1 < \delta \ll 1$, $s < \sigma$, $p \geq 6$, and $K_{Hj},K_H,K_L, K_{MHi}, K_{MLi}$ be constants determined by the proof depending only on $\sigma$, $n$, $m$ and let $T^\ast > 0$ be the largest time such that all of the following estimates hold for $I = [0,T^\ast]$ (Lemma \ref{lem:loctheory} implies $T^\ast > 0$):
\begin{itemize} 
\item[(A)] high norm distribution function estimates: for all integers $0 \leq j \leq 3$ and $t \in I$, 
\begin{subequations} \label{ineq:bootf}
\begin{align}
\brak{t}^{-p}\norm{f(t)}_{F^{\sigma+1/2-j,\delta_1}_{m+1+j}}^2  + \nu \int_0^{t}\brak{\tau}^{-p}\norm{f(\tau)}_{\cD^{\sigma+1/2-j,\delta_1}_{m+1+j}}^2 \dd \tau  & \leq 4 K_{Hj}^2 \eps^2 \label{boot:hineq} \\ 
\norm{h_0(t)}_{H^{s}_m} & \leq 4 K_{H} \eps e^{-\nu t} \label{boot:therm}; 
\end{align} 
\end{subequations}
\item[(B)] lower norm distribution function estimates: for all integers $4 \leq j < \sigma+1/2-n-2)$ and $t \in I$, 
\begin{subequations}
\begin{align}
\norm{f(t)}_{F^{\sigma+1/2-j,\delta_1}_{m+1+j}}^2  + \nu \int_0^{t}\norm{f(\tau)}_{\cD^{\sigma+1/2-j,\delta_1}_{m+1+j}}^2 \dd \tau  & \leq 4 K_{Lj}^2 \eps^2 \label{boot:lowneq}  \\ 
e^{-K\nu t} \norm{h(t)}_{\cH}^2 + \nu \int_0^t e^{-K\nu \tau} \norm{\grad_v h(\tau)}_{\cH}^2 \dd \tau & \leq 4 K_{L}^2 \eps^2; \label{boot:lowhimoment}
\end{align} 
\end{subequations} 
\item[(C)] estimates on the moments (recall the definitions \eqref{def:M1M2T} and \eqref{def:uT}), 
\begin{subequations} \label{ineq:boothydro} 
\begin{align} 
\norm{\abs{\grad_x}^{1/2} \rho}_{L^2_t(I;\cM^{\sigma,\delta})} & \leq 4K_{MH0}\eps \label{ineq:bootArho} \\  
\norm{\abs{\grad_x}^{1/2} M_1}_{L^2_t(I;\cM^{\sigma,\delta})} & \leq 4K_{MH1}\eps \label{ineq:bootAM1} \\ 
\norm{\brak{\grad_x}^{1/2} M_2}_{L^2_t(I;\cM^{\sigma,\delta})} & \leq 4K_{MH2}\eps \label{ineq:bootAM2} \\ 
\norm{\rho}_{L^\infty_t(I;\cM^{\beta,\delta})} & \leq 4K_{ML0}\eps \label{ineq:bootLinftArho} \\ 
\norm{M_1}_{L^\infty_t(I;\cM^{\beta,\delta})} & \leq 4K_{ML1}\eps \label{ineq:bootLinftAM1}\\ 
\norm{M_2}_{L^\infty_t(I;\cM^{\beta,\delta})} & \leq 4K_{ML2}\eps \label{ineq:bootLinftAM2}. 
\end{align}
\end{subequations}
\end{itemize} 
\begin{remark} \label{rmk:bootconst}
How the various constants are set is important to closing the bootstrap. 
First $\delta$ is chosen  depending only on universal constants and $0 < \delta_1 < \delta$ is set arbitrarily.
The constant $p$ is chosen large depending only on universal constants. 
The constant $\beta$ is chosen sufficiently large depending only on $n$ over various steps in the proof and the gap $\sigma-\beta$ is chosen depending only on universal constants (this hence sets $\sigma_0$). 
The constant $s < \sigma$ is set with $\sigma-s$ fixed small by the proof depending only on $m$ and $n$ (by Remark \ref{rmk:HsmInterp}). 
The constant $K$ is fixed large depending only on $m'$ and $n$ (in \S\ref{sec:HiMoment}). 
Then, $K_{HM0}$ is set depending only on the linearized equations together with $\delta,\sigma,n$ (by Lemma \ref{lem:LinLandau}).
Then $K_{ML0}$ is set depending on previously assigned constants, followed by $K_{MHj}$ and $K_{MLj}$. 
Similarly, $K_{L1}$, then $K_{L2}$, and so forth is set. Similarly $K_{H1}$ is set then $K_{H2}$ and so forth.
Finally, $\nu_0$ and $c_0$ are chosen small depending on all of the previously assigned constants.  
\end{remark} 
\begin{remark} \label{rmk:12deriv} 
The gain of $1/2$ derivative in the moment estimates without the time growth seen in \eqref{boot:hineq} is the same ``free''  gains in regularity seen in velocity averaging lemmas \cite{GolseEtAl1985,GolseEtAl1988}. 
This and the hypoelliptic smoothing in $x$ is crucial to making the proof work, as this is what allows to use $\sigma+1/2$ in \eqref{boot:lowneq} and $\sigma$ in \eqref{ineq:bootArho} (note that a larger gap was used in \cite{BMM13,Young14,BMM16,Bedrossian16}, but this is not be possible due to the nonlinear collisions). 
\end{remark}
\begin{remark}
Note that we are choosing $\delta > \delta_1$. This is very useful in the proof and only possible due to the particular structure of \eqref{eq:rhoDef} which does not permit the kinds of nonlinear interactions that would normally make this impossible.  
\end{remark}

\begin{remark}
The above scheme is significantly more detailed than the schemes used in e.g. \cite{BMM13,Young14,BMM16}. 
This is partly due to the drift term $\grad_v \cdot (vF)$ in the linear Fokker-Planck operator, but is more due to the complications from the nonlinear, $x$-dependent coefficients in the Fokker-Planck collision operator in \eqref{def:VPE}. 
The complications involve both dealing with velocity localization and in dealing with potential regularity loss. 
If one were to replace $\bar{\rho} \mapsto 1$ in \eqref{def:VPE}, some aspects of the velocity localization arguments would simplify, however, the $\bar{\rho}$ is more physical \cite{GoldstonRutherford95}. 
\end{remark}
\begin{remark} \label{rmk:collisvsnoncollis}
One of the reasons for \eqref{boot:lowhimoment} is that \eqref{boot:lowneq} and \eqref{boot:hineq} are deduced with methods more suited for the collisionless regime, which does not handle some of the nonlinear structure in the collision operator efficiently. The estimate \eqref{boot:lowhimoment} on the other hand, is not very useful for dealing with collisionless phase-mixing solutions, however, in this estimate it is easier to localize the effects of the nonlinear collision operator in velocity. 
\end{remark} 
\begin{remark} 
For simplicity of notation, we will often write $L^p_t(I;Z) = L^p_t Z$ in what follows. 
\end{remark}

Theorem \ref{thm:main} follows immediately from the following (together with Lemma \ref{lem:loctheory}). 
\begin{proposition} \label{prop:boot}
There exists a constant $c_0 = c_0(n,\sigma,m,m')$ and a constant $\nu_0 = \nu_0(n,\sigma,m,m')$ such that if $\nu < \nu_0$ and $\eps < c_0 \nu^{1/3}$, then we 
have that the estimates in \eqref{ineq:bootf} and \eqref{ineq:boothydro} all hold on $[0,T^\ast]$ with the 4's replaced by 2's.  
\end{proposition}

For $\sigma >\beta > n/2$, by the Sobolev space product rule together with \eqref{eq:uTexpan} and \eqref{ineq:boothydro}, there holds the following. 
\begin{lemma} \label{lem:uTctrls} 
For $\eps$ sufficiently small, under the bootstrap hypotheses, there holds (recall \eqref{def:MT}), 
\begin{subequations} 
\begin{align}
\norm{\brak{\grad_x}^{1/2} u}_{L^2_t(I;\cM^{\sigma,\delta})} & \lesssim K_{MH1}\eps \\ 
\norm{\brak{\grad_x}^{1/2} T}_{L^2_t(I;\cM^{\sigma,\delta})} & \lesssim K_{MH2}\eps \\
\norm{\brak{\grad_x}^{1/2} M_T}_{L^2_t(I;\cM^{\sigma,\delta})} & \lesssim (K_{MH1} + K_{MH2})\eps \label{ineq:MThi} \\
\norm{u}_{L^\infty_t(I;\cM^{\beta,\delta})} & \lesssim K_{ML1}\eps \\ 
\norm{T}_{L^\infty_t(I;\cM^{\beta,\delta})} & \lesssim K_{ML2}\eps \\
\norm{M_T}_{L^\infty_t(I;\cM^{\beta,\delta})} & \lesssim (K_{ML2} + K_{HM1})\eps.
\end{align}
\end{subequations}
\end{lemma} 
\begin{remark} 
In fact, $u_0$ and $T_0$ both decay much faster than $u_{\neq}$ and $T_{\neq}$. 
\end{remark}

We can also transfer the estimate \eqref{boot:lowhimoment} into estimates on $f$. 
\begin{lemma} \label{lem:cHvsFcD}
There holds, 
\begin{subequations} \label{ineq:cHF}
\begin{align}
\norm{f}_{F^{\textup{Ceil}(n/2+2),0}_{m'}} & \lesssim \brak{t}^{(1+\theta)\textup{Ceil}(n/2+2)}\norm{h(t)}_{\cH} \\ 
\norm{f}_{\cD^{\textup{Ceil}(n/2+2),0}_{m'}} & \lesssim \brak{t}^{(1+\theta)\textup{Ceil}(n/2+2)}\left(\norm{\grad_v h(t)}_{\cH} + \norm{h(t)}_{\cH}\right). 
\end{align}
\end{subequations}
\end{lemma} 
\begin{proof} 
By definition
\begin{align*}
\hat{f}(t,k,\eta) & = \hat{h}\left(t,k,e^{\nu t} \eta - k\left(\frac{e^{\nu t} - 1}{\nu}\right) \right), 
\end{align*}
and hence, 
\begin{align*}
\brak{k,\eta}^{\ell} \abs{D^\alpha_\eta \hat{f}(t,k,\eta)} \lesssim \left(\brak{k,e^{\nu t} \eta - k \frac{e^{\nu t}-1}{\nu}}^{\ell} + \brak{k \frac{1-e^{-\nu t}}{\nu}}^{\ell}\right) e^{\abs{\alpha} \nu t} \abs{D_\eta^\alpha \hat{h}\left(t,k,e^{\nu t} \eta - k\left(\frac{e^{\nu t} - 1}{\nu}\right) \right)}. 
\end{align*}
\end{proof}

\begin{remark} 
In the ensuing proof, we will often omit the dependence of implicit constants on the bootstrap constants $K_H$, $K_L$,... where the precise value is not important due to the presence of a small parameter such as $\eps$ or $\nu$. We instead only make this dependence explicit where necessary. 
\end{remark}

\section{Density estimates} \label{sec:density}
In this section we improve \eqref{ineq:bootArho} and \eqref{ineq:bootLinftArho}, which also lays the groundwork for \S\ref{sec:highMoments} below. 

\subsection{Properties of $S$}
We outline a few properties of $S(t,k)$, which essentially plays the role of a diffusion semigroup. 
\begin{lemma}[Properties of $S$] \label{lem:propS}
The following holds for all $t \geq 0$ and $\nu$ sufficiently small (depending only on universal constants): 
\begin{itemize}
\item[(a)] $S(t,k)$ is strictly decreasing and there exists a number $\delta_0> $ such that
\begin{align}
0 < S(t,k) < \exp\left(-\delta_0 \min(\nu k^2 t^3,\nu^{-1}k^2 t)\right). \label{ineq:Sdecay}
\end{align}
\item[(b)] For all $p \in (0,1]$ and all $\delta$ chosen sufficiently small (depending on $p$), there holds,
\begin{align}
e^{\delta \nu^{1/3}t} S^p(t-\tau,k) \lesssim_{p,\delta} e^{\delta \nu^{1/3}\tau}. 
\end{align}
\end{itemize}
\end{lemma}
\begin{remark}
Lemma \ref{lem:propS} not only implies the enhanced decay of hydrodynamic quantities due to the interaction of mixing and collisions, but also a kind of hypoelliptic smoothing effect which gains regularity in $x$. This smoothing will be used at several points in the proofs that follow.  
\end{remark}
\begin{proof}
Consider first part (a).
Recall, from \eqref{def:S}
\begin{align}
S(t,k) & = \exp\left[-\nu^{-1} k^2 \left(t + 2\frac{e^{-\nu t}-1}{\nu} - \frac{e^{-2\nu t}-1}{2\nu}\right)\right], \label{ineq:Sexp}
\end{align}
and notice that the derivative of the time-dependent factor is: 
\begin{align*}
\frac{d}{dt}\left(t + 2\frac{e^{-\nu t}-1}{\nu} - \frac{e^{-2\nu t}-1}{2\nu}\right) = 1 - 2e^{-\nu t} + e^{-2\nu t} =\left(1-e^{-\nu t}\right)^2, 
\end{align*}
which implies that $S$ is strictly decreasing for $t > 0$. 
For $\nu t < 1$, the second half of (a) follows from Taylor expansion of the exponentials in \eqref{ineq:Sexp}.
For $\nu t \geq 2$, we have
\begin{align*}
\nu^{-1} k^2 \left(t + 2\frac{e^{-\nu t}-1}{\nu} - \frac{e^{-2\nu t}-1}{2\nu}\right) \geq \nu^{-1}k^2\left(t - \frac{3}{2\nu}\right), 
\end{align*}
which proves (a) in this case. In the intermediate $1 \leq \nu t \leq 2$ case, the result follows as the exponent is strictly decreasing and both behaviors are comparable in this case.

Next, consider part (b).
In the regime $t-\tau \leq \nu^{-1}$, there holds
\begin{align*}
e^{\delta \nu^{1/3} (t-\tau)} S^p(t-\tau,k) \lesssim e^{\delta \nu^{1/3}(t-\tau) - p\delta_0 k^2 \nu (t-\tau)^3}; 
\end{align*}
if $\nu^{1/3}(t-\tau) \leq 1$ then this quantity is $\approx 1$, whereas if $\nu^{1/3}(t-\tau) \geq 1$, the result follows for $\nu$ small. If $t-\tau \geq \nu^{-1}$ then the follows from \eqref{ineq:Sdecay}. 
\end{proof} 

\subsection{Linearized Landau damping of the density with collisions} \label{sec:LinLD}
The first step in obtaining density estimates is a proper treatment of the linearized Volterra equation in \eqref{eq:rhoDef} for a (currently abstract) source: 
\begin{align}
\hat{\rho}(t,k) = \mathcal{F}(t,k) - \int_0^t S(t-\tau,k)\widehat{\rho}(\tau,k) \hat{W} \abs{k}^2 \left(\frac{1 - e^{-\nu(t-\tau)}}{\nu}\right) \widehat{\mu}\left(k \left(\frac{1 - e^{-\nu(t-\tau)}}{\nu}\right)\right) \dd \tau. \label{def:linVolt}
\end{align}
First, pre-multiply \eqref{def:linVolt} by the exponential factor: 
\begin{align}
e^{\delta \nu^{1/3}t}\hat{\rho}(t,k) & = e^{\delta \nu^{1/3}t} \mathcal{F}(t,k) + \int_0^t K^0(t-\tau,k) e^{\delta \nu^{1/3} \tau} \rho(\tau,k) \dd \tau, \label{eq:linVoltprem}
\end{align}
where the integral kernel is defined via
\begin{align*}
K^0(t,k) := e^{\delta\nu^{1/3}t} S(t,k)  \widehat{W}(k) \abs{k}^2 \left(\frac{1 - e^{-\nu t}}{\nu}\right) \widehat{\mu}\left(k \left(\frac{1 - e^{-\nu t}}{\nu}\right)\right) \mathbf{1}_{t \geq 0}. 
\end{align*}
The Volterra equation \eqref{eq:linVoltprem} is then studied using techniques similar to \cite{MouhotVillani11,BMM13,BMM16} and the references. 
The results are similar to results of \cite{Tristani2016}, however, here we gain the collisional decay factor $e^{-\delta \nu^{1/3}t}$. 
\begin{lemma}[Linearized Landau damping] \label{lem:LinLandau}
For $1 \gg \delta > 0$ and all $\nu$ sufficiently small (depending only on universal constants), we have
\begin{align*}
\norm{\abs{\grad_x}^{1/2} \rho}_{L^2_t(I;\cM^{\sigma,\delta})} & \lesssim_{\sigma} \norm{\abs{\grad_x}^{1/2}  \mathcal{F}}_{L^2_t(I;\cM^{\sigma,\delta})}. 
\end{align*}
\end{lemma} 
\begin{proof}
We write the proof in the style of \cite{MouhotVillani11}.  
Denote $\tilde{\mathcal{F}} = e^{\delta \nu^{1/3}t}\mathcal{F}(t) \mathbf{1}_{t \geq 0}$ and $\Phi$ the solution to the Volterra equation (as in \eqref{eq:linVoltprem}), 
\begin{align*}
\Phi(t,k) = \tilde{\mathcal{F}}(t,k) + \int_0^t K^0(t-\tau,k) \Phi(\tau,k) \dd \tau. 
\end{align*}
For $t \in [0,T^\ast]$, by uniqueness there holds that $\Phi(t,k) = e^{\delta \nu^{1/3}t}\rho(t,k)$. 
Next, we extend $\Phi$ by zero for $t < 0$ and formally take the Fourier transform in time:
\begin{align}
\widehat{\Phi}(\omega,k) = \widehat{\tilde{\mathcal{F}}} + \widehat{K^0}(\omega,k) \widehat{\Phi}(\omega,k). \label{eq:Laplace}
\end{align}
Technically, \eqref{eq:Laplace} is not justified, as we have not verified $\Phi$ has sufficient decay in time. However, the argument we are about to make is an a priori estimate which is easily justified by one of several methods (see e.g. \cite{MouhotVillani11,BMM13}), and hence we omit the details for brevity.

In order to get polynomial decay estimates, take $(\abs{k}\partial_{\omega})^\alpha$ derivatives of \eqref{eq:Laplace} as in \cite{BMM16}, giving 
\begin{align}
(\abs{k} \partial_\omega)^\alpha \widehat{\Phi}(\omega,k) = (\abs{k} \partial_\omega)^\alpha \widehat{\tilde{\mathcal{F}}} + \sum_{0 \leq j \leq \alpha }\begin{pmatrix} \alpha \\ j \end{pmatrix} (\abs{k} \partial_\omega)^{\alpha - j} \widehat{K^0}(\omega,k) (\abs{k} \partial_\omega)^j \widehat{\Phi}(\omega,k). \label{eq:Philin}
\end{align} 
The following two lemmas will suffice. 
\begin{lemma}[Uniform linear stability]  \label{lem:K0stab}
For all $\nu > 0$ sufficiently small, there holds 
\begin{align*}
\inf_{z \in \Complex: \textbf{Re}\, z \leq 0} \abs{1- \widehat{K^0}(z,k)} \geq \kappa > 0, 
\end{align*}
where $\kappa$ can be taken independent of $\nu > 0$. 
\end{lemma} 
\begin{remark} 
The $z \in \Complex$ is necessary to justify the a priori estimate on \eqref{eq:Laplace}; see \cite{MouhotVillani11,BMM13}. 
\end{remark}
\begin{proof} 
It is well known that this holds for $\nu = 0$ (see \cite{Landau46,Penrose,MouhotVillani11}).
The result follows by dominated convergence and Lemma \ref{lem:propS}. 
\end{proof}

\begin{lemma} \label{lem:K0bd}
There holds for all $j \geq 0$, 
\begin{align*}
\abs{(\abs{k}\partial_\omega)^j \widehat{K^0}(\omega,k)} & \lesssim_j 1. 
\end{align*}
\end{lemma} 
\begin{proof} 
Taking derivatives gives:
\begin{align*}
(\abs{k}\partial_\omega)^j \widehat{K^0}(\omega,k) & = \int_0^\infty \abs{k}^j (i t)^j e^{-i\omega t} e^{\delta\nu^{1/3}t} S(t,k)  \widehat{W}(k) \abs{k}^2 \left(\frac{1 - e^{-\nu t}}{\nu}\right) \widehat{\mu}\left(k \left(\frac{1 - e^{-\nu t}}{\nu}\right)\right) \dd t,
\end{align*}
and making the change of variables $s = \abs{k}(1-e^{-\nu t})/\nu$ gives the following, (noting that $t \leq e^{\nu t} (1-e^{-\nu t})/\nu$ and using Lemma \ref{lem:propS} again),  
\begin{align*}
\abs{(\abs{k}\partial_\omega)^j \widehat{K^0}(\omega,k)} & \lesssim \int_0^{\abs{k}\nu^{-1}} \abs{\widehat{W}(k) \abs{k} s^{j+1} \widehat{\mu}\left(s\right)} \dd s \lesssim_j 1,  
\end{align*}
where the bound can be taken independent of $\nu$ and $k$ (though not independent of $j$).
\end{proof}

The a priori estimate in Lemma \ref{lem:LinLandau} follows from \eqref{eq:Philin}, Lemmas \ref{lem:K0stab} and \ref{lem:K0bd}, and Plancherel's identity. See \cite{BMM16} for more details in a similar setting and \cite{MouhotVillani11,BMM13} for discussions on how to make the a priori estimate rigorous.   
\end{proof} 

\subsection{Nonlinear preliminaries}

In this section we make the key estimates for dealing with the nonlinear terms arising in the moment estimates. 
These lemmas combine ideas from \cite{BMM13,Bedrossian16} with the corrected critical times due to the collisions and the enhanced hypoelliptic smoothing inherent in $S$ from Lemma \ref{lem:propS}. 
It is in this first lemma where the enhanced collisional relaxation suppresses the plasma echoes; the $\nu^{1/3}$ threshold of \eqref{ineq:threshold} is mainly due to the second term in \eqref{ineq:GenRhointbd1}. 

\begin{lemma} \label{lem:GenCEst}
For $r = r(t,x)$ and $q = q(t,x,v)$ with $q = q_{\neq}$, define for any $\gamma \in \set{0,1,2}$ and $\alpha \geq 0$,  
\begin{align*}
\mathcal{C}(t,k) = \sum_{\ell \in \Integers_\ast^d} \int_{0}^t S(t-\tau,k) e^{\alpha \nu(t-\tau)} \hat{r}(\tau,\ell)\left(\abs{k}\frac{1-e^{-\nu(t-\tau)}}{\nu}\right)^{\gamma}\widehat{q}\left(\tau,k-\ell,k \frac{1-e^{-\nu t}}{\nu}-\ell \frac{1-e^{-\nu \tau}}{\nu}\right) \dd \tau. 
\end{align*}
If $\gamma=0$, then we have the estimate 
\begin{align}
\norm{\abs{\grad}^{1/2}\mathcal{C}}_{L^2_t(I;\cM^{\sigma,\delta})} & \lesssim \norm{r}_{L^\infty_t (I;\cM^{\beta,\delta})} \left(\sup_{\tau \in I} \jap{\tau}^{-p}\norm{q(\tau)}_{F^{\sigma+1/2,\delta_1}_{m}}\right) \nonumber \\ & \quad + \norm{\abs{\grad}^{1/2}r}_{L^2_t(I;\cM^{\sigma,\delta})} \left(\sup_{\tau \in I} \norm{q(\tau)}_{F^{n+2,\delta_1}_{m}}\right). \label{ineq:GenRhointbd1}
\end{align}
If $\gamma = 1$, then we have the estimate
\begin{align}
\norm{\abs{\grad}^{1/2}\mathcal{C}}_{L^2_t(I;\cM^{\sigma,\delta})} & \lesssim \nu^{-1/4}\norm{r}_{L^\infty_t (I;\cM^{\beta,\delta})} \left(\sup_{\tau \in I} \jap{\tau}^{-p}\norm{q(\tau)}_{F^{\sigma+1/2,\delta_1}_{m}}\right) \nonumber \\ & \quad + \nu^{-1/3}\norm{\abs{\grad}^{1/2}r}_{L^2_t(I;\cM^{\sigma,\delta})} \left(\sup_{\tau \in I} \norm{q(\tau)}_{F^{n+3,\delta_1}_{m}}\right). \label{ineq:GenRhointbd1}
\end{align}
If $\gamma = 2$, then we have the following estimate 
\begin{align}
\norm{\abs{\grad}^{1/2}\mathcal{C}}_{L^2_t(I;\cM^{\sigma,\delta})} & \lesssim \nu^{-1/4}\norm{r}_{L^\infty_t (I;\cM^{\beta,\delta})} \left(\int_0^T \jap{\tau}^{-p-2}\left(\norm{q(\tau)}_{\cD^{\sigma+1/2,\delta_1}_{m}}^2 + \norm{q(\tau)}_{F^{\sigma+1/2,\delta_1}_m}^2 \right) \dd \tau \right)^{1/2} \nonumber \\ & \quad + \nu^{-2/3}\norm{\abs{\grad}^{1/2} r}_{L^2_t(I;\cM^{\sigma,\delta})} \left(\sup_{\tau \in I} \norm{q(\tau)}_{F^{n+4,\delta_1}_{m}}\right). \label{ineq:GenRhointbd1}
\end{align}
\end{lemma} 
\begin{proof}
Observe that, 
\begin{align}
\abs{kt} & \leq \left(\frac{\nu t}{1 - e^{-\nu t}}\right)\left(\abs{\ell \frac{1-e^{-\nu \tau}}{\nu}} + \abs{k\left(\frac{1-e^{-\nu t}}{\nu}\right)- \ell \left(\frac{1-e^{-\nu \tau}}{\nu}\right)} \right) \nonumber \\
& \lesssim \left(\brak{\nu \tau} + \brak{\nu(t-\tau)}\right) \left(\abs{\ell \frac{1-e^{-\nu \tau}}{\nu}}+ \abs{k\left(\frac{1-e^{-\nu t}}{\nu}\right)- \ell \left(\frac{1-e^{-\nu \tau}}{\nu}\right)} \right). \label{ineq:ktdist}
\end{align}
By Lemma \ref{lem:propS}, for $\delta \ll 1$ and $\nu \ll 1$, there holds, 
\begin{align*}
e^{\delta \nu^{1/3} t} S(t-\tau,k) \lesssim e^{\delta \nu^{1/3}\tau} S^{3/4}(t-\tau,k). 
\end{align*}
Moreover, $\brak{\nu (t-\tau)}^\sigma S(t-\tau,k)^{1/4} \lesssim 1$ by Lemma \ref{lem:propS}. 
Therefore, we can divide the corresponding contributions (analogues of the so-called ``reaction'' and ``transport'' terms in \cite{BMM13}):
\begin{align}
\norm{\mathcal{C}}_{L^2_t(I;\cM^{\sigma,\delta})}^2 & \lesssim \mathcal{C}_{HL} + \mathcal{C}_{LH}, \label{def:RT} 
\end{align}
where, 
\begin{align*}
\mathcal{C}_{HL} & = \int_0^{T_\ast} \sum_{k \in \Integers^n_\ast} \left[\int_0^t S^{1/2}(t-\tau,k) \abs{k}^{1/2} \sum_{\ell \in \Integers^n_\ast} \brak{\ell \frac{1-e^{-\nu \tau}}{\nu}}^\sigma e^{\delta \nu^{1/3} \tau} \abs{\hat{r}(\tau,\ell)}\left(\abs{k}\frac{1-e^{-\nu(t-\tau)}}{\nu}\right)^{\gamma}  \right. \\ & \left. \quad\quad \times \brak{\nu \tau}^{\sigma} \abs{\widehat{q}\left(\tau,k-\ell, k\left(\frac{1 - e^{-\nu t}}{\nu}\right) - \ell \left(\frac{1-e^{-\nu \tau}}{\nu}\right) \right)} \dd \tau \right]^2 \dd t \\
\mathcal{C}_{LH} & = \int_0^{T_\ast} \sum_{k \in \Integers^n_\ast} \left[\int_0^t S^{1/2}(t-\tau,k) \abs{k}^{1/2} \sum_{\ell \in \Integers^n_\ast} e^{\delta \nu^{1/3} \tau} \abs{\hat{r}(\tau,\ell)}\left(\abs{k}\frac{1-e^{-\nu(t-\tau)}}{\nu}\right)^{\gamma}   \right. \\ & \left. \quad\quad \times \brak{\nu \tau}^{\sigma} \brak{k-\ell, k\left(\frac{1-e^{-\nu t}}{\nu}\right)- \ell \left(\frac{1-e^{-\nu \tau}}{\nu}\right)}^\sigma\abs{\widehat{q}\left(\tau,k-\ell, k\left(\frac{1 - e^{-\nu t}}{\nu}\right) - \ell \left(\frac{1-e^{-\nu \tau}}{\nu}\right) \right)} \dd \tau \right]^2 \dd t. 
\end{align*}
The case $\gamma = 0$ is an easy variant of the $\gamma = 1$ case and hence we omit the former.

\textbf{Case $\gamma = 1$:} \\ 
First, observe that 
\begin{align}
\abs{k \frac{1-e^{-\nu(t-\tau)}}{\nu}} & = e^{\nu \tau}\abs{k\left(\frac{1 - e^{-\nu t}}{\nu}\right) - \ell \left(\frac{1-e^{-\nu \tau}}{\nu}\right) - (k-\ell) \frac{1-e^{-\nu \tau}}{\nu}} \nonumber \\ 
& \lesssim e^{\nu \tau} \brak{\tau}\brak{k-\ell, k\left(\frac{1 - e^{-\nu t}}{\nu}\right) - \ell \left(\frac{1-e^{-\nu \tau}}{\nu}\right)} \label{ineq:dvCTbd}
\end{align}
Therefore, for all $\delta_1 > 0$, there holds the following, using $\brak{\tau} e^{-\frac{1}{4}\delta_1 \nu^{1/3}t} \lesssim \nu^{-1/3}$: 
\begin{align*}
\cC_{HL} & \lesssim \int_0^{T_\ast} \sum_{k \in \Integers_\ast^n} \left[\int_0^t S^{1/2}(t-\tau,k) \abs{k}^{1/2} \sum_{\ell \in \Integers_\ast^n} \brak{\ell \frac{1-e^{-\nu \tau}}{\nu}}^\sigma e^{\delta \nu^{1/3} \tau} \abs{\hat{r}(\tau,\ell)}  \right. \\ & \left. \quad\quad \times  \brak{k-\ell, k\left(\frac{1 - e^{-\nu t}}{\nu}\right) - \ell \left(\frac{1-e^{-\nu \tau}}{\nu}\right)}
\brak{\tau} e^{\nu \tau} \right. \\ & \left. \quad\quad \times  \abs{\widehat{q}\left(\tau,k-\ell, k\left(\frac{1 - e^{-\nu t}}{\nu}\right) - \ell \left(\frac{1-e^{-\nu \tau}}{\nu}\right) \right)} \dd \tau \right]^2 \dd t \\ 
& \lesssim \nu^{-2/3} \int_0^{T_\ast} \sum_{k \in \Integers_\ast^n} \left[\int_0^t S^{1/4}(t-\tau,k) \sum_{\ell \in \Integers_\ast^n} \abs{\ell}^{1/2} \brak{\ell \frac{1-e^{-\nu \tau}}{\nu}}^\sigma e^{\delta \nu^{1/3} \tau} \abs{\hat{r}(\tau,\ell)} \right. \\ & \left. \quad\quad \times \brak{k-\ell, k\left(\frac{1 - e^{-\nu t}}{\nu}\right) - \ell \left(\frac{1-e^{-\nu \tau}}{\nu}\right)}^{3/2} e^{\frac{1}{4}\delta_1\nu^{1/3}\tau} \right. \\ & \left. \quad\quad \times \abs{\widehat{q}\left(\tau,k-\ell, k\left(\frac{1 - e^{-\nu t}}{\nu}\right) - \ell \left(\frac{1-e^{-\nu \tau}}{\nu}\right) \right)} \dd \tau \right]^2 \dd t. 
\end{align*}
By Sobolev embedding in frequency and \eqref{ineq:FHsj}, for $\zeta > 0$ (using also Lemma \ref{lem:propS}), 
\begin{align*}
\cC_{HL} & \lesssim \nu^{-2/3} \norm{q}_{L^\infty_t F^{\zeta,\delta_1}_{m}}^2
 \int_0^\infty \sum_{k \neq 0} \left[\int_0^t \sum_{\ell \neq 0}  S^{1/4}(t-\tau,k)\abs{\ell}^{1/2} \brak{\ell \frac{1-e^{-\nu \tau}}{\nu}}^\sigma e^{\delta \nu^{1/3} \tau} \abs{\hat{r}(\tau,\ell)} \right. \\ & \left. \quad\quad \times \brak{k-\ell, k\left(\frac{1 - e^{-\nu t}}{\nu}\right) - \ell \left(\frac{1-e^{-\nu \tau}}{\nu}\right)}^{\frac{3}{2}-\zeta} e^{m \nu \tau - \frac{1}{4}\delta_1 \nu^{1/3}\tau } \dd \tau \right]^2 \dd t.
\end{align*}
By Schur's test, 
\begin{align}
\cC_{HL} & \lesssim \nu^{-2/3} \norm{q}_{L^\infty_t F^{\zeta,\delta_1}_{m}}^2 \norm{\abs{\grad}^{1/2}r}_{L^2 \cM^{\sigma,\delta}}^2 \nonumber \\ & \quad\quad \times \left(\sup_{t \in I} \sup_{k \in \Integers^n_\ast} \sum_{\ell \in \Integers^n_\ast}\int_0^t \cK(t,\tau,k,\ell) \dd \tau\right)\left(\sup_{\tau \in I} \sup_{\ell \in \Integers^n_\ast} \sum_{\ell \in \Integers^n_\ast}\int_\tau^T \cK(t,\tau,k,\ell) \dd \tau\right), \label{ineq:Schur}
\end{align}
where 
\begin{align*}
\cK(t,\tau,k,\ell) = S^{1/2}(t-\tau,k)e^{-\frac{1}{5}\delta_1 \nu^{1/3}\tau} \brak{k-\ell, k\left(\frac{1 - e^{-\nu t}}{\nu}\right) - \ell \left(\frac{1-e^{-\nu \tau}}{\nu}\right)}^{\frac{3}{2}-\zeta}.
\end{align*}
By the change of variables $s= \abs{\ell}(1-e^{-\nu \tau})\nu^{-1}$, 
\begin{align*}
\sum_{\ell\in \Integers^n_\ast} \int_0^t \mathcal{K}(t,\tau,k,\ell) \dd \tau  & \lesssim \sum_{\ell \in \Integers_\ast^n} \int_0^t e^{-\frac{1}{6}\delta_1 \nu^{1/3} \tau - \nu \tau}\brak{k-\ell, k\left(\frac{1-e^{-\nu t}}{\nu}\right)- \ell \left(\frac{1-e^{-\nu \tau}}{\nu}\right)}^{\frac{3}{2}-\zeta} \dd \tau \\
& \lesssim \sum_{\ell \in \Integers^n_\ast} \int_{-\infty}^{\infty}  \brak{k-\ell, k\left(\frac{1-e^{-\nu t}}{\nu}\right)- \frac{\ell}{\abs{\ell}}s}^{\frac{3}{2}-\zeta} \dd s \\
& \lesssim 1,  
\end{align*}
provided that $\zeta > n+ \frac{3}{2}$. 
For the other kernel estimate in \eqref{ineq:Schur}, we use the change of variables $s= \abs{k}(1-e^{-\nu t})\nu^{-1}$ and Lemma \ref{lem:propS} to deduce $S^{1/4}(t-\tau,k) \lesssim e^{-\nu (t-\tau)}$, which implies 
\begin{align*}
\sum_{k \in \Integers^n_\ast} \int_\tau^{T^\ast} \mathcal{K}(t,\tau,k,\ell) \dd t  & \lesssim \sum_{k \in \Integers_\ast^n} \int_\tau^{T^\ast} S^{1/4}(t-\tau,k) e^{-\frac{1}{6}\delta_1 \nu^{1/3} \tau - \nu \tau}\brak{k-\ell, k\left(\frac{1-e^{-\nu t}}{\nu}\right)- \ell \left(\frac{1-e^{-\nu \tau}}{\nu}\right)}^{\frac{3}{2}-\zeta} \dd t  \\
& \lesssim \sum_{\ell \in \Integers^n_\ast} \int_{-\infty}^{\infty} \brak{k-\ell, \frac{k}{\abs{k}}s -  \abs{\ell}(1-e^{-\nu \tau})\nu^{-1}}^{\frac{3}{2}-\zeta} \dd s \\
& \lesssim 1.  
\end{align*}
This completes the treatment of $\cC_{HL}$. 

Turn next to $\cC_{LH}$. 
For this term we use the hypoelliptic smoothing from Lemma \ref{lem:propS}:
\begin{align}
S^{1/4}(t-\tau,k)\abs{k}\left( \frac{1 - e^{-\nu(t-\tau)}}{\nu}\right) & \lesssim \abs{k}^{1/2}\left( \frac{1 - e^{-\nu(t-\tau)}}{\nu}\right)^{1/2} \frac{1}{\nu^{1/4}(t-\tau)^{1/4}}.   
\end{align}
Hence, by \eqref{ineq:dvCTbd}, we have 
\begin{align*}
\cC_{LH} & \lesssim \nu^{-1/2}\int_0^{T_\ast} \sum_{k \in \Integers_\ast^n } \left[\int_0^t \sum_{\ell \in \Integers_\ast^n} e^{\delta \nu^{1/3} \tau} \abs{\hat{r}(\tau,\ell)} \frac{e^{\nu \tau}\brak{\tau}^{1/2}}{(t-\tau)^{1/4}} S^{1/2}(t-\tau,k)  \right. \\ & \left. \quad\quad \times \brak{\nu \tau}^{\sigma} \abs{k}^{1/2} \brak{k-\ell, k\left(\frac{1-e^{-\nu t}}{\nu}\right)- \ell \left(\frac{1-e^{-\nu \tau}}{\nu}\right)}^{\sigma+1/2} \right. \\ & \left. \quad\quad \times \abs{\widehat{q}\left(\tau,k-\ell, k\left(\frac{1 - e^{-\nu t}}{\nu}\right) - \ell \left(\frac{1-e^{-\nu \tau}}{\nu}\right) \right)} \dd \tau \right]^2 \dd t. 
\end{align*}
From here, we apply an argument analogous to one used in \cite{BMM13}. 
By Cauchy-Schwarz and $\nu \ll 1$, 
\begin{align*}
\cC_{LH} & \lesssim \nu^{-1/2} \left(\sup_{t\in I}\int_0^t \sum_{\ell \in \Integers_\ast^n} e^{\delta \nu^{1/3} \tau} \abs{\hat{r}(\tau,\ell)} \frac{\brak{\tau}^{1/2+p/2}}{(t-\tau)^{1/2}} \dd \tau\right) \\ 
& \quad \times \sum_{k \in \Integers_\ast^n} \int_0^{T_\ast}  \int_0^t \left(\sum_{\ell \in \Integers_\ast^n}e^{\delta \nu^{1/3} \tau} \abs{\hat{r}(\tau,\ell)} \brak{\tau}^{1/2+p/2} \right)S(t-\tau,k) \\ 
& \quad\quad \times  \abs{k}e^{\frac{1}{2}\delta_1\nu^{1/3}\tau} \brak{\tau}^{-p} \abs{\widehat{\brak{\grad}^{\sigma+1/2}q}\left(\tau,k-\ell, k\left(\frac{1 - e^{-\nu t}}{\nu}\right) - \ell \left(\frac{1-e^{-\nu \tau}}{\nu}\right) \right)}^2 \dd \tau  \dd t.
\end{align*}
Next, make the change of variables 
\begin{align*}
\abs{k}\frac{1-e^{-\nu t}}{\nu} = s
\end{align*}
and note that by Lemma \ref{lem:propS}, $S(t-\tau,k) e^{\nu t} \lesssim e^{\nu \tau}$, to deduce (for $\beta \geq p+\frac{1}{2}$ and $t(s) = -\frac{1}{\nu}\log(1+ \frac{s\nu}{\abs{k}})$), 
\begin{align*}
\cC_{LH} & \lesssim \nu^{-1/2} \norm{r}_{L^\infty_t \cM^{\beta,\delta}} \\ 
& \quad \times \sum_{k \neq 0} \int_0^{\abs{k}\frac{1-e^{-\nu T_\ast}}{\nu}}  \int_0^{t(s)} \left(\sum_{\ell \neq 0}e^{\delta \nu^{1/3} \tau} \abs{\hat{r}(\tau,\ell)} \brak{\tau}^{1/2+p} \right) \\ 
& \quad\quad \times e^{\frac{3}{4}\delta_1\nu^{1/3}\tau} \brak{\tau}^{-2p} \abs{\widehat{\brak{\grad}^{\sigma+1/2}q}\left(\tau,k-\ell, \frac{k}{\abs{k}}s - \ell \left(\frac{1-e^{-\nu \tau}}{\nu}\right) \right)}^2 \dd \tau  \dd s \\ 
& \lesssim \nu^{-1/2} \norm{r}_{L^\infty_t \cM^{\beta,\delta}}^2 \\ 
& \quad \times \sum_{k \neq 0} \int_0^{\abs{k}\frac{1-e^{-\nu T_\ast}}{\nu}}  \int_0^{t(s)} \sum_{\ell \neq 0}\frac{1}{\brak{\ell}^{n+1}}  \\ 
& \quad\quad \times e^{\frac{3}{4}\delta_1\nu^{1/3}\tau} \brak{\tau}^{-2p} \abs{\widehat{\brak{\grad}^{\sigma+1/2}q}\left(\tau,k-\ell, \frac{k}{\abs{k}}s - \ell \left(\frac{1-e^{-\nu \tau}}{\nu}\right) \right)}^2 \dd \tau  \dd s \\ 
& \lesssim \nu^{-1/2} \norm{r}_{L^\infty_t \cM^{\beta,\delta}}^2 \int_0^{T_\ast} \sum_{\ell \neq 0}\frac{1}{\brak{\ell}^{n+1}} \sum_{k \neq 0} e^{\frac{3}{4}\delta_1\nu^{1/3}\tau} \brak{\tau}^{-2p} \\ & \quad\quad \times \int_{-\infty}^\infty \abs{\widehat{\brak{\grad}^{\sigma+1/2}q}\left(\tau,k-\ell, \frac{k}{\abs{k}}s - \ell \left(\frac{1-e^{-\nu \tau}}{\nu}\right) \right)}^2 \dd s \dd \tau \\ 
& \lesssim \nu^{-1/2} \norm{r}_{L^\infty_t \cM^{\beta,\delta}}^2 \int_0^{T_\ast} \sum_{\ell \neq 0}\frac{1}{\brak{\ell}^{n+1}} \sum_{k \neq 0} e^{\frac{3}{4}\delta_1\nu^{1/3}\tau} \brak{\tau}^{-2p} \\ & \quad\quad \times \sup_{x,\omega} \int_{-\infty}^\infty \abs{\widehat{\brak{\grad}^{\sigma+1/2}q}\left(\tau,k-\ell, \omega s - x\right)}^2 \dd s \dd \tau, \\ 
& \lesssim \nu^{-1/2} \norm{r}_{L^\infty_t \cM^{\beta,\delta}}^2 \left(\sup_{\tau \in I} \brak{\tau}^{-p} \norm{q(\tau)}_{F^{\sigma+1/2,\delta_1}_m}^2\right),  
\end{align*}  
where the last line followed by the Sobolev trace lemma applied in Fourier (see e.g. \cite{BMM13}).  

\textbf{Case $\gamma = 2$:} \\  
The $\cC_{HL}$ case is treated similarly to the $\gamma = 1$ only noting that there is an extra power of $t$. 
The main difference is the treatment of the $\cC_{LH}$ term. 
We begin with the same hypoellipticity, 
\begin{align}
S^{1/3}(t-\tau,k) \abs{k}^2 \left(\frac{1-e^{-\nu(t-\tau)}}{\nu}\right)^2 & \lesssim  \frac{1}{\nu^{-1/4}(t-\tau)^{1/4}} \abs{k}^{3/2} \left(\frac{1-e^{-\nu(t-\tau)}}{\nu}\right)^{3/2}, \label{ineq:hypoelip2}
\end{align}
to deduce 
\begin{align*}
\cC_{LH} & \lesssim \nu^{-1/2} \left(\sup_{t\in I}\int_0^t \sum_{\ell \neq 0} e^{\delta \nu^{1/3} \tau} \abs{\hat{r}(\tau,\ell)} \frac{\brak{\tau}^{1/2+p/2}}{(t-\tau)^{1/2}} \dd \tau\right) \\ 
& \quad \times \sum_{k \in \Integers_\ast^n} \int_0^{T_\ast}  \int_0^t \left(\sum_{\ell \in \Integers_\ast^n}e^{\delta \nu^{1/3} \tau} \abs{\hat{r}(\tau,\ell)} \brak{\tau}^{1/2+p/2} \right)S(t-\tau,k) \\ 
& \quad\quad \times  \abs{k}e^{\frac{1}{2}\delta_1\nu^{1/3}\tau} \brak{\tau}^{-p} \abs{\partial_v^t \widehat{\brak{\grad}^{\sigma+1/2}q}\left(\tau,k-\ell, k\left(\frac{1 - e^{-\nu t}}{\nu}\right) - \ell \left(\frac{1-e^{-\nu \tau}}{\nu}\right) \right)}^2 \dd \tau  \dd t \\ 
& \lesssim \nu^{-1/2}\norm{r}_{L^\infty_t \cM^{\beta,\delta}}^2 \sum_{k \in \Integers_\ast^n} \int_0^{T_\ast}  \int_0^t \sum_{\ell \in \Integers_\ast^n}\frac{\abs{k}}{\brak{\tau}^{2}\brak{\ell}^{n+1}}e^{\frac{1}{2}\delta_1\nu^{1/3}\tau} \brak{\tau}^{-p} \\ & \quad\quad \times \abs{\partial_v^t \widehat{\brak{\grad}^{\sigma+1/2}q}\left(\tau,k-\ell, k\left(\frac{1 - e^{-\nu t}}{\nu}\right) - \ell \left(\frac{1-e^{-\nu \tau}}{\nu}\right) \right)}^2 \dd \tau  \dd t. 
\end{align*}
Make the change of variables $s = \abs{k}\frac{1-e^{-\nu t}}{\nu}$ to deduce, 
\begin{align*}
\cC_{HL} & \lesssim \nu^{-1/2}\norm{r}_{L^\infty_t \cM^{\beta,\delta}}^2 \sum_{k \in \Integers_\ast^n} \int_0^{\abs{k}\nu^{-1}}  \int_0^t \sum_{\ell \in \Integers_\ast^n}\frac{1}{\brak{\tau}^{2}\brak{\ell}^{n+1}}e^{\frac{3}{4}\delta_1\nu^{1/3}\tau} \brak{\tau}^{-p} \\ & \quad\quad \times \abs{\partial_v^t \widehat{\brak{\grad}^{\sigma+1/2}q}\left(\tau,k-\ell, \frac{k}{\abs{k}}s - \ell \left(\frac{1-e^{-\nu \tau}}{\nu}\right) \right)}^2 \dd \tau  \dd s \\ 
& \lesssim \nu^{-1/2}\norm{r}_{L^\infty_t \cM^{\beta,\delta}}^2 \int_0^t \sum_{\ell \in \Integers_\ast^n}\frac{1}{\brak{\tau}^{2}\brak{\ell}^{n+1}}e^{\frac{3}{4}\delta_1\nu^{1/3}\tau} \brak{\tau}^{-p} \\ & \quad\quad \times \sum_{k \in \Integers_\ast^n} \int_{-\infty}^{\infty} \abs{\partial_v^t \widehat{\brak{\grad}^{\sigma+1/2}q}\left(\tau,k-\ell, \frac{k}{\abs{k}}s - \ell \left(\frac{1-e^{-\nu \tau}}{\nu}\right) \right)}^2 \dd s \dd \tau \\ 
 & \lesssim \nu^{-1/2}\norm{r}_{L^\infty_t \cM^{\beta,\delta}}^2 \int_0^t\frac{1}{\brak{\tau}^{2}}e^{\frac{3}{4}\delta_1\nu^{1/3}\tau} \brak{\tau}^{-p} \\ & \quad\quad \times \sup_{x,\omega}\sum_{k \in \Integers_\ast^n} \int_{-\infty}^{\infty} \abs{\partial_v^t \widehat{\brak{\grad}^{\sigma+1/2}q}\left(\tau,k, \omega s - x\right)}^2 \dd s \dd \tau. 
\end{align*}
By the Sobolev trace lemma and \eqref{ineq:cDHsj}, it follows that 
\begin{align*}
\cC_{HL} & \lesssim  \nu^{-1/2}\norm{r}_{L^\infty_t \cM^{\beta,\delta}}^2 \int_0^{T_\ast} \brak{\tau}^{-p-2} e^{\frac{3}{4}\delta_1\nu^{1/3} \tau} \norm{\brak{v}^{n/2+}\partial_v^t \brak{\grad}^{\sigma+1/2}q(\tau)}_2^2 \dd \tau \\ 
& \lesssim \nu^{-1/2}\norm{r}_{L^\infty_t \cM^{\beta,\delta}}^2 \int_0^{T_\ast} \brak{\tau}^{-p-2} \norm{q(\tau)}^2_{\cD^{\sigma+1/2,\delta_1}_m} + \brak{\tau}^{-p-2}\norm{q(\tau)}^2_{F^{\sigma+1/2,\delta_1}_m} \dd \tau, 
\end{align*}
which completes the proof. 
\end{proof} 

The next lemma is the space-homogeneous variant of the previous lemma. 
\begin{lemma} \label{lem:GenCEst0mode}
For $r = r(t,x)$ and $q_0 = q_0(t,v)$, define for $\gamma \in \set{0,1,2}$ and $\alpha \geq 0$,  
\begin{align*}
\mathcal{C}(t,k) = \sum_{k \in \Integers_\ast^d}\int_{0}^t S(t-\tau,k) e^{\alpha \nu(t-\tau)} \hat{r}(\tau,k)\left(\abs{k}\frac{1-e^{-\nu(t-\tau)}}{\nu}\right)^{\gamma}\widehat{q}_0\left(\tau, k \left(\frac{1-e^{-\nu(t-\tau)}}{\nu}\right) \right) \dd \tau. 
\end{align*}
Then, 
\begin{align}
\norm{\abs{\grad}^{1/2}\mathcal{C}}_{L^2_t(I;\cM^{\sigma,\delta})} & \lesssim \nu^{-\gamma/3}\norm{r}_{L^\infty_t (I;\cM^{\beta,\delta})}\left(\sup_{t\in I} \brak{t}^{-\beta/2}\norm{q_0(t)}_{H^{\sigma}_m}\right) \nonumber \\ & \quad + \norm{\abs{\grad}^{1/2}r}_{L^2_t(I;\cM^{\sigma,\delta})} \norm{q_0}_{L^\infty_t H^{n+\gamma+2}_{m}}
\end{align}
\end{lemma}
\begin{remark} \label{rmk:HsmInterp}
By Sobolev interpolation, for $\sigma=s(1-\zeta) + (\sigma+1/2)\zeta$ (using also \eqref{ineq:FHsj} and \eqref{def:hfrelat}),  
\begin{align*}
\norm{h_0}_{H^\sigma_m} \lesssim \norm{h_0}_{H^{s}_m}^{1-\zeta} \norm{h_0}_{H^{\sigma+1/2}_m}^{\zeta} \lesssim \norm{h_0}_{H^{s}_m}^{1-\zeta} \norm{h_0}_{H^{\sigma+1/2}_m}^{\zeta} \lesssim \norm{h_0}_{H^{s}_m}^{1-\zeta}e^{\zeta (m+\frac{n}{2})\nu t}   \norm{f}_{F^{\sigma+1/2}_m}^{\zeta}. 
\end{align*}
Therefore, if $1-\zeta \geq \zeta(m+\frac{n}{2})$ and $\beta \geq \zeta p/2$, then by the bootstrap hypotheses, there holds
\begin{align}
\sup_{t\in I} \brak{t}^{-\beta/2} \norm{h_0(t)}_{H^\sigma_m} & \lesssim \sup_{t\in I} \norm{h_0(t)}_{H^{s}_m}^{1-\zeta}e^{\zeta (m+\frac{n}{2})\nu t}   \left(\brak{t}^{-\beta/2}\norm{f(t)}_{F^{\sigma+1/2}_m}^{\zeta}\right) \nonumber \\
& \lesssim \sup_{t\in I} e^{-(1-\zeta)\nu t + \zeta(m+\frac{n}{2})\nu t} \eps \nonumber \\
& \lesssim \eps. \label{ineq:h0Hsig}
\end{align}
Using $(\sigma-s)/(\sigma+1/2-s) = \zeta$, it is clear that we may choose $\sigma-s$ depending only on $m$ and $n$ to accomplish this. 
\end{remark}
\begin{proof}
First, observe that 
\begin{align*}
\brak{k,kt}^{\sigma} & \lesssim \brak{k,k\tau}^{\sigma} + \brak{k(t-\tau)}^{\sigma} \lesssim \brak{k,k\tau}^{\sigma} + \brak{\nu (t-\tau)}^{\sigma}\brak{k\frac{1-e^{-\nu(t-\tau)}}{\nu}}^{\sigma}. 
\end{align*}
The  $\brak{\nu(t-\tau)}^\sigma$ is absorbed by $S(t-\tau,k)$ via Lemma \ref{lem:propS}. Hence, as in Lemma \ref{lem:GenCEst}, we write
\begin{align}
\norm{\mathcal{C}}_{L^2_t(I;\cM^{\sigma,\delta})}^2 & \lesssim \mathcal{C}_{HL} + \mathcal{C}_{LH}, \label{def:RTgenest} 
\end{align}
where, 
\begin{align*}
\mathcal{C}_{HL} & = \int_0^{T_\ast} \sum_{k \in \Integers_\ast^n} \left[\int_0^t S^{1/2}(t-\tau,k) \abs{k}^{1/2} \brak{k, k \tau}^\sigma e^{\delta \nu^{1/3} \tau} \abs{\hat{r}(\tau,k)}\left(\abs{k}\frac{1-e^{-\nu(t-\tau)}}{\nu}\right)^{\gamma}  \right. \\ & \left. \quad\quad \times \abs{\widehat{q}_0\left(\tau, k\left(\frac{1 - e^{-\nu (t-\tau)}}{\nu} \right)\right)} \dd \tau \right]^2 \dd t \\
\mathcal{C}_{LH} & = \int_0^{T_\ast} \sum_{k \in \Integers_\ast^n} \left[\int_0^t S^{1/2}(t-\tau,k) \abs{k}^{1/2} e^{\delta \nu^{1/3} \tau} \abs{\hat{r}(\tau,k)}\left(\abs{k}\frac{1-e^{-\nu(t-\tau)}}{\nu}\right)^{\gamma}  \right. \\ & \left. \quad\quad \times \brak{k\frac{1-e^{-\nu(t-\tau)}}{\nu}}^{\sigma} \abs{\widehat{q}_0\left(\tau, k\left(\frac{1 - e^{-\nu (t-\tau)}}{\nu} \right)\right)} \dd \tau \right]^2 \dd t. 
\end{align*}
To treat $\cC_{HL}$, apply Sobolev embedding and  Young's inequality: 
\begin{align*}
\cC_{HL} & \lesssim \norm{\abs{\grad}^{1/2} r}_{L^2_t \cM^{\sigma,\delta}}^2 \norm{q_0}_{L^\infty_t H^{n+\gamma+2}_m}^2 \left(\sup_{t \in I}\sup_{k \in \Integers^n_\ast} \int_0^t S^{1/2}(\tau,k)\brak{k\frac{1 - e^{-\nu \tau}}{\nu}}^{-n-1} \dd \tau \right)^2 \\
& \lesssim \norm{\abs{\grad}^{1/2} r}_{L^2_t \cM^{\sigma,\delta}}^2 \norm{q_0}_{L^\infty_t H^{n + \gamma +2}_m}^2,
\end{align*}
where the last line followed by the change of variables $s = \abs{k} \frac{1-e^{-\nu \tau}}{\nu}$ and Lemma \ref{lem:propS}. 

To treat the $\cC_{LH}$ term we use Lemma \ref{lem:propS} to deduce 
\begin{align*}
S^{1/2}(t-\tau,k) \left(\frac{1-e^{-\nu(t-\tau)}}{\nu}\right)^{\gamma} \lesssim \nu^{-\gamma/3}, 
\end{align*}
and hence we deduce 
\begin{align*}
\mathcal{C}_{LH} & \lesssim \nu^{-2\gamma/3}\int_0^{T_\ast} \sum_{k \in \Integers_\ast^n} \left[\int_0^t S^{1/2}(t-\tau,k) \abs{k}^{1/2} e^{\delta \nu^{1/3} \tau} \abs{k}^\gamma \abs{\hat{r}(\tau,k)}\abs{\widehat{\brak{\grad}^{\sigma}q}_0\left(\tau, k\left(\frac{1 - e^{-\nu (t-\tau)}}{\nu} \right)\right)} \dd \tau \right]^2 \dd t \\ 
& \lesssim \nu^{-2\gamma/3} \norm{r}^2_{L^\infty_t \cM^{\beta,\delta}} \int_0^{T_\ast} \sum_{k \in \Integers_\ast^n} \left[\int_0^t S^{1/2}(t-\tau,k) \frac{1}{\brak{k}^{n+1}\brak{\tau}^{2}}\abs{\widehat{\brak{\grad}^{\sigma}q}_0\left(\tau, k\left(\frac{1 - e^{-\nu (t-\tau)}}{\nu} \right)\right)} \dd \tau \right]^2 \dd t \\ 
& \lesssim \nu^{-2\gamma/3} \norm{r}^2_{L^\infty_t \cM^{\beta,\delta}} \sum_{k \in \Integers_\ast^n} \int_0^{T_\ast}  \int_0^t S(t-\tau,k) \frac{1}{\brak{k}^{2n+2}\brak{\tau}^{2}}\abs{\widehat{\brak{\grad}^{\sigma}q}_0\left(\tau, k\left(\frac{1 - e^{-\nu (t-\tau)}}{\nu} \right)\right)}^2 \dd \tau \dd t \\ 
& \lesssim \nu^{-2\gamma/3} \norm{r}^2_{L^\infty_t \cM^{\beta,\delta}} \sum_{k \in \Integers_\ast^n} \int_0^{T_\ast}  \int_{-\infty}^\infty \frac{1}{\brak{k}^{2n+2}\brak{\tau}^{2}}\abs{\widehat{\brak{\grad}^{\sigma}q}_0\left(\tau,\frac{k}{\abs{k}}s\right)}^2 \dd s \dd \tau \\ 
& \lesssim \nu^{-2\gamma/3} \norm{r}_{L^\infty_t \cM^{\beta,\delta}}^2 \norm{q_0}_{L^\infty_t H^{\sigma}_m}^2. 
\end{align*}
\end{proof}

\subsection{$L^2_t \cM^{\sigma,\delta}$ estimate on the density} \label{sec:rhoHi}
In this section, we improve \eqref{ineq:bootArho}. 
For notational convenience, write
\begin{align}
B(t,k) = e^{\delta \nu^{1/3}t} \abs{k}^{1/2} \brak{k,tk}^\sigma. \label{def:B}
\end{align}
By Lemma \ref{lem:LinLandau}, there holds, 
\begin{align}
\norm{\abs{\grad}^{1/2}\rho}_{L^2_{t} \cM^{\sigma,\delta}}^2 & \lesssim \cI  + \mathcal{N}_{CL} + J_{\cC_{\mu}} + J_{\cC_{h}}, \label{ineq:Brho}
\end{align}
where the three contributions are from the collisionless nonlinearity and the decaying contributions from the collision operator,
\begin{subequations} \label{def:BrhoTrms}
\begin{align}
\cI & = \int_0^{T_\ast} \sum_{k \neq 0} \abs{B(t,k) S(t,k) \widehat{h_{in}}(k,k\frac{1-e^{-\nu t}}{\nu})}^2 \dd t \\
\mathcal{N}_{CL} & = \int_0^\infty \sum_{k \neq 0} \left[B(t,k)\int_0^t S(t-\tau,k) \sum_{\ell \neq 0} \hat{\rho}(\tau,\ell) \widehat{W}(\ell) \ell \cdot k\left( \frac{1 - e^{-\nu(t-\tau)}}{\nu}\right)  \right. \\ & \left. \quad\quad \times \widehat{f}\left(\tau,k-\ell, k\left(\frac{1 - e^{-\nu t}}{\nu}\right) - \ell \left(\frac{1-e^{-\nu \tau}}{\nu}\right) \right) \dd \tau \right]^2 \dd t \\
J_{\cC_{\mu}} & = \int_0^\infty \sum_{k \neq 0} \left[B(t,k)\int_0^t S(t-\tau,k) \widehat{\cC_{\mu}}\left(\tau,\bar{\eta}(\tau;k,\eta_{CT}(t))\right) \right]^2 \dd t \\
J_{\cC_{h}} & = \int_0^\infty \sum_{k \neq 0} \left[B(t,k)\int_0^t S(t-\tau,k) \widehat{\cC_{h}}\left(\tau,\bar{\eta}(\tau;k,\eta_{CT}(t))\right) \right]^2 \dd t. 
\end{align}
\end{subequations}
By $\abs{t} \leq (e^{\nu t}-1)/\nu$, the change of coordinates $s = \abs{k}(1-e^{-\nu t})/\nu$, Lemma \ref{lem:propS}, and the Sobolev trace lemma (see e.g. the usage in \cite{BMM13}), there holds
\begin{align}
\cI  & \lesssim \sum_{k \in \Integers_\ast^n}\int_0^\infty e^{2\delta \nu^{1/3 }t} S^2(t,k) \brak{k,kt}^{2\sigma}  \abs{k} \abs{\widehat{h_{in}}\left(k,k\left(\frac{1-e^{-\nu t}}{\nu}\right) \right)   }^2 \dd t \nonumber \\ 
& \lesssim \sum_{k \in \Integers_\ast^n} \int_0^\infty e^{2\delta \nu^{1/3 }t} S^2(t,k) e^{\sigma \nu t}  \abs{k} \abs{\widehat{\brak{\grad}^\sigma h_{in}}\left(k,k\left(\frac{1-e^{-\nu t}}{\nu}\right) \right)   }^2 \dd t \nonumber \\ 
& \lesssim \sum_{k \in \Integers_\ast^n} \int_{-\infty}^{\infty} \abs{\widehat{\brak{\grad}^\sigma h_{in}}\left(k,\frac{k}{\abs{k}}s\right)   }^2 \dd s \nonumber \\ 
& \lesssim \sum_{k \in \Integers_\ast^n} \sup_{\omega \in \mathbb S^{n-1}} \int_{-\infty}^{\infty} \abs{\widehat{\brak{\grad}^\sigma h_{in}}\left(k, \omega s\right)   }^2 \dd s \nonumber \\ 
& \lesssim \norm{h_{in}}_{H^{\sigma}_{m}}^2 \nonumber \\ 
& \lesssim \eps^2.  \label{ineq:Bhin} 
\end{align}
The remainder of the section is dedicated to dealing with the nonlinear contributions in \eqref{ineq:Brho}. 

\subsubsection{Nonlinear collisionless contributions}
In this section we deal with the contribution of the acceleration term, $\cN_{CL}$, in \eqref{ineq:Brho}.  
Consider the contributions involving $h_0$:  
\begin{align*}
\mathcal{N}_{CL;0} & := \int_0^{T_\ast} \sum_{k \in \Integers_\ast^n} \left[B(t,k)\int_0^t S(t-\tau,k) \hat{\rho}(\tau,k) \widehat{W}(k)  \abs{k}^2 \left( \frac{1 - e^{-\nu(t-\tau)}}{\nu}\right)  \right. \\ & \left. \quad\quad \times \widehat{h}\left(\tau,0,k\left(\frac{1 - e^{-\nu(t-\tau)}}{\nu}\right)\right) \dd \tau \right]^2 \dd t.
\end{align*}
By Lemma \ref{lem:GenCEst0mode} and Remark \ref{rmk:HsmInterp}, it follows that 
\begin{align*}
\abs{\mathcal{N}_{CL;0}} & \lesssim \norm{\abs{\grad}^{1/2} \rho}^2_{L^2_t \cM^{\sigma,\delta}}\norm{h_0(t)}^2_{L^\infty_t H^{n/2+2}_{m}} + \nu^{-2/3}\norm{\rho}_{L^\infty_t \cM^{\beta,\delta}}^2 \norm{h_0}_{L^\infty_t H^{\sigma}_{m}}^2 \\ 
& \lesssim \eps^4 + \nu^{-2/3}\eps^4, 
\end{align*}
which is consistent with Proposition \ref{prop:boot} by choosing $\eps < c_0 \nu^{1/3}$ for $c_0$ sufficiently small (depending on the bootstrap constants). 

Next, we consider the case $k\neq \ell$:
\begin{align*}
\mathcal{N}_{CL;\neq} & := \int_0^{T_\ast} \sum_{k \in \Integers_\ast^n} \left[B(t,k)\int_0^t \sum_{\ell \in \Integers_\ast^n} S(t-\tau,k) \hat{\rho}(\tau,\ell) \widehat{W}(\ell)\ell \cdot k \left( \frac{1 - e^{-\nu(t-\tau)}}{\nu}\right)  \right. \\ & \left. \quad\quad \times \widehat{f}\left(\tau,k-\ell, k\left(\frac{1 - e^{-\nu t}}{\nu}\right) - \ell \left(\frac{1-e^{-\nu \tau}}{\nu}\right)\right) \dd \tau \right]^2 \dd t.
\end{align*}
By Lemma \ref{lem:GenCEst}, there holds for $\delta$ and $\nu$ sufficiently small, 
\begin{align*}
\abs{\mathcal{N}_{CL,\neq}} & \lesssim \nu^{-1/2}\norm{\abs{\grad}^{1/2}\rho}_{L^\infty_t \cM^{\beta,\delta}}^2\left(\sup_{\tau \in [0,T^\ast]} \brak{\tau}^{-p}\norm{f(\tau)}_{F^{\sigma+1/2,\delta_1}_m}\right)^2 \\ 
& \quad + \nu^{-2/3}\norm{\rho}_{L^2_t \cM^{\sigma,\delta}}^2 \left(\sup_{\tau \in [0,T^\ast]} \norm{f(\tau)}_{F^{n+1,\delta_1}_m}\right)^2 \\ 
& \lesssim \nu^{-1/2}\eps^4 + \nu^{-2/3}\eps^4. 
\end{align*}
This is consistent with Proposition \ref{prop:boot} by choosing $\eps < c_0 \nu^{1/3}$ for $c_0$ sufficiently small. 

\subsubsection{$J_{\cC_\mu}$ collision contributions}
Consider next the contributions from $\cC_\mu$. From \eqref{def:cCmu}, we have 
\begin{align*}
J_{\cC_\mu} & \lesssim \int_0^{T^\ast} \sum_{k \in \Integers_\ast^n} \left[B(t,k)\int_0^t \nu S(t-\tau,k) \widehat{M_T}(\tau,k) \left(\widehat{\Delta_v \mu}\right)\left(k\left(\frac{1-e^{-\nu (t-\tau)}}{\nu}\right)\right) \dd \tau \right]^2 \dd t \\ 
& \quad + \int_0^{T^\ast} \sum_{k \in \Integers_\ast^n} \left[B(t,k)\int_0^t \nu S(t-\tau,k) \widehat{M_1}(\tau,k) \cdot \left(\widehat{\grad_v \mu}\right)\left(k\left(\frac{1-e^{-\nu (t-\tau)}}{\nu}\right)\right) \dd \tau \right]^2 \dd t. 
\end{align*}
By an easy variant of Lemma \ref{lem:GenCEst0mode}, there holds 
\begin{align}
J_{\cC_\mu} & \lesssim \nu^2 \norm{\abs{\grad}^{1/2} M_T}_{L^2_t \cM^{\sigma,\delta}}^2 + \nu^2 \norm{\abs{\grad}^{1/2} M_1}_{L^2_t \cM^{\sigma,\delta}}^2. 
\end{align}
which, by Lemma \ref{lem:uTctrls}  and \eqref{ineq:bootAM1}, is consistent with Proposition \ref{prop:boot} for $\nu$ small. 

\subsubsection{$J_{\cC_h}$ collision contributions}
Turn next to $J_{\cC_h}$ in \eqref{ineq:Brho}, defined in \eqref{def:BrhoTrms}. 
Recall from \eqref{def:cCh} the definition of $\cC_h$, which induces the natural decomposition of $J_{\cC_h}  \lesssim \sum_{j=1}^{4} J_{\cC_h;j}$. 

Consider the contribution from $J_{\cC_h;1}$,  which is challenging due to the presence of two derivatives: 
\begin{align*}
J_{\cC_h;1} & \lesssim   \int_0^{T^\ast} \sum_{k \in \Integers_\ast^n} \left[B(t,k)\int_0^t \nu S(t-\tau,k) \hat{\rho}(\tau,k) \abs{k}^2 \left(\frac{1-e^{-\nu(t-\tau)}}{\nu}\right)^2 \widehat{h}_0\left(k\left(\frac{1-e^{-\nu (t-\tau)}}{\nu}\right)\right) \dd \tau \right]^2 \dd t \\ 
& \quad + \nu^2 \int_0^{T_\ast} \sum_{k \in \Integers_\ast^n} \left[B(t,k)\int_0^t \nu S(t-\tau,k) \sum_{\ell \in \Integers_\ast^n} \hat{\rho}(\tau,\ell)\abs{k}^2\left(\frac{1-e^{-\nu (t-\tau)}}{\nu} \right)^2 \right. \\ & \left. \quad\quad \times \widehat{f}_{\neq}\left(\tau,k-\ell,k\left(\frac{1-e^{-\nu t}}{\nu}\right)- \ell \left(\frac{1-e^{-\nu \tau}}{\nu}\right)\right) \dd \tau \right]^2 \dd t \\ 
& : =  J_{\cC_h;1,0} + J_{\cC_h;1,\neq}.
\end{align*}
By Lemma \ref{lem:GenCEst0mode}, 
\begin{align*}
J_{\cC_h;1,0} & \lesssim \nu^{4/3}\norm{\rho}_{L^\infty_t (I;\cM^{\beta,\delta})}^2\norm{h_0}_{L^\infty_t H^{\sigma}_m}^2 + \nu^2\norm{\abs{\grad}^{1/2}\rho}_{L^2_t(I;\cM^{\sigma,\delta})}^2 \norm{h_0}_{L^\infty_t H^{n/2+2}_{m}}^2, 
\end{align*}
which is consistent with Proposition \ref{prop:boot} for $\nu$ and $\eps$ sufficiently small by Remark \ref{rmk:HsmInterp} and the bootstrap hypotheses. 
By Lemma \ref{lem:GenCEst} with $\gamma=2$, there holds 
\begin{align*}
J_{\cC_h;1,\neq} & \lesssim \nu^{1/2}\norm{\rho}_{L^\infty_t (I;\cM^{\beta,\delta})}^2 \left(\int_0^T \nu \jap{\tau}^{-p-2}\left(\norm{f(\tau)}_{\cD^{\sigma+1/2,\delta_1}_{m}}^2 + \norm{f(\tau)}_{F^{\sigma+1/2,\delta_1}_m}^2 \right) \dd \tau \right) \\ & \quad + \nu^{2/3}\norm{\abs{\grad}^{1/2} \rho}_{L^2_t(I;\cM^{\sigma,\delta})}^2 \left(\sup_{\tau \in I} \norm{f(\tau)}_{F^{n+2,\delta_1}_{m}}\right)^2, 
\end{align*}
which is consistent with Proposition \ref{prop:boot} for $\nu$ sufficiently small by the bootstrap hypotheses.  
This completes the treatment of $J_{\cC_h;1}$.
The $J_{\cC_h;2}$ and $J_{\cC_h;4}$ terms are treated analogously via Lemmas \ref{lem:GenCEst0mode} and \ref{lem:GenCEst} with $\gamma=1$, as $(M_1)_0$ vanishes by conservation of momentum.

Next, consider $J_{\cC_h;3}$ which we divide via the zero and non-zero mode contributions of $M_T$: 
\begin{align*}
J_{\cC_h;3} & \lesssim \int_0^{T^\ast} \sum_{k \neq 0} \left[B(t,k)\int_0^t \nu S(t-\tau,k) (\widehat{M_T})_0(\tau) \abs{k}^2\left(\frac{1-e^{-\nu (t-\tau)}}{\nu} \right)^2 \right. \\ & \left. \quad\quad \times \widehat{f}\left(\tau,k,k\left(\frac{1-e^{-\nu t}}{\nu}\right)\right) \dd \tau \right]^2 \dd t \\
& \quad + \int_0^{T^\ast} \sum_{k \neq 0} \left[B(t,k)\int_0^t \nu S(t-\tau,k) \widehat{M_T}(\tau,k) \abs{k}^2 \left(\frac{1-e^{-\nu(t-\tau)}}{\nu}\right)^2 \widehat{h}_0\left(k\left(\frac{1-e^{-\nu (t-\tau)}}{\nu}\right)\right) \dd \tau \right]^2 \dd t \\
& \quad + \int_0^{T_\ast} \sum_{k \neq 0} \left[B(t,k)\int_0^t \nu S(t-\tau,k) \sum_{\ell \neq 0, \ell \neq k} \widehat{M_T}(\tau,\ell)\abs{k}^2\left(\frac{1-e^{-\nu (t-\tau)}}{\nu} \right)^2 \right. \\ & \left. \quad\quad \times \widehat{f}\left(\tau,k-\ell,k\left(\frac{1-e^{-\nu t}}{\nu}\right)- \ell \left(\frac{1-e^{-\nu \tau}}{\nu}\right)\right) \dd \tau \right]^2 \dd t \\
& := J_{\cC_h;3,0\neq} + J_{\cC_h;3,0} + J_{\cC_h;3,\neq}. 
\end{align*}
The latter two terms, $J_{\cC_h;3,0}$, $J_{\cC_h;3,\neq}$, are treated as in $J_{\cC_h;1}$ above and are hence omitted for the sake of brevity.  
Hence, it remains only to consider $J_{\cC_h;3,0\neq}$.
Using 
\begin{align*}
\abs{kt} \lesssim e^{\nu t} \abs{k}\frac{1-e^{-\nu t}}{\nu} = e^{\nu \tau} e^{\nu (t-\tau)} \abs{k}\frac{1-e^{-\nu t}}{\nu}, 
\end{align*}
implies by Lemma \ref{lem:propS} and the hypoelliptic smoothing \eqref{ineq:hypoelip2},   
\begin{align*}
J_{\cC_h;3,0\neq} & \lesssim \int_0^{T^\ast} \sum_{k \neq 0} \left[\int_0^t \nu^{3/4} S^{1/2}(t-\tau,k) e^{\delta \nu^{1/3}\tau}\frac{(\widehat{M_T})_0(\tau)}{(t-\tau)^{1/4}} \abs{k}^{3/2}\left(\frac{1-e^{-\nu (t-\tau)}}{\nu} \right)^{3/2} \right. \\ & \left. \quad\quad \times e^{\frac{1}{4}\delta_1 \nu^{1/3} \tau}\abs{k}\widehat{\brak{\grad}^{\sigma}f}\left(\tau,k,k\left(\frac{1-e^{-\nu t}}{\nu}\right)\right) \dd \tau \right]^2 \dd t. 
\end{align*}
Note  
\begin{align*}
\abs{k \frac{1-e^{-\nu(t-\tau)}}{\nu}} = e^{\nu \tau}\abs{k\frac{1-e^{-\nu t}}{\nu} - k \frac{1-e^{-\nu \tau}}{\nu}} \lesssim \brak{\tau} e^{\nu \tau} \brak{k,k\frac{1-e^{-\nu t}}{\nu}}, 
\end{align*}
which implies 
\begin{align*}
J_{\cC_h;3,0\neq} & \lesssim \int_0^{T^\ast} \sum_{k \neq 0} \left[\int_0^t \nu^{3/4} S^{1/2}(t-\tau,k) e^{\delta \nu^{1/3}\tau}\frac{\brak{\tau}^{1/2}(\widehat{M_T})_0(\tau)}{(t-\tau)^{1/4}} \abs{k}\left(\frac{1-e^{-\nu (t-\tau)}}{\nu} \right) \right. \\ & \left. \quad\quad \times e^{\frac{1}{2}\delta_1 \nu^{1/3 }\tau}\abs{k}\widehat{\brak{\grad}^{\sigma+1/2}f}\left(\tau,k,k\left(\frac{1-e^{-\nu t}}{\nu}\right)\right) \dd \tau \right]^2 \dd t \\ 
& \lesssim \nu^{3/2}\left(\sup_{t \in I}\int_0^t e^{\delta \nu^{1/3}\tau}\frac{\brak{\tau}(\widehat{M_T})_0(\tau)}{(t-\tau)^{1/2}} \dd \tau\right) \int_0^{T^\ast} \sum_{k\neq 0} \int_0^t S^{1/2}(t-\tau,k) e^{\delta \nu^{1/3}\tau}\abs{(\widehat{M_T})_0(\tau)} \\ & \quad\quad \times e^{\frac{3}{4}\delta_1 \nu^{1/3 }\tau}\abs{k}\abs{\widehat{\partial_v^t \brak{\grad}^{\sigma+1/2}f}\left(\tau,k,k\left(\frac{1-e^{-\nu t}}{\nu}\right)\right)}^2 \dd t \dd \tau \\ 
&  \lesssim \nu^{3/2} \eps  \int_0^{T^\ast} e^{\delta \nu^{1/3}\tau}\abs{(\widehat{M_T})_0(\tau)} \\ &  \quad\quad \times \int_0^{T^\ast} \sum_{k\neq 0} e^{-\nu(t-\tau)} e^{\frac{3}{4}\delta_1 \nu^{1/3 }\tau}\abs{k}\abs{\widehat{\partial_v^t \brak{\grad}^{\sigma+1/2}f}\left(\tau,k,k\left(\frac{1-e^{-\nu t}}{\nu}\right)\right)}^2 \dd t \dd \tau
\end{align*}
From here we proceed as in the treatment of $\cC_{LH}$ in Lemma \ref{lem:GenCEst} and we deduce,
\begin{align*}
J_{\cC_h;3,0\neq} & \lesssim \nu^{1/2} \eps^2 \left(\int_0^T \nu \jap{\tau}^{-p-2}\left(\norm{f(\tau)}_{\cD^{\sigma+1/2,\delta_1}_{m}}^2 + \norm{f(\tau)}_{F^{\sigma+1/2,\delta_1}_m}^2 \right) \dd \tau \right) \lesssim \nu^{1/2} \eps^4,  
\end{align*}
which is consistent with Proposition \ref{prop:boot} for $\nu$ sufficiently small. This completes the treatment of $J_{\cC_h}$ in \eqref{ineq:Brho} and hence of the improvement of \eqref{ineq:bootArho}. 

\subsection{$L^\infty_t$  control on $\rho$} \label{sec:Linfrho}
This estimate is an easy consequence of \eqref{eq:rhoDef} and the higher norm estimates.
Similar estimates are done in \cite{MouhotVillani11,BMM13} and hence we only sketch the details.
Write
\begin{align*}
B'(t,k) = \brak{k,kt}^{\beta} e^{\delta \nu^{1/3}t}  
\end{align*}
and compute 
\begin{align*}
\norm{B' \rho(t)}_{L^2_k}^2 & \lesssim \cI  + \mathcal{L} + \mathcal{N}_{CL} + J_{\cC_{\mu}} + J_{\cC_{h}},
\end{align*}
where 
%\begin{subequations} 
\begin{align*}
\cI & = \sum_{k \in \Integers_\ast^n} \abs{B(t,k)S(t,k)\widehat{h_{in}} \left( k,k\frac{1-e^{-\nu t}}{\nu} \right) }^2 \\ 
\mathcal{L} & = \sum_{k \in \Integer^n_\ast} \left[B'(t,k) \int_0^t S(t-\tau,k) \hat{\rho}(\tau,k) \widehat{W}(k) \abs{k}^2 \left( \frac{1-e^{-\nu(t-\tau)}}{\nu} \right) \widehat{\mu}\left(k\frac{1-e^{-\nu(t-\tau)}}{\nu} \right) \dd \tau \right]^2 \\ 
\mathcal{N}_{CL} & = \sum_{k \in \Integer_\ast^n} \left[B'(t,k)\int_0^t S(t-\tau,k) \sum_{\ell \neq 0} \hat{\rho}(\tau,\ell) \widehat{W}(\ell) \ell \cdot k\left( \frac{1 - e^{-\nu(t-\tau)}}{\nu}\right)  \right. \\ & \left. \quad\quad \times \widehat{f}\left(\tau,k-\ell, k\left(\frac{1 - e^{-\nu t}}{\nu}\right) - \ell \left(\frac{1-e^{-\nu \tau}}{\nu}\right) \right) \dd \tau \right]^2 \\
J_{\cC_{\mu}} & = \sum_{k \in \Integer_\ast^n} \left[B'(t,k)\int_0^t S(t-\tau,k) \widehat{\cC_{\mu}}\left(\tau,\bar{\eta}(\tau;k,\eta_{CT}(t))\right) \right]^2  \\
J_{\cC_{h}} & = \sum_{k \in \Integer_\ast^n} \left[B(t,k)\int_0^t S(t-\tau,k) \widehat{\cC_{h}}\left(\tau,\bar{\eta}(\tau;k,\eta_{CT}(t))\right) \right]^2. 
\end{align*}
%\end{subequations} 
By Lemma \ref{lem:propS} and Sobolev embedding on the Fourier side: 
\begin{align*}
\cI & \lesssim \sum_{k \in \Integers_\ast^n} S^2(t,k) e^{2\delta \nu^{1/3}t} \brak{k,k \frac{1-e^{-\nu t}}{\nu}}^{2\beta} \abs{\widehat{h_{in}}\left(k,k\left(\frac{1-e^{-\nu t}}{\nu}\right) \right) }^2 \\
& \lesssim \norm{h_{in}}_{H^{\beta}_{m}}^2 \lesssim \eps^2. 
\end{align*}
Consider next the linear collisionless term. 
By Lemma \ref{lem:propS} and Cauchy-Schwarz, 
\begin{align*}
\abs{\mathcal{L}} & \lesssim \sum_{k \in \Integer_\ast^n} \left[\int_0^t \brak{k,kt}^{\beta-\sigma} \abs{B\rho(\tau,k)} \brak{k\frac{1-e^{-\nu(t-\tau)}}{\nu}}^{1+\sigma} \widehat{\mu}\left(k\frac{1-e^{-\nu(t-\tau)}}{\nu} \right) \dd \tau\right]^2 \\ 
& \lesssim \norm{\rho}_{L^2_t \cM^{\sigma,\delta}}^2 \\ 
& \lesssim K_{MH0}^2\eps^2, 
\end{align*}
which is consistent with Proposition \ref{prop:boot} provided we set $K_{ML0} \gg K_{MH0}$. 
The nonlinear acceleration term is treated similarly: 
\begin{align*}
\abs{\mathcal{N}_{CL}} & \lesssim \sum_{k \neq 0} \left[\int_0^t \brak{k,kt}^{-2n-2} S^{1/2}(t-\tau,k) \sum_{\ell \neq 0} \brak{\ell,\ell \tau}^{\sigma} e^{\delta \nu^{1/3}\tau} \abs{\hat{\rho}(\tau,\ell) \widehat{W}(\ell) \ell \cdot k\left( \frac{1 - e^{-\nu(t-\tau)}}{\nu}\right)}  \right. \\ & \left. \quad\quad \times \abs{\widehat{f}\left(\tau,k-\ell, k\left(\frac{1 - e^{-\nu t}}{\nu}\right) - \ell \left(\frac{1-e^{-\nu \tau}}{\nu}\right) \right)} \dd \tau \right]^2 \\
& \quad + \sum_{k \neq 0} \left[\int_0^t \brak{k,kt}^{-2n-2} S^{1/2}(t-\tau,k) \sum_{\ell \neq 0} e^{\delta \nu^{1/3}\tau} \abs{\hat{\rho}(\tau,\ell) \widehat{W}(\ell) \ell \cdot k\left( \frac{1 - e^{-\nu(t-\tau)}}{\nu}\right)}  \right. \\ & \left. \quad\quad \times  \abs{\widehat{\brak{\grad}^{\sigma+1/2-6}f}\left(\tau,k-\ell, k\left(\frac{1 - e^{-\nu t}}{\nu}\right) - \ell \left(\frac{1-e^{-\nu \tau}}{\nu}\right) \right)} \dd \tau \right]^2. 
\end{align*}
We take the contributions from $f$ in $L^\infty$ in time and frequency (using the moment controls and $\delta_1 > 0$ and $\nu$ sufficiently small) and $\rho$ in $L^2_{t,k}$: 
\begin{align*}
\abs{\mathcal{N}_{CL}} & \lesssim \norm{\rho}_{L^2_t \cM^{\sigma,\delta}}^2 \norm{f}_{F^{\sigma+1/2-6,\delta_1}_{m+1}}^2 \lesssim \eps^4. 
\end{align*}
From here it is apparent that with a sufficient regularity gap between $\beta$ and $\sigma$ and sufficiently high moment controls on $f$, 
the collisional terms can all be similarly treated. Hence, we omit the details for the sake of brevity. 

\section{Higher order moment estimates} \label{sec:highMoments}
In this section, we improve \eqref{ineq:bootAM1}, \eqref{ineq:bootAM2}, \eqref{ineq:bootLinftAM1}, and \eqref{ineq:bootLinftAM2}. 
We will only sketch the details for the much harder $L^2_t \cM^{\sigma,\delta}$ control. 
After techniques are developed for this estimate, the $L^\infty_t$ estimates are a straightforward  adaptation of \S\ref{sec:Linfrho} and hence are omitted for the sake of brevity. 

\subsection{Estimate on $M_1$, the first moment} \label{sec:M1}
As for $\rho$, by the conservation law \eqref{eq:M1ctrl}, the $x$-average of the $M_1$ vanishes. 
We begin by deriving the analogue of \eqref{eq:rhoDef} for $M_1$, beginning with
\begin{align}
\partial_t \grad_\eta \hat{f} + \grad_\eta \cL  = -\nu \bar{\eta}^2 \grad_\eta \hat{f} -2 \nu e^{\nu t} \bar{\eta}  \hat{f} + \grad_\eta \mathcal{N},  \label{eq:fhatM1}
\end{align}
Duhamel's formula gives (applying also \eqref{eq:hatfDuhamel} on the linear term $\nu e^{\nu t}\bar{\eta} \hat{f}$),  
\begin{align}
\grad_\eta \hat{f} & = S(t,0;k,\eta)\grad_\eta \hat{f}_{in}(k,\eta)  - \int_0^t S(t,\tau;k,\eta) (\grad_\eta \mathcal{L})(\tau,k,\bar{\eta}(\tau;k,\eta)) \dd \tau \nonumber \\
& \quad - \int_0^t S(t,\tau;k,\eta) (\grad_\eta \mathcal{N})(\tau,k,\bar{\eta}(\tau;k,\eta)) \dd \tau \nonumber \\
& \quad - 2\nu \int_0^t S(t,\tau;k,\eta) e^{\nu \tau} \bar{\eta}(\tau;k,\eta) S(\tau,0;k,\eta) \widehat{f_{in}}(k,\eta) \dd \tau \nonumber \\ 
& \quad - 2\nu \int_0^t S(t,\tau;k,\eta) e^{\nu \tau} \bar{\eta}(\tau;k,\eta) \int_0^\tau S(\tau,\tau';k,\eta) \mathcal{L}(\tau',k,\bar{\eta}(\tau';k,\eta)) \dd \tau' \dd \tau \nonumber \\
& \quad - 2\nu \int_0^t S(t,\tau;k,\eta) e^{\nu \tau} \bar{\eta}(\tau;k,\eta) \int_0^\tau S(\tau,\tau';k,\eta) \mathcal{N}(\tau',k,\bar{\eta}(\tau';k,\eta)) \dd \tau' \dd \tau. \label{eq:M1f}
\end{align}
Evaluating at $\eta = \eta_{CT}(t,k)$ gives the expression for $M_1$ (recall the definition \eqref{def:S}),  
\begin{align}
\widehat{M_1}(t,k)  & = e^{-\nu t} S(t,k)\grad_\eta \hat{f}_{in}(k,\eta_{CT}(t,k))  - \int_0^t S(t-\tau,k) e^{-\nu t} (\grad_\eta \mathcal{L})(\tau,k,\bar{\eta}(\tau;k,\eta_{CT}(t,k))) \dd \tau \nonumber \\
& \quad - \int_0^t S(t-\tau,k) e^{-\nu t} (\grad_\eta \mathcal{N})(\tau,k,\bar{\eta}(\tau;k,\eta_{CT}(t,k))) \dd \tau \nonumber \\
& \quad - 2\nu \int_0^t S(t-\tau,k) e^{\nu (\tau-t)} \bar{\eta}(\tau;k,\eta) S(\tau,0;k,\eta_{CT}(t,k)) \widehat{f_{in}}(k,\eta_{CT}(t,k)) \dd \tau \nonumber \\ 
& \quad - 2\nu \int_0^t S(t-\tau,k) e^{\nu (\tau-t)} \bar{\eta}(\tau;k,\eta) \int_0^\tau S(\tau,\tau';k,\eta_{CT}(t,k)) \mathcal{L}(\tau',k,\bar{\eta}(\tau';k,\eta_{CT}(t,k))) \dd \tau' \dd \tau \nonumber \\
& \quad - 2\nu \int_0^t S(t-\tau;k) e^{\nu (\tau-t)} \bar{\eta}(\tau;k,\eta) \int_0^\tau S(\tau,\tau';k,\eta_{CT}(t,k)) \mathcal{N}(\tau',k,\bar{\eta}(\tau';k,\eta_{CT}(t,k))) \dd \tau' \dd \tau \nonumber \\
& \quad + \mathcal{I}_1 + \mathcal{L}_{1} + \mathcal{N}_1 + \mathcal{I}_{0}  + \mathcal{L}_0 + \mathcal{N}_{0}. \label{eq:M1}
\end{align}
As in \S\ref{sec:density}, we apply $B$ to both sides of \eqref{eq:M1} and take $L^2$ norms in $t$ and $k$. 

\subsubsection{Treatment of $\mathcal{I}_1$, $\mathcal{L}_1$, and $\mathcal{N}_1$}
These terms can be treated by a small variation of arguments used to treat $\rho$ in \S\ref{sec:density}. 
The $\cI_1$ initial data term is treated in the same manner as in \eqref{ineq:Bhin}; we omit the details: 
\begin{align}
\norm{\abs{\grad}^{1/2}\mathcal{I}}_{L^2_t \cM^{\sigma,\delta}} \lesssim \eps. 
\end{align}

\subsubsection{The $\mathcal{L}_1$ term} \label{sec:M1L1}
Computing $\grad_{\eta}\cL$ gives
\begin{align*}
e^{-\nu t}\grad_\eta \cL(\tau,k,\bar{\eta}(\tau;k,\eta_{CT}(t,k))) & = e^{-\nu(t-\tau)} \hat{\rho}(\tau,k) \widehat{W}(k) k \widehat{\mu}\left(k\frac{1-e^{\nu(t-\tau)}}{\nu}\right) \\ & \quad + e^{-\nu(t-\tau)} \hat{\rho}(\tau,k) \hat{W}(k) \abs{k}^2 \left( \frac{1-e^{-\nu(t-\tau)}}{\nu}\right) \grad_\eta \widehat{\mu}\left(k\frac{1-e^{-\nu(t-\tau)}}{\nu}\right). 
\end{align*}
By Lemma \ref{lem:GenCEst0mode}, it follows that 
\begin{align*}
\norm{\abs{\grad}^{1/2}\mathcal{L}_1}_{L^2_t \cM^{\sigma,\delta}} \lesssim \norm{\rho}_{L^2_t \cM^{\sigma,\delta}}, 
\end{align*}
which is consistent with Proposition \ref{prop:boot} provided we choose $K_{MH1} \gg K_{MH0}$ (see Remark \ref{rmk:bootconst}). 

\subsubsection{The $\mathcal{N}_1$ term}
First, write 
\begin{align*}
\grad_\eta \mathcal{N}(\tau,k,\bar{\eta}(\tau,k,\eta_{CT}(t,k)) & = (\grad_\eta \cN_{CL})(\tau,k,\bar{\eta}(\tau;k,\eta_{CT})) \\ 
& \quad + \nu \left(\grad_{\eta}\widehat{\cC_{\mu}}\right)(\tau,k,\bar{\eta}(\tau;k,\eta_{CT})) + \nu (\grad_\eta\widehat{\cC_h})(\tau,k,\bar{\eta}(\tau;k,\eta_{CT})) \\ 
& = \cN_{CL} + \nu \cJ_\mu + \nu \cJ_h. 
\end{align*}
The nonlinear collisionless term is given by
\begin{align*}
e^{-\nu t}(\grad_\eta \cN_{CL})(\tau,k,\bar{\eta}(\tau;k,\eta_{CT})) &= \widehat{\rho}(\tau,k) \widehat{W}(k) \left(k e^{-\nu(t- \tau)} \widehat{h_0}\left(\tau,k\frac{1-e^{-\nu(t-\tau)}}{\nu}\right)\right) \\ 
& \quad + \widehat{\rho}(\tau,k) \widehat{W}(k) \left(e^{-\nu(t- \tau)} \abs{k}^2 \frac{1-e^{-\nu(t-\tau)}}{\nu} \grad_\eta\widehat{h_0}\left(\tau,k\frac{1-e^{-\nu(t-\tau)}}{\nu}\right)\right) \\
& \quad + \sum_{k \neq \ell} \widehat{\rho}(\tau,\ell) \widehat{W}(\ell) \ell e^{-\nu(t-\tau)} \widehat{f}\left(\tau,k-\ell,k\frac{1-e^{-\nu t}}{\nu} - \ell \frac{1-e^{-\nu \tau}}{\nu}\right) \\
& \quad + \sum_{k \neq \ell} \widehat{\rho}(\tau,\ell) \widehat{W}(\ell) \ell \cdot k \frac{1-e^{-\nu(t-\tau)}}{\nu} e^{-\nu (t-\tau)} \\ & \quad\quad \times e^{-\nu \tau} \grad_\eta \widehat{f}\left(\tau,k-\ell,k\frac{1-e^{-\nu t}}{\nu} - \ell \frac{1-e^{-\nu \tau}}{\nu}\right) \\ 
& = \sum_{j=1}^4 \mathcal{N}_{CL;j}.
\end{align*}
By Lemma \ref{lem:GenCEst0mode}, there holds 
\begin{align*}
\norm{\int_0^t S(t-\tau,k)\mathcal{N}_{CL;1}(\tau,k) \dd \tau }_{L^2_t \cM^{\sigma,\delta}} & \lesssim \norm{\rho}_{L^\infty_t \cM^{\beta,\delta}}\left(\sup_{\tau \in I} \brak{\tau}^{-\beta/2} \norm{h_0(\tau)}_{H^{\sigma}_m}\right) \\ & \quad + \norm{\abs{\grad}^{1/2}\rho}_{L^2_t \cM^{\sigma,\delta}} \norm{h_0}_{L^\infty_t H^{n/2+1+}_{m}}. 
\end{align*}
By \eqref{ineq:h0Hsig} and the bootstrap hypotheses, this is consistent with Proposition \ref{prop:boot} for $\eps$ sufficiently small. 
The $\mathcal{N}_{CL;2}$ term is treated similarly by Lemma \ref{lem:GenCEst0mode},
\begin{align*}
\norm{\int_0^t S(t-\tau,k)\mathcal{N}_{CL;2}(\tau,k) \dd \tau   }_{L^2_t \cM^{\sigma,\delta}} & \lesssim \nu^{-1/3}\norm{\rho}_{L^\infty_t \cM^{\beta,\delta}}\left(\sup_{\tau \in I} \brak{\tau}^{-\beta/2}\norm{h_0(\tau)}_{H^{\sigma}_m}\right) \\ & \quad + \norm{\abs{\grad}^{1/2}\rho}_{L^2_t \cM^{\sigma,\delta}} \norm{h_0}_{L^\infty_t H^{n/2+1+}_{m}}. 
\end{align*}
By \eqref{ineq:h0Hsig}, this is consistent with Proposition \ref{prop:boot} for $\eps \nu^{-1/3}$ sufficiently small. 
Similarly, $\mathcal{N}_{CL;3}$ and $\mathcal{N}_{CL;4}$ are both treated via Lemma \ref{lem:GenCEst} as in \S\ref{sec:rhoHi} without additional complications. 

As in the collisionless terms, the treatment of the collisional contributions is similar to the treatment of collisions in \S\ref{sec:rhoHi}. 
Indeed, direct calculation gives
\begin{align*}
\nu e^{-\nu t} \grad_\eta J_{\cC_\mu} & = \nu \widehat{M_T}(\tau,k) e^{-\nu(t-\tau)} \grad_\eta\widehat{\Delta_v\mu}\left(k\frac{1-e^{-\nu(t-\tau)}}{\nu}\right) + \nu \widehat{M_1}(\tau,k) e^{-\nu(t- \tau)}\grad_\eta \widehat{\grad_v \cdot (\mu v)} \left(k\frac{1-e^{-\nu(t-\tau)}}{\nu}\right). 
\end{align*}
It is apparent that the contributions are estimated via the same techniques as applied in \S\ref{sec:rhoHi}; the details are omitted for brevity.
Similarly, the contributions of the nonlinear collision term $\cJ_{h}$ are not significantly harder than the analogous estimates in \S\ref{sec:rhoHi}. 

\subsubsection{Treatment $\mathcal{I}_0$, $\mathcal{L}_0$, and $\mathcal{N}_0$} \label{sec:I0M0N0}
First, observe that since $S(t,k)$ satisfies the semigroup property (see \eqref{def:S}), 
\begin{align}
\cI_0 + \cL_0 + \cN_0 & =  - 2\nu S(t,k) \left(\int_0^t e^{\nu (\tau-t)} \bar{\eta}(\tau;k,\eta_{CT}(t,k)) \dd \tau\right) \widehat{f_{in}}(k,\eta_{CT}(t,k)) \nonumber \\ 
& \hspace{-3cm}- 2\nu \int_0^t S(t-\tau',k) \left(\int_{\tau'}^t e^{\nu (\tau-t)} \bar{\eta}(\tau;k,\eta_{CT}(t,k)) \dd \tau\right) \mathcal{L}(\tau',k,\bar{\eta}(\tau';k,\eta_{CT}(t,k))) \dd \tau' \nonumber \\
&\hspace{-3cm}  - 2\nu \int_0^t S(t-\tau';k) \left(\int_{\tau'}^t e^{\nu (\tau-t)} \bar{\eta}(\tau;k,\eta_{CT}(t,k)) \dd \tau\right) \mathcal{N}(\tau',k,\bar{\eta}(\tau';k,\eta_{CT}(t,k))) \dd \tau'.  \label{eq:FubiniM1}
\end{align}
Hence, we are led to consider the quantity 
\begin{align}
\int_{\tau'}^t e^{\nu (\tau-t)}\bar{\eta}(\tau;k,\eta_{CT}(t,k)) \dd \tau & = \int_{\tau'}^t e^{\nu (\tau-t)} k \frac{1-e^{-\nu (t-\tau)}}{\nu} \dd \tau %\nonumber \\ 
 = \frac{k}{2\nu^2}\left(e^{-\nu (t-\tau')} - 1\right)^2.\label{eq:etabarM1Fubini}
\end{align}
It follows that the terms in \eqref{eq:FubiniM1} can be treated as in \S\ref{sec:M1L1} and \S\ref{sec:rhoHi} after noting the following: 
\begin{lemma} \label{lem:propSMoment}
There holds for any fixed $p$, 
\begin{align*}
S^{p}(t-\tau',k) \frac{\abs{k}}{2\nu^2}\left(e^{-\nu (t-\tau')} - 1\right)^2 \lesssim \nu^{-2/3}.
\end{align*}
\end{lemma} 
\begin{proof} 
From the elementary inequality $\abs{e^{x}-1} \leq xe^x$ for $x > 0$ and Lemma \ref{lem:propS}, 
\begin{align*}
S^{p}(t-\tau',k) \frac{\abs{k}}{2\nu^2}\left(e^{-\nu (t-\tau')} - 1\right)^2 & \lesssim S^{p}(t-\tau',k) \frac{\abs{k}}{2\nu^2}(\nu(t-\tau'))^2 e^{2\nu(t-\tau')} \\
& \lesssim \frac{\abs{k}(t-\tau')^2}{(\nu \abs{k}^2 (t-\tau')^3)^{2/3}} \lesssim \abs{k}^{-1/3}\nu^{-2/3}.
\end{align*}
\end{proof} 
The above arguments complete the improvement of \eqref{ineq:bootAM1} for $\eps \nu^{-1/3} \ll 1$. 
The improvement of \eqref{ineq:bootLinftAM1} follows by the above observations together with the techniques of \S\ref{sec:Linfrho}. %The details are omitted for brevity.  

\subsection{Second moment}
By \eqref{eq:M2ctrl}, the zero mode in $x$ of $M_2$ is automatically controlled by the electric field, and hence it suffices to consider 
only the $x$-dependent modes. 
Begin by taking another derivative of \eqref{eq:fhatM1}: 
\begin{align}
\partial_t \Delta_\eta \hat{f} +\Delta_\eta \cL & = -\nu \bar{\eta}^2 \Delta_\eta \hat{f} - 2\nu e^{\nu t} \bar{\eta} \cdot \grad_\eta \hat{f} - 2\nu e^{2\nu t} \hat{f}  + \Delta_\eta \mathcal{N},  \label{eq:fhatM2}
\end{align}
and applying Duhamel's formula (together with \eqref{eq:hatfDuhamel} and  \eqref{eq:M1f}): 
\begin{align*}
\Delta_\eta \hat{f} & = S(t,0;k,\eta)\Delta_\eta \hat{f}_{in}(k,\eta)  - \int_0^t S(t,\tau;k,\eta) (\Delta_\eta \mathcal{L})(\tau,k,\bar{\eta}(\tau;k,\eta)) \dd \tau \nonumber \\
& \quad - \int_0^t S(t,\tau;k,\eta) (\Delta_\eta \mathcal{N})(\tau,k,\bar{\eta}(\tau;k,\eta)) \dd \tau \nonumber \\
& \quad - \nu \int_0^t S(t,\tau;k,\eta) e^{2\nu \tau}  S(\tau,0;k,\eta) \widehat{f_{in}}(k,\eta) \dd \tau \nonumber \\ 
& \quad - 2\nu \int_0^t S(t,\tau;k,\eta) e^{2\nu \tau} \int_0^\tau S(\tau,\tau';k,\eta) \mathcal{L}(\tau',k,\bar{\eta}(\tau';k,\eta)) \dd \tau' \dd \tau \nonumber \\
& \quad - 2\nu \int_0^t S(t,\tau;k,\eta) e^{2\nu \tau} \int_0^\tau S(\tau,\tau';k,\eta) \mathcal{N}(\tau',k,\bar{\eta}(\tau';k,\eta)) \dd \tau' \dd \tau \nonumber \\ 
& \quad - 2\nu \int_0^t S(t,\tau;k,\eta) e^{\nu \tau} \bar{\eta}(\tau;k,\eta)  S(t,0;k,\eta)\grad_\eta \hat{f}_{in}(k,\eta) \dd \tau \nonumber \\ 
& \quad - 2\nu \int_0^t S(t,\tau;k,\eta) e^{\nu \tau} \bar{\eta}(\tau;k,\eta) \int_0^\tau S(\tau,\tau';k,\eta) (\grad_\eta \mathcal{L})(\tau',k,\bar{\eta}(\tau';k,\eta)) \dd \tau' \dd \tau \nonumber \\
& \quad - 2\nu \int_0^t S(t,\tau;k,\eta) e^{\nu \tau} \bar{\eta}(\tau;k,\eta) \int_0^\tau S(\tau,\tau';k,\eta) (\grad_\eta \mathcal{N})(\tau',k,\bar{\eta}(\tau';k,\eta)) \dd \tau' \dd \tau \nonumber \\
& \quad - 4\nu \int_0^t S(t,\tau;k,\eta) e^{\nu \tau} \bar{\eta}(\tau;k,\eta)  \nu \int_0^{\tau} S(\tau,\tau';k,\eta) e^{\nu \tau'} \bar{\eta}(\tau';k,\eta)\widehat{f_{in}}(k,\eta) \dd \tau' \dd \tau \nonumber \\ 
& \quad - 4\nu \int_0^t S(t,\tau;k,\eta) e^{\nu \tau} \bar{\eta}(\tau;k,\eta)  \nu \int_0^{\tau} S(\tau,\tau';k,\eta) e^{\nu \tau'} \bar{\eta}(\tau';k,\eta) \nonumber \\ & \quad\quad \times \int_0^{\tau'} S(\tau',\tau'';k,\eta) \mathcal{L}(\tau'',k,\bar{\eta}(\tau'';k,\eta)) \dd \tau'' \dd \tau' \dd \tau \nonumber  \\ 
& \quad - 4\nu \int_0^t S(t,\tau;k,\eta) e^{\nu \tau} \bar{\eta}(\tau;k,\eta)  \nu \int_0^{\tau} S(\tau,\tau';k,\eta) e^{\nu \tau'} \bar{\eta}(\tau';k,\eta) \nonumber  \\ & \quad\quad \times \int_0^{\tau'} S(\tau',\tau'';k,\eta) \mathcal{N}(\tau'',k,\bar{\eta}(\tau'';k,\eta)) \dd \tau'' \dd \tau' \dd \tau. 
\end{align*}
Restricting to $\eta = \eta_{CT}(t,k)$ and applying Fubini as in \eqref{eq:FubiniM1} gives the expression for $M_2$: 
\begin{align}
\widehat{M_2}(t,k)& = e^{-2\nu t}S(t,k)\Delta_\eta \hat{f}_{in}(k,\eta_{CT})  - \int_0^t e^{-2\nu t} S(t-\tau,k) (\Delta_\eta \mathcal{L})(\tau,k,\bar{\eta}(\tau;k,\eta_{CT})) \dd \tau \nonumber \\
& \quad - \int_0^t e^{-2\nu t} S(t-\tau,k) (\Delta_\eta \mathcal{N})(\tau,k,\bar{\eta}(\tau;k,\eta_{CT})) \dd \tau \nonumber \\
& \quad - \nu \int_0^t S(t,k) e^{-2\nu(t-\tau)}  \widehat{f_{in}}(k,\eta_{CT}) \dd \tau \nonumber \\ 
& \quad - \nu \int_0^t S(t-\tau',k) \left(\int_{\tau'}^t  e^{-2\nu(t-\tau)} \dd \tau\right) \mathcal{L}(\tau',k,\bar{\eta}(\tau';k,\eta_{CT})) \dd \tau' \nonumber \\
& \quad - \nu \int_0^t S(t-\tau';k) \left(\int_{\tau'}^t e^{-2\nu(t-\tau)} \dd \tau \right)  \mathcal{N}(\tau',k,\bar{\eta}(\tau';k,\eta_{CT})) \dd \tau' \nonumber \\ 
& \quad - \nu e^{-\nu t}S(t,k) \left(\int_0^t e^{-\nu (t-\tau)} \bar{\eta}(\tau;k,\eta_{CT}) \dd \tau\right)  \cdot \grad_\eta \hat{f}_{in}(k,\eta_{CT})\nonumber \\ 
& \quad - \nu e^{-\nu t} \int_0^t S(t-\tau';k) \left(\int_{\tau'}^t e^{-\nu(t- \tau)} \bar{\eta}(\tau;k,\eta_{CT})\dd \tau \right)  (\grad_\eta \mathcal{L})(\tau',k,\bar{\eta}(\tau';k,\eta_{CT})) \dd \tau' \nonumber \\
& \quad - \nu e^{-\nu t} \int_0^t  S(t-\tau',k) \left(\int_{\tau'}^t e^{-\nu(t-\tau)} \bar{\eta}(\tau;k,\eta_{CT}) \dd \tau \right) (\grad_\eta \mathcal{N})(\tau',k,\bar{\eta}(\tau';k,\eta_{CT})) \dd \tau' \nonumber \\
& \quad - 4\nu^2 S(t,k) \left(\int_{\tau''}^t \int_{\tau''}^{\tau} e^{\nu \tau} \bar{\eta}(\tau;k,\eta_{CT}) e^{\nu \tau'} \bar{\eta}(\tau';k,\eta_{CT})\dd \tau' \dd \tau \right) \int_0^t\widehat{f_{in}}(k,\eta_{CT}) \nonumber \\ 
& \quad - 4\nu^2 e^{-2\nu t}  \int_0^t S(t-\tau'',k) \left(\int_{\tau''}^t \int_{\tau''}^{\tau} e^{\nu \tau} \bar{\eta}(\tau;k,\eta_{CT}) e^{\nu \tau'} \bar{\eta}(\tau';k,\eta_{CT})\dd \tau' \dd \tau \right) \nonumber \\ & \quad\quad \times \mathcal{L}(\tau'',k,\bar{\eta}(\tau'';k,\eta_{CT})) \dd \tau'' \nonumber  \\  
& \quad - 4\nu^2e^{-2\nu t} \int_0^t S(t-\tau'',k)\left(\int_{\tau''}^t \int_{\tau''}^{\tau} e^{\nu \tau} \bar{\eta}(\tau;k,\eta_{CT}) e^{\nu \tau'} \bar{\eta}(\tau';k,\eta_{CT})\dd \tau' \dd \tau \right) \nonumber \\ & \quad\quad \times \mathcal{N}(\tau'',k,\bar{\eta}(\tau'';k,\eta_{CT})) \dd \tau''. \label{eq:M2Duhamel}
\end{align}
For the fifth and sixth integrals, notice that $\int_{\tau'}^t e^{-\nu (t-\tau)} d\tau = (1 - e^{-\nu(t-\tau')})/\nu \leq \nu^{-1}$. Hence, we may treat these integrals using the techniques in \S\ref{sec:rhoHi}. 
The next few integrals in \eqref{eq:M2Duhamel} involve \eqref{eq:etabarM1Fubini}, and hence these contributions are similarly treated after applying Lemma \ref{lem:propSMoment}. 
The quantity arising in the latter three integrals of \eqref{eq:M2Duhamel} is
\begin{align*}
e^{-2\nu t}\int_{\tau''}^t  \int_{\tau''}^\tau e^{\nu \tau} \bar{\eta}(\tau;k,\eta_{CT}) e^{\nu \tau'} \bar{\eta}(\tau';k,\eta_{CT}) \dd \tau' \dd \tau & 
\\ & \hspace{-5cm} = e^{-2\nu t}\int_{\tau''}^t  \int_{\tau''}^\tau e^{\nu \tau} e^{\nu \tau'}\abs{k}^2 \left(\frac{1-e^{-\nu(t-\tau')}}{\nu}\right)\left(\frac{1-e^{-\nu(t-\tau)}}{\nu}\right)  \dd \tau' \dd \tau \\ 
& \hspace{-5cm} =  e^{-2\nu t}\int_{\tau''}^t \abs{k}^2 e^{\nu \tau }\left(\frac{1-e^{-\nu(t-\tau)}}{\nu}\right)\left(\frac{e^{\nu \tau} - e^{\nu \tau''}}{\nu^2} - \frac{e^{-\nu t + 2\nu\tau} - e^{-\nu t + 2\nu \tau''}}{2\nu^2}\right) \dd \tau  \\ 
& \hspace{-5cm} = \frac{\abs{k}^2 e^{-4 \nu t}}{8\nu^4}\left(e^{\nu \tau''} - e^{\nu t}\right)^4. 
\end{align*}
As in \S\ref{sec:I0M0N0}, we may treat these integrals using the techniques in \S\ref{sec:rhoHi} given the following: 
\begin{lemma} 
There holds for any fixed $p \in (0,1)$, 
\begin{align}
S^p(t-\tau'',k) \frac{\abs{k}^2 e^{-4 \nu t}}{8\nu^4}\left(e^{\nu \tau''} - e^{\nu t}\right)^4 \lesssim \nu^{-1}. 
\end{align}
\end{lemma} 
\begin{proof} 
The lemma follows from Lemma \ref{lem:propS} and the following observations: 
\begin{align*}
\frac{e^{-4 \nu t}}{8\nu^4}\left(e^{\nu \tau''} - e^{\nu t}\right)^4  = \frac{1}{8\nu^4}\left(e^{\nu (\tau''- t)} - 1\right)^4 \leq \frac{1}{8\nu^4}\left(\nu (\tau''-t)\right)^4 e^{4 \nu (t-\tau'')}. 
\end{align*}
\end{proof} 
Along with Lemma \ref{lem:propSMoment}, we may improve \eqref{ineq:bootAM2} and \eqref{ineq:bootLinftAM2} by properly adjusting the bootstrap constants and choosing $\eps \nu^{-1/3} \ll 1$.
We omit the details for the sake of brevity. 

\section{Estimates on the distribution function} 
\subsection{Preliminary lemmas}
The following versions of Young's inequality occur frequently in the proof. These are easy variants of lemmas found in \cite{BMM13}; the proofs are omitted for brevity. 
\begin{lemma} \label{lem:YoungDisc}
\begin{itemize} 
\item[(a)] Let $f(k,\eta),g(k,\eta) \in L^2(\Integer^d \times \Real^d)$ and $\jap{k}^\sigma h(t,k) \in L^2(\Integer^d)$ for $\sigma > n/2$. Then, for any $t \in \Real$, 
\begin{align} 
\abs{\sum_{k,\ell \in \Integers_\ast^n} \int  f(k,\eta) h(t,\ell) g(k-\ell,\eta-\ell\frac{1-e^{-\nu t}}{\nu}) \, \dd \eta} \lesssim_{d,\sigma} \norm{f}_{L^2_{k,\eta} }\norm{g}_{L^2_{k,\eta} } \norm{\jap{k}^\sigma h(t)}_{L_k^2}. \label{ineq:L2L2L1}      
\end{align} 
\item[(b)] Let $f(k,\eta), \jap{k}^\sigma g(k,\eta) \in L^2(\Integer^d \times \Real^d)$ and $h(t,k) \in L^2(\Integer^d)$ for $\sigma > n/2$. Then, for any $t \in \Real$, 
\begin{align} 
\abs{\sum_{k,\ell \in \Integers_\ast^n} \int  f(k,\eta) h(t,\ell) g(k-\ell,\eta-\ell \frac{1-e^{-\nu t}}{\nu}) \, \dd \eta} \lesssim_{n,\sigma} \norm{f}_{L^2_{k,\eta}}\norm{\jap{k}^{\sigma}g }_{L^2_{k,\eta}} \norm{h(t)}_{L_k^2}.    \label{ineq:L2L1L2}      
\end{align} 
\end{itemize}
\end{lemma} 

\subsection{High norm estimate}
In this section we prove that the constant in \eqref{boot:hineq} can be improved. Hence, fix $j$, $0 \leq j \leq 3$. 
Let $\alpha$ be an arbitrary multi-index such that $\abs{\alpha} \leq m+1+j$ and set $\gamma = \sigma+1/2-j$. 
Write, 
\begin{align}
A(t,k,\eta) & = \brak{t}^{-p/2} A_{\sigma+1/2-j,\delta_1}\mathbf{1}_{k \neq 0} +  \brak{t}^{-p/2} A_{\sigma+1/2-j,0}\mathbf{1}_{k = 0}. 
\end{align}
From the definition of $A$ (see \eqref{def:Asc}) and \eqref{eq:fhat}, we have the following 
\begin{align}
\frac{1}{2}\frac{d}{dt}\left(e^{-2\abs{\alpha}\nu t} \norm{A (v^\alpha f)}_{2}^2\right) & = \delta_1 \nu^{1/3} e^{-2\abs{\alpha}\nu t} \norm{A(v^\alpha f_{\neq})}_2^2  -p \frac{t}{\brak{t}^2} e^{-2\abs{\alpha}\nu t} \norm{A (v^\alpha f)}_{2}^2 \nonumber  \\ & \quad - e^{-2\abs{\alpha}\nu t} \norm{\sqrt{-\partial_t M M} A (v^\alpha f_{\neq})}_{2}^2 \nonumber \\
& \quad - \nu \abs{\alpha} e^{-2\abs{\alpha}\nu t} \norm{A (v^\alpha f)}_{2}^2 + e^{-2\abs{\alpha} \nu t} \brak{ A (v^\alpha f), A (v^\alpha \partial_t f)}_2\nonumber  \\
& := \delta_1 \nu^{1/3} e^{-2\abs{\alpha}\nu t} \norm{A(v^\alpha f)}_2^2 + \cD_{p} + \cD_M + \cD_{\alpha} + \nu e^{-2\abs{\alpha}\nu t} \brak{ A (v^\alpha f),  A (v^\alpha \partial_{vv}^t f)}_2 \nonumber \\
 & \quad + e^{-2\abs{\alpha}\nu t} \brak{ A (v^\alpha f), A (v^\alpha E\cdot \partial_v^t \mu)}_2 + e^{-2\abs{\alpha}\nu t} \brak{ A (v^\alpha f), A (v^\alpha E\cdot \partial_v^t f)}_2 \nonumber \\ 
& \quad  + \nu e^{-2\abs{\alpha}\nu t} \brak{ A (v^\alpha f), A (v^\alpha C_{\mu})}_2 + \nu e^{-2\abs{\alpha}\nu t} \brak{ A (v^\alpha f), A (v^\alpha C_{h})}_2 \nonumber \\
& := \delta_1 \nu^{1/3} e^{-2\abs{\alpha}\nu t} \norm{A(v^\alpha f)}_2^2 + \cD_{p} + \cD_M + \cD_\alpha + \nu \mathcal{D} + \mathcal{L} + \mathcal{N} + \nu J_\mu + \nu J_h. \label{def:hifEvo}
\end{align}

\subsubsection{Dissipation term} \label{sec:HiDiss}
Here we consider the $\mathcal{D}$ contribution in \eqref{def:hifEvo}. 
First, we expand into several contributions, 
\begin{align*}
\nu \mathcal{D} & = -\nu e^{-2\abs{\alpha}\nu t} \norm{\partial_v^t A(v^\alpha f)}_2^2 \\ & \quad + \nu e^{-2\abs{\alpha}\nu t} \sum_{\beta \leq \alpha: \abs{\beta} = \abs{\alpha} - 1}\sum_{k \in \Integer^n} \int A D^\alpha_\eta \overline{\hat{f}}(k,\eta) A(k,\eta) D^{\alpha-\beta}_\eta \left(\eta e^{\nu t} - k \frac{e^{-\nu t} - 1}{\nu}\right)^2 D^{\beta}_\eta\hat{f}(k,\eta) \dd \eta \\ 
&  \quad + \nu e^{-2\abs{\alpha}\nu t} \sum_{\beta \leq \alpha: \abs{\beta} = \abs{\alpha}-2}\sum_{k \in \Integer^n} \int A D^\alpha_\eta \overline{\hat{f}}(k,\eta) A(k,\eta) D^{\alpha-\beta}_\eta \left(\eta e^{\nu t} - k \frac{e^{-\nu t} - 1}{\nu}\right)^2 D^{\beta}_\eta \hat{f}(k,\eta) \dd \eta \\ 
& = -\nu\norm{\partial_v^t A(v^\alpha f_{\neq})}_2^2  + \sum_{j=1}^2 \mathcal{D}_j.
\end{align*}
By \eqref{ineq:dotMMAbd}, there holds for all $\delta_1 > 0$ sufficiently small
\begin{align*}
-\frac{\nu}{2} e^{-2\abs{\alpha}\nu t} \norm{\partial_v^t A(v^\alpha f_{\neq})}_2^2 - \frac{1}{2}\cD_M \leq -\delta_1 \nu^{1/3}\norm{A(v^\alpha f_{\neq})}_{2}^2. 
\end{align*}
This cancels the $+\delta_1 \nu^{1/3} e^{-2\abs{\alpha}\nu t} \norm{A(v^\alpha f_{\neq})}_2^2$ term in \eqref{def:hifEvo} and is what allows to deduce the rapid exponential decay rate for spatially dependent modes. 

Next, we are concerned with estimating the error terms $\cD_1$ and $\cD_2$. 
First, note that if $\abs{\alpha} = 0$ then neither $\cD_1$ nor $\cD_2$ are present. 
Consider $\mathcal{D}_1$. By Cauchy-Schwarz, 
\begin{align*}
\abs{\cD_1} \lesssim \nu e^{-2\abs{\alpha}\nu t} \sum_{\beta \leq \alpha: \abs{\beta} = \abs{\alpha} - 1}\norm{A (v^\alpha f)}_{2} K_{\abs{\alpha}-1} \left(e^{\nu t} \frac{1}{K_{\abs{\alpha}-1}} \norm{\partial_v^t A (v^{\alpha-\beta} f)}_2\right),
\end{align*}
and 
\begin{align*}
\abs{\cD_2} \lesssim \nu e^{-2\abs{\alpha}\nu t} \sum_{\beta \leq \alpha: \abs{\beta} = \abs{\alpha}- 2} \norm{A (v^\alpha f)}_{2} K_{\abs{\alpha}-2} \left(e^{2\nu t} \frac{1}{K_{\abs{\alpha}-2}} \norm{ A (v^{\alpha-\beta} f)}_2 \right). 
\end{align*}
Hence, for some constant $C$ which can be taken depending only on $\sigma$ and $n$, 
\begin{align*}
\abs{\cD_1} \leq \frac{1}{4}\nu e^{-2\abs{\alpha}\nu t}\norm{A (v^\alpha f)}_{2}^2 +  C K_{\abs{\alpha}-1}\nu e^{2\abs{\alpha}t+2\nu t} \sum_{\beta \leq \alpha: \abs{\beta} = \abs{\alpha} - 1} \frac{1}{K_{\abs{\alpha}-1}} \norm{\partial_v^t A (v^{\alpha-\beta} f)}_2^2. 
\end{align*}
The first term is absorbed by $\cD_\alpha$. The latter term is consistent with Proposition \ref{prop:boot} by choosing  $K_{\abs{\alpha}} \gg K_{\abs{\alpha}-1}$ (depending only on $C$). 
The $\cD_2$ term is similarly controlled  via:
\begin{align*}
\abs{\cD_2} \leq \frac{1}{4}\nu e^{-2\abs{\alpha}\nu t}\norm{A (v^\alpha f)}_{2}^2 +  C K_{\abs{\alpha}-2}\nu e^{2\abs{\alpha}t+4\nu t} \sum_{\beta \leq \alpha: \abs{\beta} = \abs{\alpha} - 2} \frac{1}{K_{\abs{\alpha}-2}} \norm{A (v^{\alpha-\beta} f)}_2^2. 
\end{align*}
As above, this is consistent Proposition \ref{prop:boot} by choosing  $K_{\abs{\alpha}} \gg K_{\abs{\alpha}-2}$ (depending only on $C$). 

\subsubsection{Linear collisionless term} \label{sec:HiLinCless}
Write, 
\begin{align*}
\mathcal{L} & =  e^{-2\abs{\alpha}\nu t}\sum_{k\in \Integer^n_\ast} \int A D_\eta^\alpha \overline{\widehat{f}}(k,\eta) A(t,k,\eta) \widehat{\rho}(t,k) \widehat{W}(k) \\ & \quad\quad \times k \cdot D_\eta^\alpha \left(\left(\eta e^{\nu t} - k \frac{1-e^{\nu t}}{\nu}\right)\widehat{\mu}\left(e^{\nu t}\eta - k\frac{e^{\nu t} - 1}{\nu}\right) \right)\dd \eta. 
\end{align*}
Using
\begin{align*}
\abs{\eta} \leq \abs{k \frac{1-e^{-\nu t}}{\nu}} + \abs{\eta-k \frac{1-e^{-\nu t}}{\nu}} \leq \abs{kt} + \abs{\eta-k \frac{1-e^{-\nu t}}{\nu}}, 
\end{align*}
we have,
\begin{align*}
\abs{\mathcal{L}} & \lesssim \brak{t}^{-p/2} e^{-2\abs{\alpha}\nu t}\sum_{k \in \Integers_\ast^n} \int A \abs{D_\eta^\alpha \widehat{f}(k,\eta)} \left(\abs{kt} + \abs{\eta-k\frac{1-e^{-\nu t}}{\nu}}\right) e^{\delta_1 \nu^{1/3}t} \brak{k,kt}^{\gamma} \abs{\widehat{\rho}(t,k)} \\ & \quad\quad \times \abs{k\widehat{W}(k)} \brak{\eta-k \frac{1-e^{-\nu t}}{\nu}}^{\gamma} \abs{D_\eta^\alpha \left(\left(\eta e^{\nu t} - k \frac{1-e^{\nu t}}{\nu}\right)\widehat{\mu}\left(e^{\nu t}\eta - k\frac{e^{\nu t} - 1}{\nu}\right) \right)} \dd \eta. 
\end{align*} 
By Lemma \ref{lem:YoungDisc} there holds for some constant $C = C(\sigma,j,n) >0$,  
\begin{align*}
\abs{\mathcal{L}} & \lesssim \brak{t}^{-p/2+1}e^{-\abs{\alpha}\nu t}\norm{A(v^\alpha f)}_2 \norm{\abs{\grad_x}^{1/2}\rho(t)}_{\cM^{\sigma,\delta_1}} \\
& \leq \frac{1}{10} \brak{t}^{-2} e^{-2\abs{\alpha}\nu t}\norm{A(v^\alpha f)}_2^2 + C\brak{t}^{-4-p} \norm{\abs{\grad_x}^{1/2}\rho(t)}_{\cM^{\sigma,\delta}}^2. 
\end{align*}
This is consistent with Proposition \ref{prop:boot} provided $p \geq 4$ provided we choose $K_{Hj} \gg K_{MH0}$. 

\subsubsection{Nonlinear collisionless term} \label{sec:HiNLCless}
The dissipation is too weak to control this nonlinear term directly using only $\eps \ll \nu^{1/3}$. 
Therefore, we need to use collisionless energy estimate techniques for transport equations, which amounts to variations of treatments that appear in \cite{BMM13,Bedrossian16} adapted to include $M(t,k,\eta)$ and the rapid decay into the norm. 
First, we single out the term with the highest derivatives in $\eta$: 
\begin{align*} 
\mathcal{N} & = -\sum_{k,\ell \in \Integer^n} e^{-2\abs{\alpha}\nu t} \int A D_\eta^\alpha \overline{\hat{f}} A \widehat{E}(t,\ell) \partial_v^t D_\eta^\alpha \hat{f}\left(t,k-\ell, \eta- \ell \frac{1-e^{-\nu t}}{\nu}\right) \dd \eta \\ 
& \quad -\sum_{\beta \leq \alpha : \abs{\beta} =1}\sum_{k,\ell \in \Integer^n} e^{-2\abs{\alpha}\nu t} \int A D_\eta^\alpha \overline{\hat{f}} A \widehat{E}(t,\ell) e^{\nu t} D_\eta^{\alpha-\beta} \hat{f}\left(t,k-\ell, \eta- \ell \frac{1-e^{-\nu t}}{\nu}\right) \dd \eta \\ 
& = \mathcal{N}_M + \mathcal{N}_{P}. 
\end{align*} 
Consider the main term $\mathcal{N}_M$, and employ a classical commutator trick and frequency decomposition:  
\begin{align*} 
\mathcal{N}_{M} &=  \sum_{k,\ell \in \Integer^n} e^{-2\abs{\alpha}\nu t} \int A D_\eta^\alpha \overline{\hat{f}} \left(A(k,\eta) - A(k-\ell,\eta - \ell \frac{1-e^{-\nu t}}{\nu})\right)\hat{E}(t,\ell) \\ & \quad\quad \cdot \; \widehat{\partial_v^t} D_\eta^\alpha \hat{f}\left(t,k-\ell, \eta- \ell \frac{1-e^{-\nu t}}{\nu}\right) \dd \eta \\ 
& = \sum_{k,\ell \in \Integer^n} e^{-2\abs{\alpha}\nu t} \int \mathbf{1}_{\abs{\ell, \ell \frac{1-e^{-\nu t}}{\nu}} \leq \abs{k-\ell, \eta - \ell \frac{1-e^{-\nu t}}{\nu}}} A D_\eta^\alpha \overline{\hat{f}} \left(A(k,\eta) - A(k-\ell,\eta - \ell \frac{1-e^{-\nu t}}{\nu})\right) \\ & \quad\quad \times \hat{E}(t,\ell) \cdot \widehat{\partial_v^t} D_\eta^\alpha \hat{f}\left(t,k-\ell, \eta- \ell \frac{1-e^{-\nu t}}{\nu}\right) \dd \eta \\ 
& \quad + \sum_{k,\ell \in \Integer^n} e^{-2\abs{\alpha}\nu t} \int \mathbf{1}_{\abs{\ell, \ell \frac{1-e^{-\nu t}}{\nu}} \geq \abs{k-\ell, \eta - \ell \frac{1-e^{-\nu t}}{\nu}}} A D_\eta^\alpha \overline{\hat{f}} \left(A(k,\eta) - A(k-\ell,\eta - \ell \frac{1-e^{-\nu t}}{\nu})\right) \\ & \quad\quad \times \hat{E}(t,\ell) \cdot \widehat{\partial_v^t} D_\eta^\alpha \hat{f}\left(t,k-\ell, \eta- \ell \frac{1-e^{-\nu t}}{\nu}\right) \dd \eta \\ 
& = \mathcal{N}_{LH} + \mathcal{N}_{HL}. 
\end{align*}
Consider first the contribution from $\mathcal{N}_{LH}$.
Write 
\begin{align*}
D(t,k) = \mathbf{1}_{k=0} + \mathbf{1}_{k\neq 0} e^{\delta_1 \nu^{1/3}t}.
\end{align*}
Next, expand
\begin{align}
A(k,\eta) - A(k-\ell,\eta - \ell \frac{1-e^{-\nu t}}{\nu}) & = \nonumber \\ & \hspace{-6cm} \left(\brak{k,\eta}^{\gamma} - \brak{k-\ell,\eta - \ell \frac{1-e^{-\nu t}}{\nu}}^{\gamma}\right)M(t,k,\eta)D(t,k)  \nonumber \\
& \hspace{-6cm} \quad + \brak{k-\ell,\eta - \ell \frac{1-e^{-\nu t}}{\nu}}^{\gamma}\left(M(t,k,\eta) - M(t,k-\ell,k-\ell,\eta - \ell \frac{1-e^{-\nu t}}{\nu})\right)D(t,k) \nonumber \\ 
& \hspace{-6cm} \quad + \brak{k-\ell,\eta - \ell \frac{1-e^{-\nu t}}{\nu}}^{\gamma}M(t,k-\ell,k-\ell,\eta - \ell \frac{1-e^{-\nu t}}{\nu})\left(D(t,k) - D(t,k-\ell)\right). \label{ineq:trans0}
\end{align}
On the support of the integrand in $\cN_{LH}$, there holds, 
\begin{align}
\abs{\brak{k,\eta}^{\gamma} - \brak{k-\ell,\eta - \ell \frac{1-e^{-\nu t}}{\nu}}^{\gamma}} \lesssim \brak{k-\ell,\eta-\ell \frac{1-e^{-\nu t}}{\nu}}^{\gamma-1}\abs{\ell, \ell \frac{1-e^{-\nu t}}{\nu}}. \label{ineq:trans1}
\end{align}
By  Lemma \ref{lem:propM}, there holds 
\begin{align}
\abs{M(t,k,\eta) - M(t,k-\ell,k-\ell,\eta - \ell \frac{1-e^{-\nu t}}{\nu})} & \lesssim \frac{\brak{t}^2}{\nu^{1/3}\max(\brak{\eta},\brak{k})} \brak{\ell,\ell \frac{1-e^{-\nu t}}{\nu}}^3.  \label{ineq:trans2}
\end{align}
Similarly, using that $\ell \neq 0$ (and hence, it cannot be that both $k$ and $k-\ell$ are zero), there holds 
\begin{align}
\abs{D(t,k) - D(t,k-\ell)} \leq \mathbf{1}_{\abs{k}\abs{k-\ell} = 0} \abs{e^{\delta_1 \nu^{1/3}t} - 1} \leq  \mathbf{1}_{\abs{k}\abs{k-\ell} = 0} \delta_1 \nu^{1/3} t e^{\delta_1 \nu^{1/3}t}. \label{ineq:trans3}
\end{align}
When $\abs{k}\abs{k-\ell} \neq 0$ then the commutator involving $D(t,k) - D(t,k-\ell)$ vanishes. In the case $\abs{k}\abs{k-\ell} = 0$ then this commutator does not gain regularity and hence we need to use the dissipation. Hence, applying the decomposition \eqref{ineq:trans0} and the multiplier estimates \eqref{ineq:trans0}, \eqref{ineq:trans1}, and \eqref{ineq:trans2}, it follows from Lemma \ref{lem:YoungDisc} (and $\beta > n/2$) that, 
\begin{align*}
\abs{\mathcal{N}_{LH}} & \lesssim \brak{t}^{-\beta/2+1}e^{-(\delta-\delta_1) \nu^{1/3}t} e^{-2\abs{\alpha}\nu t} \norm{A v^\alpha f(t)}_2^2 \norm{\rho(t)}_{\cM^{\beta,\delta}} \\ 
& \quad + \nu^{-1/3} \brak{t}^{-\beta/2+3}e^{-(\delta-\delta_1) \nu^{1/3}t} e^{-2\abs{\alpha}\nu t} \norm{A v^\alpha f(t)}_2^2 \norm{\rho(t)}_{\cM^{\beta,\delta}} \\ 
& \quad + \nu^{1/3} \brak{t}^{-\beta/2+2}e^{-(\delta-\delta_1) \nu^{1/3}t} e^{-2\abs{\alpha}\nu t} \norm{A v^\alpha f(t)}_2 \norm{\partial_v^t A(v^\alpha f)}_2 \norm{\rho(t)}_{\cM^{\beta,\delta}} \\ 
& \lesssim \nu^{-1/3} \brak{t}^{-\beta/2+3} e^{-2\abs{\alpha}\nu t} \norm{A v^\alpha f(t)}_2^2 \norm{\rho(t)}_{\cM^{\beta,\delta}} \\ & \quad +  \nu \norm{\rho(t)}_{\cM^{\beta,\delta}} e^{-2\abs{\alpha}\nu t} \norm{\partial_v^t A(v^\alpha f)}_2^2. 
\end{align*}
This is consistent with Proposition \ref{prop:boot} by choosing $\eps \nu^{-1/3} < c_0$ sufficiently small. 

Next, turn to $\mathcal{N}_{HL}$. 
This contribution can be treated in a manner analogous to the collisionless papers \cite{BMM13,BMM16,Bedrossian16}. 
By the frequency localization, Lemma \ref{lem:YoungDisc}, and \eqref{boot:lowneq} (recall \eqref{def:B}), 
\begin{align*}
\abs{\mathcal{N}_{HL}} & \lesssim \sum_{k,\ell \in \Integer^n} e^{-2\abs{\alpha}\nu t} \int \abs{AD_\eta^\alpha \hat{f}(t,k,\eta)} \frac{1}{\abs{\ell}} A\left(t,\ell,\ell\frac{1-e^{-\nu t}}{\nu}\right) \\ & \quad\quad \times \abs{\hat{\rho}(t,\ell) \widehat{\partial_v^t D_\eta^\alpha f}\left(t,k-\ell, \eta- \ell \frac{1-e^{-\nu t}}{\nu}\right)} \dd \eta \\
& \lesssim \sum_{k,\ell \in \Integer^n} e^{-2\abs{\alpha}\nu t - (\delta - \delta_1)\nu^{1/3}t} \int \abs{AD_\eta^\alpha \hat{f}(t,k,\eta)} \\ & \quad\quad \times \abs{\ell}^{1/2} \left(\frac{1-e^{-\nu t}}{\nu}\right)\abs{e^{\nu t}\eta - (k-\ell)\frac{e^{\nu t} - 1}{\nu}} \abs{B\hat{\rho}(t,\ell) \widehat{D_\eta^\alpha f}\left(t,k-\ell, \eta- \ell \frac{1-e^{-\nu t}}{\nu}\right)} \dd \eta \\
& \lesssim e^{-(\delta-\delta_1) \nu^{1/3}t + \nu t} \brak{t}^{2-p/2} e^{-2\abs{\alpha}\nu t} \norm{A(v^\alpha f)}_2\norm{\rho}_{\cM^{\sigma,\delta}} \norm{\brak{\grad}^{-3}A(v^\alpha f)}_2 \\
& \lesssim \eps \brak{t}^{-2} e^{-2\abs{\alpha}\nu t} \norm{A(v^\alpha f)}_2^{2} + \brak{t}^{6-p} \eps \norm{\rho}_{\cM^{\sigma,\delta}}^2, 
\end{align*}
which is consistent provided $p \geq 6$,  $\eps \ll 1$, and $\nu \ll 1$ by the bootstrap hypotheses. Note that we used $\delta > \delta_1$ here.  
This completes the treatment of $\mathcal{N}_M$. 

Turn next to $\mathcal{NL}_P$. 
By arguments used above, 
\begin{align*}
\abs{\mathcal{N}_P} & \lesssim e^{-2\abs{\alpha}\nu t}e^{-(\delta-\delta_1)\nu^{1/3}t}\norm{\rho(t)}_{\cM^{\sigma,\delta}}\norm{A(v^\alpha f(t))}_2\sum_{\beta \leq \alpha:\abs{\beta} = 1} \norm{A(v^{\alpha-\beta}f(t))}_2 e^{\nu t} \\ 
& \lesssim e^{-\abs{\alpha} \nu t} e^{-\frac{1}{2}(\delta-\delta_1)\nu^{1/3}t}\norm{\rho(t)}_{\cM^{\sigma,\delta}}\norm{A(v^\alpha f(t))}_2 \norm{f(t)}_{F^{\sigma+1/2-j,\delta_1}_{m+1+j}}. 
\end{align*}
which is consistent with Proposition \ref{prop:boot} for $\eps \nu^{-1/6} \ll 1$ after summing over $\alpha$. 

\subsubsection{$\cC_\mu$ collision term} \label{sec:HicCmu}
Consider next the $\cC_\mu$ collision term:
\begin{align*}
\nu C_\mu & = \nu e^{-2\abs{\alpha}\nu t} \int  A(t,\grad) (v^\alpha f) A(t,\grad)\left( M_T v^\alpha \Delta_v \mu - v^\alpha M_1 \cdot \grad_v\mu \right) \dd x \, \dd v 
 = \nu\sum_{j=1}^2 C_{\mu;j}. 
\end{align*}
The terms $C_{\mu;1}$ and $C_{\mu;2}$ are similar and hence it suffices to consider  $C_{\mu;1}$.
First, separate zero and non-zero mode contributions: 
\begin{align*}
\nu C_{\mu;1} & = \nu e^{-2\abs{\alpha}\nu t}\int AD_\eta^\alpha \overline{\hat{f}}(t,0,\eta) A(t,0,\eta) \widehat{M_T}(t,0) D_\eta^\alpha \left(\widehat{(\Delta_v \mu)}(\eta e^{\nu t})\right) \dd \eta \\ 
 & \quad + \nu e^{-2\abs{\alpha}\nu t} \sum_{k \in \Integer^n_\ast}\int A\overline{\hat{f}}(t,k,\eta) A(t,k,\eta) \widehat{M_T}(t,k) D_\eta^\alpha \left(\widehat{(\Delta_v \mu)}(\eta e^{\nu t} - k\frac{e^{\nu t} - 1}{\nu})\right) \dd \eta \\ 
& = \nu C_{\mu;1,0} + \nu C_{\mu;1,\neq}. 
\end{align*}
Using, $\abs{\eta} \leq \abs{k\frac{1-e^{-\nu t}}{\nu}} + e^{-\nu t}\abs{e^{\nu t} \eta - k\frac{e^{\nu t}-1}{\nu}}$, we have by Cauchy-Schwarz (using $\abs{k \frac{1-e^{-\nu t}}{\nu}} \leq \abs{kt}$),
\begin{align*}
\nu \abs{C_{\mu;1,\neq}} \lesssim \nu e^{-\abs{\alpha}\nu t} \brak{t}^{(-p+1)/2} \norm{A(v^\alpha f)}_2 \norm{\abs{\grad}^{1/2} (M_T)_{\neq}(t)}_{\cM^{\sigma,\delta}}, 
\end{align*}
which is consistent with Proposition \ref{prop:boot} by choosing $K_{Hj} \gg K_{MH2}$ and $p>1$. 
For $C_{\mu;1,0}$, 
\begin{align*}
\nu \abs{C_{\mu;1,0}} \lesssim \nu e^{-\abs{\alpha}\nu t} \brak{t}^{-p/2} \norm{A(v^\alpha f)}_2 \abs{(M_T)_0(t)}, 
\end{align*}
which is consistent with Proposition \ref{prop:boot} by Lemma \ref{lem:uTctrls}. 

\subsubsection{$\cC_h$ collision term} \label{sec:HicCh}
Subdivide the contributions of $\cC_h$ into the four natural contributions in \eqref{def:cCh}, inducing a similar decomposition of the analogous term in the energy estimate as $\nu \sum_{j=1}^4 J_{\cC_h;j}$. 

Consider the contribution of the first term $J_{\cC_h;1}$ and divide into two distinct contributions: 
\begin{align*}
\nu J_{\cC_h;1} & = \nu e^{-2\abs{\alpha}\nu t} \sum_{k \in \Integers^n}  \int A D_\eta^\alpha \overline{\hat{f}}  A(t,k,\eta) \hat{\rho}(t,k)  D_\eta^\alpha \left(\widehat{(\Delta_v^t f)}(t,0, \eta  - k \frac{1- e^{-\nu t}}{\nu})\right) \dd \eta \\ 
& \quad + \nu e^{-2\abs{\alpha}\nu t} \sum_{k,\ell \in \Integers^n}  \int A D_\eta^\alpha \overline{\hat{f}}  A(t,k,\eta) \hat{\rho}(t,\ell)  D_\eta^\alpha \left(\widehat{(\Delta_v^t f)}(t,k-\ell, \eta  - \ell \frac{1- e^{-\nu t}}{\nu})\right) \dd \eta \\ 
& = \nu J_{\cC_h;1}^0 + \nu J_{\cC_h;1}^{\neq}.
\end{align*}
Consider first the interaction with zero frequency modes: 
\begin{align*}
\nu J_{\cC_h;1}^0 & = -\nu e^{-2\abs{\alpha}\nu t} \sum_{k \in \Integers^n}  \int A D_\eta^\alpha \overline{\hat{f}}  A(t,k,\eta) \hat{\rho}(t,k) \abs{e^{\nu t} \eta - k \frac{e^{\nu t}- 1}{\nu}}^2 \widehat{D_\eta^\alpha f}(t,0, \eta  - k \frac{1- e^{-\nu t}}{\nu}) \dd \eta  \\ 
& \quad - \nu e^{-2\abs{\alpha}\nu t + \nu t} \sum_{\beta \leq \alpha : \abs{\beta} = 1}\sum_{k \in \Integers^n}  \int A D_\eta^\alpha \overline{\hat{f}}  A(t,k,\eta) \hat{\rho}(t,k) D^\beta_\eta \eta \left(e^{\nu t} \eta - k \frac{e^{\nu t}- 1}{\nu}\right) \\ & \quad\quad \times \widehat{D_\eta^{\alpha-\beta} f}(t,0, \eta  - k \frac{1- e^{-\nu t}}{\nu}) \dd \eta  \\ 
& \quad - \nu e^{-2\abs{\alpha}\nu t + 2\nu t} \sum_{\beta \leq \alpha : \abs{\beta} = 2}\sum_{k \in \Integers^n}  \int A D_\eta^\alpha \overline{\hat{f}}  A(t,k,\eta) \hat{\rho}(t,k) \widehat{D_\eta^{\alpha-\beta} f}(t,0, \eta  - k \frac{1- e^{-\nu t}}{\nu}) \dd \eta \\ 
& = \nu\sum_{j=1}^3J_{\cC_h;1,j}^0. 
\end{align*}
Using $\brak{k,\eta} \lesssim \brak{k,k\frac{1-e^{-\nu t}}{\nu}} + \brak{\eta - k \frac{1-e^{-\nu t}}{\nu}}$ and $\frac{1-e^{-\nu t}}{\nu} \leq t$ we have (using also Cauchy-Schwarz), 
\begin{align*}
\abs{\nu J_{\cC_h;1,j}^0} & \lesssim \nu e^{-2\abs{\alpha} \nu t + \nu t}\brak{t}^{1/2}\norm{\partial_v^t A(v^\alpha f)}_2 e^{(\delta_1-\delta)\nu^{1/3}t} \norm{\abs{\grad}^{1/2}\rho}_{\cM^{\sigma,\delta}}\norm{\brak{\grad}^{-5}A(v^\alpha f)}_{2} \\ 
& \quad + \nu e^{-2\abs{\alpha} \nu t}\norm{\partial_v^t A(v^\alpha f)}_2^2 e^{(\delta_1-\delta)\nu^{1/3}t} \norm{\rho}_{\cM^{\beta,\delta}} \\ 
& \lesssim \nu^{5/6} \eps e^{-\abs{\alpha} t}\norm{\partial_v^t A(v^\alpha f)}_2 \norm{\abs{\grad}^{1/2}\rho}_{\cM^{\sigma,\delta}} + \nu \eps e^{-2\abs{\alpha} \nu t}\norm{\partial_v^t A(v^\alpha f)}_2^2, 
\end{align*}
where the last line followed from using $\delta > \delta_1$. 
This is consistent with Proposition \ref{prop:boot} for $\eps$ and $\nu$ sufficiently small. 
The remaining terms $J_{\cC_h;1,2}^0$ and $J_{\cC_h;1,3}^0$ are treated with an easy variant. % and are omitted for the sake of brevity.  

Consider next the interaction with spatially dependent modes $J_{\cC_h;1}^{\neq}$, 
\begin{align*}
\nu J_{\cC_h;1}^{\neq} & = -\nu e^{-2\abs{\alpha}\nu t} \sum_{k,\ell \in \Integers^n} \int A D_\eta^\alpha \overline{\hat{f}}  A(t,k,\eta) \hat{\rho}(t,\ell)  \abs{e^{\nu t} \eta - k \frac{e^{\nu t}-1}{\nu}}^2 D_\eta^\alpha \hat{f}(t,k-\ell, \eta  - \ell \frac{1 - e^{-\nu t}}{\nu}) \dd \eta \\ 
& \quad - 2 e^{\nu t}\sum_{\beta \leq \alpha: \abs{\beta} = 1} \nu e^{-2\abs{\alpha}\nu t} \sum_{k,\ell \in \Integers^n} \int A D_\eta^\alpha \overline{\hat{f}}  A(t,k,\eta) \hat{\rho}(t,\ell)  \\ & \quad\quad \times  D_\eta^\beta(\eta) \left(e^{\nu t} \eta - k \frac{e^{\nu t}-1}{\nu}\right) D_\eta^{\alpha-\beta} \hat{f}(t,k-\ell, \eta  - \ell \frac{1 - e^{-\nu t}}{\nu}) \dd \eta \\  
& \quad - 2e^{2\nu t}\sum_{\beta \leq \alpha: \abs{\beta} = 2} \nu e^{-2\abs{\alpha}\nu t} \sum_{k,\ell \in \Integers^n} \int A D_\eta^\alpha \overline{\hat{f}}  A(t,k,\eta) \hat{\rho}(t,\ell) D_\eta^{\alpha-\beta} \hat{f}(t,k-\ell, \eta  - \ell \frac{1 - e^{-\nu t}}{\nu}) \dd \eta \\ 
& = \sum_{j = 1}^3  \nu J_{\cC_h;1,j}.  
\end{align*}
Notice the structure which implies that $e^{\nu t} \eta - k\frac{e^{\nu t}-1}{\nu} = e^{\nu t}(\eta-\ell \frac{1-e^{-\nu t}}{\nu}) - (k-\ell)\frac{e^{\nu t}-1}{\nu}$. 
This is crucially important to controlling a potential loss of regularity here. 
Note 
\begin{align*}
\brak{k,\eta} & \lesssim \brak{\ell,\ell\frac{1-e^{-\nu t}}{\nu}} +  \brak{k-\ell, \eta-\ell \frac{1-e^{-\nu t}}{\nu}}, \\
e^{\nu t}\abs{\eta - k \frac{1-e^{-\nu t}}{\nu}} & \lesssim e^{\nu t} \brak{t} \brak{k-\ell,\eta - \ell \frac{1-e^{-\nu t}}{\nu}},
\end{align*}
and further that $\brak{\ell,\ell \frac{1-e^{-\nu t}}{\nu}} \lesssim \brak{t} \brak{\ell}$. 
Therefore, from Lemma \ref{lem:YoungDisc} we have the following product type estimate for $\sigma + 1/2 -j > n/2+1$ (we are grouping one of the $\partial_v^t$'s with the leading factor), 
\begin{align*}
\nu \abs{J_{\cC_h;1,1}} & \lesssim \nu \brak{t} e^{-2\abs{\alpha}\nu t} \norm{\partial_v^t A(v^\alpha f)}_2 e^{- \frac{1}{2}\delta \nu^{1/3}t} \norm{\brak{\grad_x,\grad_x \frac{1-e^{-\nu t}}{\nu}}^{1/2} \rho}_{\cM^{\sigma,\delta}} \norm{\brak{\grad}^{-5}A(v^\alpha f)}_2 \\ 
& \quad  + \nu e^{-2\abs{\alpha}\nu t} \norm{\partial_v^t A(v^\alpha f)}_2 \brak{t}^{-\beta/2} e^{-\frac{1}{2}\delta \nu^{1/3}t} \norm{\rho}_{\cM^{\beta,\delta}} \norm{\partial_v^t A(v^\alpha f)}_2 \\ 
& \lesssim \nu \brak{t}^{3/2} e^{-\frac{1}{2}\delta \nu^{1/3}t} e^{-2\abs{\alpha}\nu t} \norm{\partial_v^t A(v^\alpha f)}_2  \norm{\abs{\grad_x}^{1/2} \rho}_{\cM^{\sigma,\delta}} \norm{\brak{\grad}^{-5}A(v^\alpha f)}_2 \\ 
& \quad  + \nu \brak{t}^{-\beta/2} e^{-\frac{1}{2}\delta \nu^{1/3}t} e^{-2\abs{\alpha}\nu t} \norm{\partial_v^t A(v^\alpha f)}_2^2  \norm{\rho}_{\cM^{\beta,\delta}}  \\
& \lesssim \nu^{1/2} \eps e^{-\abs{\alpha}\nu t} \norm{\partial_v^t A(v^\alpha f)}_2 \norm{\abs{\grad_x}^{1/2} \rho}_{\cM^{\sigma,\delta}} + \nu e^{-2\abs{\alpha} \nu t} \eps \norm{\partial_v^t A (v^\alpha f)}_2^2, 
\end{align*}
which is consistent with Proposition \ref{prop:boot} for $\eps$ sufficiently small. 
The terms $J_{\cC_h;1,2}$ and $J_{\cC_h;1,3}$ are treated with easier variants; we omit the treatment for the sake of brevity. 
This completes the treatment of $J_{\cC_h;1}$. 
The $J_{\cC_h;4}$ term is similar (as $(M_1)_0 = 0$) and is hence omitted for the sake of brevity.
The $J_{\cC_h;3}$ term is similar, except for the contribution of $(M_T)_0$.
However, this is treated by an easy variation of the above arguments, and is hence omitted for the sake of brevity.  

Turn next to $J_{\cC_h;2}$, where the main complication is the loss of velocity localization. 
This will require a variant of the commutator trick used to treat the collisionless transport term which permits to trade regularity for velocity localization.  
Begin by extracting the leading order term up to commutators: 
\begin{align*}
\nu J_{\cC_h;2} & = \nu e^{-2\abs{\alpha} \nu t -\nu t} \brak{A(v^\alpha f), A(t,\grad) \left(\rho(t,x+v\frac{1-e^{-\nu t}}{\nu}) v^\alpha \partial_v^t \cdot (vf)\right)} \\ 
& = \nu e^{-2\abs{\alpha} \nu t-\nu t} \brak{A(v^\alpha f), A(t,\grad) \left(\rho(t,x+v\frac{1-e^{-\nu t}}{\nu}) v^\alpha v \cdot \partial_v^t f\right)} \\ 
& \quad + n \nu e^{-2\abs{\alpha} \nu t} \brak{A(v^\alpha f), A(t,\grad) \left(\rho(t,x+v\frac{1-e^{-\nu t}}{\nu}) v^\alpha f \right)} \\  
& = \nu e^{-2\abs{\alpha} \nu t-\nu t} \brak{A(v^\alpha f), \rho(t,x+v\frac{1-e^{-\nu t}}{\nu})  v \cdot \partial_v^t A(v^\alpha f) } \\ 
& \quad + n \nu e^{-2\abs{\alpha} \nu t} \brak{A(v^\alpha f), A(t,\grad) \left(\rho(t,x+v\frac{1-e^{-\nu t}}{\nu}) v^\alpha f \right)} \\ 
& \quad + \nu e^{-2\abs{\alpha} \nu t-\nu t} \Big\langle A(v^\alpha f), A(t,\grad) \left(\rho(t,x+v\frac{1-e^{-\nu t}}{\nu}) v^\alpha v \cdot \partial_v^t f\right) \\ &  \quad\quad  - \rho(t,x+v\frac{1-e^{-\nu t}}{\nu})  v \cdot \partial_v^t A(v^\alpha f) \Big\rangle \\ 
& = \nu \sum_{j=1}^3 J_{\cC_h;2,j}. 
\end{align*}
To treat $J_{\cC_h;2,1}$ we integrate by parts to reduce by one power of $v$. 
From there, both $J_{\cC_h;2,1}$ and $J_{\cC_h;2,2}$ are straightforward using previously discussed techniques and are hence omitted for brevity: 
\begin{align*}
\nu J_{\cC_h;2,1} + \nu J_{\cC_h;2,2} & \lesssim \nu\norm{A(v^\alpha f)}_{2} \Big( \brak{t}^{-\beta/2}\norm{A(v^\alpha f)}_2 \norm{\rho(t)}_{\cM^{\beta,\delta}} \\ 
& \quad\quad  + \brak{t}^{1-p/2}\norm{\abs{\grad}^{1/2}\rho}_{\cM^{\sigma,\delta}}\norm{f}_{F^{\sigma+1/2-j-4,\delta_1}_{m+1+j}}\Big), 
\end{align*}
which is consistent with Proposition \ref{prop:boot} for $p > 3$ and $\nu$ sufficiently small. 
Hence, the remaining difficulty is the commutator term $J_{\cC_h;2,3}$. 
Expand on the Fourier side: 
\begin{align*} 
\nu J_{\cC_h;2,3} & = \nu e^{-2\abs{\alpha}\nu t-\nu t} \sum_{k,\ell \in \Integers^n} \int A D_\eta^\alpha \hat{f}(t,k,\eta) \hat{\rho}(t,\ell) \\& \quad\quad \times\left(A(t,k,\eta) D_\eta^\alpha \grad_\eta \cdot (\widehat{\partial_v^t f})(t,k-\ell, \eta-\ell \frac{1-e^{-\nu t}}{\nu}) - \grad_\eta \cdot \partial_v^t A D_\eta^\alpha \widehat{f}(t,k-\ell,\eta- \ell \frac{1-e^{-\nu t}}{\nu})\right) \dd \eta \\ 
& = \nu e^{-2\abs{\alpha}\nu t-\nu t} \sum_{k,\ell \in \Integers^n} \int A D_\eta^\alpha \hat{f}(t,k,\eta) \hat{\rho}(t,\ell) \\
& \quad\quad \times\left(-\sum_{\beta \leq \alpha: \abs{\beta} = 1} n e^{\nu t} A D_\eta^{\alpha} \widehat{f}(t,k-\ell,\eta- \ell \frac{1-e^{-\nu t}}{\nu}) - (\grad_\eta A) D_\eta^\alpha \partial_v^t \hat{f}(t,k-\ell,\eta-\ell \frac{1-e^{-\nu t}}{\nu}) \right) \dd \eta \\ 
& \quad + \nu e^{-2\abs{\alpha}\nu t -\nu t} \sum_{k,\ell \in \Integers^n} \int A D_\eta^\alpha \hat{f}(t,k,\eta) \hat{\rho}(t,\ell) \\
& \quad\quad \times\left(A(t,k,\eta) - A(t,k-\ell,\eta-\ell \frac{1-e^{-\nu t}}{\ell})\right)D_\eta^\alpha \grad_\eta \cdot (\widehat{\partial_v^t f}) \dd \eta \\  
& = \nu J_{\cC_h;2,3,a} + J_{\cC_h;2,3,b}. 
\end{align*}
The first term is treated using \eqref{ineq:detaM} to control $\grad_\eta A$ and the techniques used above to treat $J_{\cC_h;1}$. The details are omitted for brevity and we deduce 
\begin{align*}
\nu \abs{J_{\cC_h;2,3,a}} & \lesssim \nu e^{ -(\delta-\delta_1)\nu^{1/3}t} \norm{\rho(t)}_{\cM^{\beta,\delta}}\norm{A(v^\alpha f)}_{2} \left(\norm{f}_{F^{\sigma+1/2-j,\delta_1}_{m+1+j}} + \norm{f}_{\cD^{\sigma+1/2-j,\delta_1}_{m+1+j}}\right), 
\end{align*}
 which is consistent with Proposition \ref{prop:boot} for $\eps$ and $\nu$ sufficiently small. 
The $J_{\cC_h;1,3,b}$ term requires a more detailed treatment due to the inevitable loss of regularity in $\eta$, equivalently, a loss of localization in $v$. 
Begin by commuting the derivatives in $\eta$ with $\partial_v^t$: when one of the derivatives lands on $\partial_v^t$, the loss in $\eta$ is traded for a slight growth in time and this is overcome easily using $\rho$, hence, for some constant $C$: 
\begin{align*}
\nu J_{\cC_h;2,3,b} & \leq \nu e^{-2\abs{\alpha}\nu t -\nu t} \sum_{k,\ell} \int \partial_v^t A D_\eta^\alpha \hat{f}(t,k,\eta) \hat{\rho}(t,\ell) \\
& \quad\quad \times\left(A(t,k,\eta) - A(t,k-\ell,\eta-\ell \frac{1-e^{-\nu t}}{\ell})\right)D_\eta^\alpha \grad_\eta \widehat{f}(t,k-\ell,\eta-\ell\frac{1-e^{-\nu t}}{\nu}) \dd \eta \\ 
& \quad + C\nu e^{\nu t - (\delta - \delta_1)\nu^{1/3}t}\norm{A(v^\alpha f)}_2 \left(\brak{t}^{1-p/2} \norm{\abs{\grad}^{1/2}\rho}_{\cM^{\sigma,\delta}} \norm{f}_{F^{\sigma+1/2-3-j,\delta_1}_{m+1+j}} \right. \\ & \left. \quad\quad + \brak{t}^{-\beta/2}\norm{\rho}_{\cM^{\beta,\delta}}\norm{f}_{F^{\sigma+1/2-j,\delta_1}_{m+1+j}}\right). 
\end{align*}
Next consider the first term. 
When $\abs{k-\ell,\eta- \ell \frac{1-e^{-\nu t}}{\nu}} \approx \abs{k,\eta}$, we can apply the arguments used to treat $\cN_{LH}$ \S\ref{sec:HiNLCless}, specifically, 
\eqref{ineq:trans0}, \eqref{ineq:trans1}, \eqref{ineq:trans2}, and \eqref{ineq:trans3}. 
Hence the difference in the $A$'s will gain a power of $\eta$ provided $k \neq \ell$.  
In the case $k = \ell$, one either gains regularity or gains $\nu^{1/3} t$; in the latter case, we will have to use the contribution of the dissipation to regain the regularity. 
In the case that $\abs{k-\ell,\eta- \ell \frac{1-e^{-\nu t}}{\nu}} \not\approx \abs{k,\eta}$, then the derivatives are effectively landing on $\rho$, rather than on $f$.
Therefore, 
\begin{align*}
\nu \abs{J_{\cC_h;2,3,b}} & \lesssim \nu e^{(\delta_1 - \delta)\nu^{1/3}t}\norm{\partial_v^t A(v^\alpha f)}_2 \brak{t}^{1-p/2}\norm{\abs{\grad_x}^{1/2}\rho}_{\cM^{\sigma,\delta}}\norm{f}_{F^{\sigma + 1/2 -j - 3,\delta_1}_{m+2+j}}  \\
& \quad + \nu e^{(\delta_1 - \delta)\nu^{1/3}t}\norm{\partial_v^t A(v^\alpha f)}_2 \brak{t}^{-\beta/2}\norm{\rho}_{\cM^{\beta,\delta}} \\ & \quad\quad \times \left(\norm{f}_{F^{\sigma-1/2-j,\delta_1}_{m+2+j}} + \nu^{1/3} t \left(\norm{f}_{\cD^{\sigma-1/2 - j,\delta_1}_{m+2+j}} + \norm{f}_{F^{\sigma+1/2-j,\delta_1}_{m+1+j}}\right)\right) \\ 
& \quad + \nu e^{\frac{1}{2}(\delta_1 - \delta)\nu^{1/3}t}\norm{A(v^\alpha f)}_2 \\ & \quad\quad \times \left(\brak{t}^{1-p/2} \norm{\abs{\grad}^{1/2}\rho}_{\cM^{\sigma,\delta}} \norm{f}_{F^{\sigma+1/2-3-j,\delta_1}_{m+1+j}} + \brak{t}^{-\beta/2}\norm{\rho}_{\cM^{\beta,\delta}}\norm{f}_{F^{\sigma+1/2-j,\delta_1}_{m+1+j}}\right), 
\end{align*}
which is consistent with Proposition \ref{prop:boot} $\eps$ and $\nu$ sufficiently small. 
This completes the improvement of the high norm estimate \eqref{boot:hineq}. 

\subsection{Estimates for the low norm: $\sigma+1/2-j$ with $\sigma+1/2 - \textup{Floor}(n/2 + 2) \geq j \geq 4$}
In this section we improve the lower norm (but higher velocity localization) controls on the distribution function, \eqref{boot:lowneq} to remove the polynomial growth in time. 
The proof is similar to the high norm estimate \eqref{boot:hineq} with some minor alterations, which we sketch briefly below.  
As above, for simplicity write, $\gamma= \sigma+1/2-j$ and
\begin{align*}
A'(t,k,\eta) & = A_{\gamma,\delta_1}\mathbf{1}_{k \neq 0} +  A_{\gamma,0}\mathbf{1}_{k = 0}. 
\end{align*}
Let $\alpha$ be a multi-index with $\abs{\alpha} \leq m+1/2-j$. 
As in \S\ref{sec:HiDiss}, we have
\begin{align}
\frac{1}{2}\frac{d}{dt}\left(e^{-2\abs{\alpha}\nu t} \norm{A' (v^\alpha f)}_{2}^2\right)  & = \delta_1 \nu^{1/3} e^{-2\abs{\alpha}\nu t} \norm{A'(v^\alpha f_{\neq})}_2^2 - e^{-2\abs{\alpha}\nu t} \norm{\sqrt{-\partial_t M M} A' (v^\alpha f_{\neq})}_{2}^2 \nonumber \\
& \quad - \nu \abs{\alpha} e^{-2\abs{\alpha}\nu t} \norm{A' (v^\alpha f)}_{2}^2 + e^{-2\abs{\alpha} \nu t} \brak{ A' (v^\alpha f), A' (v^\alpha \partial_t f)}_2\nonumber  \\
& := \delta_1 \nu^{1/3} e^{-2\abs{\alpha}\nu t} \norm{A'(v^\alpha f)}_2^2 + \cD_M + \cD_{\alpha} + \nu e^{-2\abs{\alpha}\nu t} \brak{ A' (v^\alpha f),  A' (v^\alpha \partial_{vv}^t f)}_2 \nonumber \\
 & \quad + e^{-2\abs{\alpha}\nu t} \brak{ A' (v^\alpha f), A' (v^\alpha E\cdot \partial_v^t \mu)}_2 + e^{-2\abs{\alpha}\nu t} \brak{ A' (v^\alpha f), A' (v^\alpha E\cdot \partial_v^t f)}_2 \nonumber \\ 
& \quad  + \nu e^{-2\abs{\alpha}\nu t} \brak{ A' (v^\alpha f), A' (v^\alpha C_{\mu})}_2 + \nu e^{-2\abs{\alpha}\nu t} \brak{ A' (v^\alpha f), A' (v^\alpha C_{h})}_2 \nonumber \\
& := \delta_1 \nu^{1/3} e^{-2\abs{\alpha}\nu t} \norm{A'(v^\alpha f)}_2^2 + \cD_M + \cD_\alpha + \nu \mathcal{D} + \mathcal{L} + \mathcal{N} + \nu J_\mu + \nu J_h \label{def:MedfEvo}
\end{align} 

\subsubsection{Dissipation term}
The $\mathcal{D}$ contributions can be treated as in \S\ref{sec:HiDiss} and hence the treatment is omitted for brevity. 

\subsubsection{Linear collisionless term} \label{sec:LoLCless}
Analogous to \S\ref{sec:HiLinCless}, we have, now using $\sigma- \gamma > 3$ and Lemma \ref{lem:YoungDisc}, 
\begin{align*}
\abs{\mathcal{L}} & \lesssim \brak{t}^{-p/2} e^{-2\abs{\alpha}\nu t}\sum_k \int A' \abs{D_\eta^\alpha \widehat{f}(k,\eta)} \left(\abs{kt} + \abs{\eta-k\frac{1-e^{-\nu t}}{\nu}}\right) e^{\delta \nu^{1/3}t} \brak{k,kt}^{\gamma} \abs{\widehat{\rho}(t,k)} \\ & \quad\quad \times \abs{k\widehat{W}(k)} \brak{\eta-k \frac{1-e^{-\nu t}}{\nu}}^{\gamma} \abs{D_\eta^\alpha \left(\left(\eta e^{\nu t} - k \frac{1-e^{\nu t}}{\nu}\right)\widehat{\mu}\left(e^{\nu t}\eta - k\frac{e^{\nu t} - 1}{\nu}\right) \right)} \dd \eta \\ 
& \lesssim \brak{t} e^{-\abs{\alpha}\nu t} \norm{A' (v^\alpha f)}_2 \norm{\rho(t)}_{\cM^{\gamma,\delta}} \\ 
& \lesssim \brak{t}^{-4} e^{-2\abs{\alpha}\nu t} \norm{A' (v^\alpha f)}_2^2 + \norm{\rho(t)}^2_{\cM^{\sigma,\delta}}, 
\end{align*}
which is consistent with Proposition \ref{prop:boot} by choosing $K_{Lj} \gg K_{HM0}$. 

\subsubsection{Nonlinear collisionless term}
As in \S\ref{sec:HiNLCless}, we subdivide this term analogously as 
\begin{align*}
\mathcal{N} = \mathcal{N}_{LH} + \mathcal{N}_{HL} + \mathcal{N}_{P}. 
\end{align*}
The $\mathcal{N}_{LH}$ and $\mathcal{N}_{P}$ contributions are treated analogously and hence are omitted for brevity. 
As in the collisionless treatment of \cite{BMM13}, the $\mathcal{N}_{HL}$ is treated via a variant of the argument in \S\ref{sec:LoLCless}: 
\begin{align*}
\abs{\mathcal{N}_{HL}} & \lesssim e^{- \delta \nu^{1/3}t} \brak{t}^{2} e^{-2\abs{\alpha}\nu t} \norm{A'(v^\alpha f)}_2^2 \norm{\rho}_{\cM^{\sigma-3,\delta}} \\ 
& \lesssim \eps\brak{t}^{-2} e^{-2\abs{\alpha}\nu t} \norm{A'(v^\alpha f)}_2^{2} + \eps^{-1}\norm{\rho(t)}_{\cM^{\sigma,\delta}}^2 e^{-2\abs{\alpha}\nu t} \norm{A'(v^\alpha f)}_2^{2},
\end{align*}
which is consistent with Proposition \ref{prop:boot} for $\eps$ sufficiently small. 

\subsubsection{$\cC_\mu$ collision term} \label{sec:LocCmu}
The $\cC_\mu$ contributions are treated by a combination of \S\ref{sec:HicCmu}  and the methods used on the collisionless term in \S\ref{sec:LoLCless}. 
Indeed, 
\begin{align*}
\nu\abs{\cC_{\mu}} & \lesssim \nu e^{-\abs{\alpha}\nu t} \norm{A'(v^\alpha f)}_2 \left(\brak{t} \norm{M_T(t)}_{\cM^{\sigma-3,\delta}} + \norm{M_1(t)}_{\cM^{\sigma-3,\delta}}\right) \\
& \lesssim \nu e^{-\abs{\alpha}\nu t} \norm{A'(v^\alpha f)}_2 \left(\brak{t}^{-2}\norm{M_T (t)}_{\cM^{\sigma,\delta}} + \norm{M_1(t)}_{\cM^{\sigma-3,\delta}}\right),
\end{align*}
which is consistent with Proposition \ref{prop:boot} by \eqref{ineq:MThi} and  $K_{Lj} \gg K_{HM0} + K_{HM1} + K_{HM2}$. 

\subsubsection{$\cC_h$ collision term} \label{sec:LocCh}
The analogues of $J_{\cC_h;j}$ for $j=1,3,4$ are treated as above in \S\ref{sec:HicCh} and are hence omitted for the sake of brevity. 
Next, we sketch the treatment of $J_{\cC_h;2}$ where, as above, the complication is  dealing with the potential loss of localization.
The treatment is a slight variant of that in \S\ref{sec:HicCh}. 
Hence, 
\begin{align*}
\nu J_{\cC_h;2} & = \nu e^{-2\abs{\alpha} \nu t - \nu t} \brak{A'(v^\alpha f), \rho(t,x+v\frac{1-e^{-\nu t}}{\nu})  v \cdot \partial_v^t A'(v^\alpha f) } \\
& \quad + n \nu e^{-2\abs{\alpha} \nu t} \brak{A'(v^\alpha f), A'(t,\grad) \left(\rho(t,x+v\frac{1-e^{-\nu t}}{\nu}) v^\alpha f \right)} \\ 
& \quad + \nu e^{-2\abs{\alpha} \nu t -\nu t} \big\langle A'(v^\alpha f), A'(t,\grad) \left(\rho(t, x+v\frac{1-e^{-\nu t}}{\nu}) v^\alpha v \cdot \partial_v^t f\right)
 \\ & \quad\quad - \rho(t,x+v\frac{1-e^{-\nu t}}{\nu})  v \cdot \partial_v^t A'(v^\alpha f) \big\rangle \\ 
& = \nu \sum_{j=1}^3 J_{\cC_h;2,j}.
\end{align*}
As in \S\ref{sec:HicCh}, the main issue is the commutator term $J_{\cC_h;2,3}$. 
As above, we write 
\begin{align*}
\nu J_{\cC_h;2,3} & = \nu e^{-2\abs{\alpha}\nu t -\nu t} \sum_{k,\ell} \int A D_\eta^\alpha \hat{f}(t,k,\eta) \hat{\rho}(t,\ell) \\
& \quad\quad \times\left(-\sum_{\beta \leq \alpha: \abs{\beta} = 1} n e^{\nu t} A D_\eta^{\alpha} \widehat{f}(t,k-\ell,\eta- \ell \frac{1-e^{-\nu t}}{\nu}) - (\grad_\eta A) D_\eta^\alpha \partial_v^t \hat{f}(t,k-\ell,\eta-\ell \frac{1-e^{-\nu t}}{\nu}) \right) \dd \eta \\ 
& \quad + \nu e^{-2\abs{\alpha}\nu t -\nu t} \sum_{k,\ell} \int A D_\eta^\alpha \hat{f}(t,k,\eta) \hat{\rho}(t,\ell) \\
& \quad\quad \times\left(A(t,k,\eta) - A(t,k-\ell,\eta-\ell \frac{1-e^{-\nu t}}{\ell})\right)D_\eta^\alpha \grad_\eta \cdot (\widehat{\partial_v^t f}) \dd \eta \\  
& = \nu J_{\cC_h;2,3,a} + J_{\cC_h;2,3,b}. 
\end{align*}
As above, $J_{\cC_h;2,3,a}$ is treated by methods used several times already, 
\begin{align*}
\nu \abs{J_{\cC_h;2,3,a}} & \lesssim \nu e^{- (\delta-\delta_1)\nu^{1/3}t} \norm{\rho(t)}_{\cM^{\beta,\delta}}\norm{A'(v^\alpha f)}_{2} \left(\norm{f}_{F^{\gamma,\delta_1}_{m+1}} + \norm{f}_{\cD^{\gamma,\delta_1}_{m+1}}\right).
\end{align*}
Turn next to $J_{\cC_h;2,3,b}$.
By the same argument as in \S\ref{sec:HicCh} involving \eqref{ineq:trans0}, \eqref{ineq:trans1}, \eqref{ineq:trans2}, and \eqref{ineq:trans3}, we deduce (note that we crucially use $\gamma -1  \geq n/2+1$ so that Sobolev embedding can be applied),
\begin{align*}
\nu \abs{J_{\cC_h;1,3,b}} & \lesssim \nu e^{(\delta_1 - \delta)\nu^{1/3}t}\norm{\partial_v^t A'(v^\alpha f)}_2 \brak{t}^{1/2}\norm{\abs{\grad_x}^{1/2}\rho}_{\cM^{\sigma-j,\delta}}\norm{f}_{F^{\gamma-1,\delta_1}_{m+2+j}}  \\
& \quad + \nu e^{(\delta_1 - \delta)\nu^{1/3}t}\norm{\partial_v^t A'(v^\alpha f)}_2 \brak{t}^{-\beta/2}\norm{\rho}_{\cM^{\beta,\delta}} \\ & \quad\quad \times \left(\norm{f}_{F^{\gamma-1,\delta_1}_{m+2+j}} + \nu^{1/3} t \left(\norm{f}_{\cD^{\gamma-1,\delta_1}_{m+2+j}} + \norm{f}_{F^{\gamma,\delta_1}_{m+1+j}}\right)\right) \\ 
& \quad + \nu e^{\frac{1}{2}(\delta_1 - \delta)\nu^{1/3}t}\norm{A'(v^\alpha f)}_2 \\ & \quad\quad \times \left(\brak{t}^{1/2} \norm{\abs{\grad_x}^{1/2}\rho}_{\cM^{\sigma-j,\delta}} \norm{f}_{F^{\gamma-1,\delta_1}_{m+1+j}} + \brak{t}^{-\beta/2}\norm{\rho}_{\cM^{\beta,\delta}}\norm{f}_{F^{\gamma,\delta_1}_{m+1+j}}\right).  
\end{align*}
If $\gamma-1 > n/2+2$, then we apply $\sigma-j \geq 3$ to add additional time-decay via $\norm{\abs{\grad_x}^{1/2}\rho}_{\cM^{\sigma-j,\delta}} \lesssim \brak{t}^{-3} \norm{\abs{\grad_x}^{1/2}\rho}_{\cM^{\sigma,\delta}}$ and hence this is consistent with Proposition \ref{prop:boot} for $\delta > \delta_1$ and $\eps$ and $\nu$ sufficiently small.
If $\gamma-1 = n/2+2$, we use that $\sigma-j = n/2+2$ and hence for $\beta$ sufficiently large  (depending on $n$),
\begin{align*}
\norm{\abs{\grad}_x^{1/2} \rho}_{\cM^{\sigma-j,\delta}} \lesssim \brak{t}^{-\beta/2}\norm{\rho}_{\cM^{\beta,\delta}} \lesssim \brak{t}^{-2(n-4)}\norm{\rho}_{\cM^{\beta,\delta}}. 
\end{align*}
Therefore, for $\beta/2 > 4 + n/2$ we may counterbalance the growth coming from \eqref{ineq:cHF}, and hence we may apply Lemma \ref{lem:cHvsFcD} together with the bootstrap control \eqref{boot:lowhimoment}.
Hence, for $\eps$ and $\nu$ sufficiently small, this contribution is also consistent with Proposition \ref{prop:boot}. 

\subsection{$\cH$ high moment estimate} \label{sec:HiMoment}
In this section we improve \eqref{boot:lowhimoment}, hence deducing the best velocity localizations. For example, this makes it possible to close the loss of velocity localization in the treatments of the nonlinear collision terms in \S\ref{sec:HicCh} and \S\ref{sec:LocCh}. 

Let $\alpha,\gamma$ be multi-indices such that $\abs{\alpha} + \abs{\gamma} \leq \textup{Ceil}(n/2+2)$.  
In the ensuing computation, denote $m'-\abs{\alpha} = q$ and compute the following 
\begin{align*}
\frac{1}{2}\frac{d}{dt}\left(\frac{1}{\brak{t}^{2\theta\abs{\gamma}}}\norm{D^\alpha_x D^\gamma_v h}_{L^2_{q}}^2\right) & = -\theta\abs{\gamma}\frac{t}{\brak{t}^{2\theta\abs{\gamma} + 2}} \norm{D^\alpha_x D^\gamma_v h}_{L^2_{q}}^2 - \frac{1}{\brak{t}^{2\theta\abs{\gamma}}} \brak{D^\alpha_x D^\gamma_v h,D^\alpha_x D^\gamma_v (v\cdot \grad_x h)}_{L^2_q} \\ & \quad - \frac{1}{\brak{t}^{2\theta\abs{\gamma}}}\brak{D^\alpha_x D^\gamma_v h,D^\alpha_x D^\gamma_v(E\cdot \grad_v h)}_{L^2_q} \\ & \quad + \nu\frac{1}{\brak{t}^{2\theta\abs{\gamma}}}\brak{D^\alpha_x D^\gamma_v h,D^\alpha_x D^\gamma_v\left((1+\rho)\left((1+T)\Delta_vh + \grad_v\cdot((v-u) h)\right)\right)}_{L^2_q} \\ & \quad + \nu\frac{1}{\brak{t}^{2\abs{\gamma}}}\brak{D^\alpha_x D^\gamma_v h,D^\alpha_x D^\gamma_v \cC_\mu}_{L^2_q}  \\ 
& = -\theta\abs{\gamma}\frac{t}{\brak{t}^{2\theta\abs{\gamma} + 2}} \norm{D^\alpha_x D^\gamma_v h}_{L^2_{q}}^2 \\ & \quad -\sum_{\gamma' \leq \alpha: \abs{\gamma'} = \abs{\alpha} - 1}\frac{1}{\brak{t}^{2\theta\abs{\gamma}}} \brak{D^\alpha_x D^\gamma_v h, (D^{\gamma-\gamma'}_v v)\cdot \grad_x D^{\gamma'}_v D^\alpha_x h)}_{L^2_q} \\& \quad - \frac{1}{\brak{t}^{2\theta\abs{\gamma}}}\brak{D^\gamma_v D^\alpha_x h,D^\alpha_x D^\gamma_v(E\cdot \grad_v h)}_{L^2_q} \\ & \quad + \nu\frac{1}{\brak{t}^{2\theta\abs{\gamma}}}\brak{D^\alpha_x D^\gamma_v h,D^\alpha_x D^\gamma_v\left((1+\rho)\left((1+T)\Delta_vh + \grad_v\cdot((v-u) h)\right)\right)}_{L^2_q} \\ & \quad + \nu\frac{1}{\brak{t}^{2\theta\abs{\gamma}}}\brak{D^\alpha_x D^\gamma_v h,D^\alpha_x D^\gamma_v \cC_\mu}_{L^2_q}. 
\end{align*}
The transport term is estimated via (note that this term vanishes if $\gamma = 0$), 
\begin{align*}
\frac{1}{\brak{t}^{2\theta\abs{\gamma}}}\abs{\sum_{\gamma' \leq \gamma: \abs{\gamma'} = \abs{\gamma} - 1} \brak{D^\alpha_x D^\gamma_v h, (D^{\gamma-\gamma'}_v v)\cdot \grad_x D^{\gamma'}_v D^\alpha_x h)}_{L^2_q}} & \\ & \hspace{-6cm} \lesssim \sum_{\gamma' \leq \gamma: \abs{\gamma'} = \abs{\gamma} - 1} \frac{1}{\brak{t}^{\theta\abs{\gamma}+\theta}} \norm{D^\alpha_x D^\gamma_v h}_{L^2_q} \frac{1}{\brak{t}^{\theta\abs{\gamma'}}}\norm{\grad_x D^{\gamma'}_v D^\alpha_x h}_{L^2_q} \\ 
& \hspace{-6cm} \lesssim \frac{1}{\brak{t}^{\theta}}\norm{h(t)}_{\cH_q}^2,
\end{align*}
which is consistent choosing $\theta >  1$. 
The electric field term is estimated via distributing the derivatives, integrating by parts on the leading derivative term, and applying Sobolev embedding 
\begin{align*}
\abs{\frac{1}{\brak{t}^{2\theta\abs{\gamma}}}\brak{D^\gamma_v D^\alpha_x h,D^\alpha_x D^\gamma_v(E\cdot \grad_v h)}_{L^2_q}} & \\
& \hspace{-6cm}  \lesssim \frac{\norm{E(t)}_{L^\infty}}{\brak{t}^{2\theta\abs{\gamma}}}\norm{D^\gamma_v D^\alpha_x h}_{L^2_{q-1/2}}^2 + \sum_{\alpha' \leq \alpha: \abs{\alpha'} \leq \abs{\alpha} -1 } \abs{\frac{1}{\brak{t}^{2\theta\abs{\gamma}}}\brak{D^\gamma_v D^\alpha_x h,D^{\alpha-\alpha'}_x E \cdot \grad_v D^{\alpha'}_x D^\gamma_v h}_{L^2_q}} \\ 
& \hspace{-6cm} \lesssim \brak{t}^{-\beta/2} \norm{\rho}_{\cM^{\beta,\delta}}\norm{h}_{\cH_q}^2.   
\end{align*}
The source terms from collisions with $\mu$ are easily handled via Sobolev embedding and the higher regularity and decay available on $\rho$ and $T$,  
\begin{align*}
\frac{\nu}{\brak{t}^{2\theta\abs{\gamma}}}\brak{D^\gamma_v D^\alpha_x h,D^\alpha_x D^\gamma_v\cC_\mu}_{L^2_q} & = \frac{\nu}{\brak{t}^{2\theta\abs{\gamma}}}\brak{D^\gamma_v D^\alpha_x h,D^\alpha_x D^\gamma_v\left((1 + \rho)\left( T\Delta_v \mu - u \cdot \grad_v\mu \right)\right)}_{L^2_q} \\ 
& \lesssim \nu \eps \brak{t}^{-\beta/2 - 2\theta\abs{\gamma}}\norm{D^\gamma_v D^\alpha_x h}_{L^2_q}, 
\end{align*}
which is consistent with Proposition \ref{prop:boot} for $\nu$ sufficiently small. 

Turn finally to the collision terms; by choosing $\eps$ sufficiently small and using the controls $\norm{\rho}_{\cM^{\beta,0}} + \norm{T}_{\cM^{\beta,0}} \lesssim \eps$, 
\begin{align*}
\nu\frac{1}{\brak{t}^{2\theta\abs{\gamma}}}\brak{D^\alpha_x D^\gamma_v h,D^\alpha_x D^\gamma_v\left((1+\rho)\left((1+T)\Delta_vh + \grad_v\cdot((v-u) h)\right)\right)}_{L^2_q} & \\
& \hspace{-9cm} = \nu\frac{1}{\brak{t}^{2\theta\abs{\gamma}}}\brak{D^\alpha_x D^\gamma_v h,D^\alpha_x D^\gamma_v\left((1+\rho)(1+T)\Delta_vh\right)}_{L^2_q} \\
& \hspace{-9cm} \quad + n\nu\frac{1}{\brak{t}^{2\theta\abs{\gamma}}}\brak{D^\alpha_x D^\gamma_v h,D^\alpha_x D^\gamma_v \left((1+\rho)h\right)}_{L^2_q} + \nu\frac{1}{\brak{t}^{2\theta\abs{\gamma}}}\brak{D^\alpha_x D^\gamma_v h,D^\alpha_x D^\gamma_v \left((1+\rho)(v-u)\cdot \grad_v h\right)}_{L^2_q} \\ 
& \hspace{-9cm} \leq -\frac{\nu}{\brak{t}^{2\theta\abs{\gamma}}}\norm{\sqrt{(1+\rho)(1+T)}\grad_v D^\alpha_x D^\gamma_v h}_{L^2_q}^2  \\
& \hspace{-9cm} \quad + \sum_{\alpha' \leq \alpha: \abs{\alpha'} \leq \abs{\alpha}-1}\frac{C\nu \eps}{\brak{t}^{2\theta\abs{\gamma}}}\norm{\grad_v D^{\alpha'}_x D^{\gamma}_v h}_{L^2_q} \norm{ \grad_v D^\alpha_x D^\gamma_v h}_{L^2_q} \\  
& \hspace{-9cm} \quad + \sum_{\alpha' \leq \alpha}\frac{C\nu}{\brak{t}^{2\theta\abs{\gamma}}}\norm{\grad_v D^{\alpha'}_x D^{\gamma}_v h}_{L^2_q} \norm{ D^\alpha_x D^\gamma_v h}_{L^2_q} \\ 
& \hspace{-9cm} \quad + C\nu \norm{h}_{\cH}^2 + \nu\sum_{\alpha'\leq \alpha: \abs{\alpha'} \leq \abs{\alpha} - 1}\frac{1}{\brak{t}^{2\theta\abs{\gamma}}}\brak{D^\alpha_x D^\gamma_v h, (D^{\alpha-\alpha'}_x\rho) v \cdot \grad_v D_x^{\alpha'} D^\gamma_v h }_{L^2_q}. 
\end{align*}
Using that higher spatial derivatives are being measured in a lower moment, for $\eps$ sufficiently small, we have for some constant $C>0$, 
\begin{align*}
\nu\frac{1}{\brak{t}^{2\theta\abs{\gamma}}}\brak{D^\alpha_x D^\gamma_v h,D^\alpha_x D^\gamma_v\left((1+\rho)\left((1+T)\Delta_vh + \grad_v\cdot((v-u) h)\right)\right)}_{L^2_q} & \leq \\
& \hspace{-9cm} -\frac{1}{2}\frac{\nu}{\brak{t}^{2\theta\abs{\gamma}}}\norm{\grad_v D^\alpha_x D^\gamma_v h}_{L^2_q}^2 + \frac{1}{10 \brak{t}^{2\theta\abs{\gamma}}} \sum_{\alpha' \leq \alpha} \norm{\grad_v D^{\alpha'} D^\gamma_v h}_{L^2_q}^2 \\
& \hspace{-9cm} + \left(C\nu + C\eps \brak{t}^{-\beta/2}\right) \norm{h(t)}_{\cH}^2.
\end{align*}
This is hence consistent with Proposition \ref{prop:boot} provided we choose $\eps$ small and suitably adjust the bootstrap constants. 
This completes the improvement of \eqref{boot:lowhimoment}. 

\section{Spatially homogeneous thermalization estimate} \label{sec:Zeromode}
In this section we obtain the decay estimate on $h_0$ in $H^{s}_m$. 
Certain aspects of the ensuing calculations are reminiscent of arguments used in \cite{BGM15I,BGM15II,BGM15III} to control the nonlinear influence of $x$-dependent modes.
Of crucial importance is the following spectral gap estimate on $L$ in $L^2_m$ obtained by Gallay and Wayne in \cite{GallayWayne02}. 
Only the integer case is treated in \cite{GallayWayne02}; the fractional exponent case follows by Hilbert space interpolation. 

\begin{theorem}[Properties of $e^{tL}$ (from \cite{GallayWayne02})] \label{lem:specgap}
Let $m > n/2$ be an integer and $s \geq 0$. 
Define 
\begin{align*}
a(t) = 1-e^{-t}. 
\end{align*}
There holds
\begin{itemize} 
\item For all $g \in H^s_m$, and multi-indices $\alpha$
\begin{align*}
\norm{D^\alpha e^{\nu tL} g}_{H^s_m(\Real^n)} \lesssim \frac{1}{a(t)^{\abs{\alpha}/2}} \norm{f}_{H^s_m(\Real^n)}. 
\end{align*}
\item Suppose further that $g = g(v)$ is mean-zero. Then, 
\begin{align*}
\norm{D^\alpha e^{\nu tL} g}_{H^s_m(\Real^n)} \lesssim \frac{e^{-\nu t}}{a(t)^{\abs{\alpha}/2}} \norm{g}_{H^s_m(\Real^n)}. 
\end{align*}
\end{itemize} 
\end{theorem}
We then use Duhamel's formula to write the evolution of the zero mode (note that the linear term $E \cdot \grad_v \mu$ vanishes): 
\begin{align*}
h_0(t) & = e^{\nu t L} h_{in,0} - \int_0^t e^{\nu(t-\tau)L}\left(E\cdot \grad_v h\right)_0 \dd \tau +  \int_0^t e^{\nu(t-\tau)L}\left(\nu\cC_\mu + \nu \cC_h \right)_0 \dd \tau \\ 
& = e^{\nu t L} h_{in, 0} + \sum_{j=1}^3 N_j. 
\end{align*}
First, consider the collisionless nonlinear term $N_1$. 
By Theorem \ref{lem:specgap}, (note that $E \cdot \grad_v h = \grad_v \cdot (Eh)$ as the electric field is independent of $v$), there holds 
\begin{align}
\norm{N_1}_{H^s_m} \lesssim \int_0^t \frac{e^{-\nu (t-\tau)}}{a(t-\tau)^{1/2}} \norm{\grad_v \cdot \left(E h\right)_0}_{H^{s-1}_m} \dd \tau. \label{ineq:N1bd}
\end{align}
Note, 
\begin{align*}
\widehat{\left(E\cdot \grad_v h\right)}_0(t,\eta) & = \sum_{k \in \Integer^n_\ast} \widehat{E}(t,-k) \cdot i\eta\widehat{h}(t,k,\eta) \\ 
& = \sum_{k \in \Integer^n_\ast} \widehat{E}(t,-k) \cdot i \eta \widehat{f}(t,k,e^{-\nu t} \eta + k\frac{1-e^{-\nu t}}{\nu}).
\end{align*}
Therefore, 
\begin{align}
\abs{\brak{\eta}^{s-1}\widehat{\left(E\cdot \grad_v h\right)}_0(t,\eta)} & \lesssim e^{s \nu t}\sum_{k \in \Integer^n_\ast} \left(\brak{k\frac{1-e^{-\nu t}}{\nu}}^{s-1} + \brak{e^{-\nu t}\eta + k\frac{1-e^{-\nu t}}{\nu}}^{s-1} \right) \nonumber \\ & \quad\quad \times \abs{\widehat{E}(t,-k)} \abs{\eta}\abs{\widehat{f}(t,k,e^{-\nu t} \eta + k\frac{1-e^{-\nu t}}{\nu})}. \label{ineq:N1stp1}
\end{align}
Note that $k \neq 0$, and hence we gain $e^{-\delta_1 \nu^{1/3}t}$ decay from $f$ (as well as the decay from $\rho$). 
Hence, taking derivatives in $\eta$ and applying Lemma \ref{lem:YoungDisc} then gives (using $\nu$ small to absorb the exponential growth using the rapid decay $\nu^{1/3} t$ decay) by the bootstrap hypotheses and $\beta < s < \sigma+1/2$, 
\begin{align*}
\norm{\left(E \cdot \grad_v h\right)_0}_{H^{s-1}_m} & \lesssim e^{-\delta_1\nu^{1/3}\tau}\left(\norm{\rho(\tau)}_{\cM^{s-1,\delta}} \norm{f(\tau)}_{F^{\beta,\delta_1}_m} + \brak{\tau}^{-\beta/2}\norm{\rho(\tau)}_{\cM^{\beta,\delta}} \norm{f(\tau)}_{F^{s,\delta_1}_m}\right) \\ 
& \lesssim e^{-\delta_1\nu^{1/3}\tau}\left(\norm{\rho(\tau)}_{\cM^{\sigma,\delta}}^{\theta} \norm{\rho(\tau)}_{\cM^{\beta,\delta}}^{1-\theta} \norm{f(\tau)}_{F^{\beta,\delta_1}_m} + \brak{\tau}^{-\beta/2} \norm{\rho(\tau)}_{\cM^{\beta,\delta}} \norm{f(\tau)}_{F^{s,\delta_1}_m}\right) \\ 
& \lesssim e^{-\frac{1}{2}\delta_1\nu^{1/3}\tau}\left(\eps^{2-\theta}\norm{\rho(\tau)}_{\cM^{\sigma,\delta}}^{\theta} + \brak{\tau}^{-\beta/2} \eps^2\right), 
\end{align*}
where $\theta \in (0,1)$ such that  $s-1 = \theta \sigma + (1-\theta) \beta$. 
Therefore, \eqref{ineq:N1bd} gives using $\sigma-s > 0$ 
\begin{align*}
\norm{N_1}_{H^s_m} & \lesssim  e^{-\nu t} \int_0^t \frac{e^{\nu \tau} e^{-\frac{1}{2}\delta_1 \nu^{1/3} \tau}}{a(t-\tau)^{1/2}} \left(\eps^{2-\theta}\norm{\rho(\tau)}_{\cM^{\sigma,\delta}}^{\theta} + \brak{\tau}^{-\beta/2} \eps^2\right)\dd \tau \\\ 
& \lesssim e^{-\nu t} \eps^{2-\theta}\left(\int_0^t \frac{e^{-\frac{1}{4(2-\theta')}\delta_1 \nu^{1/3}\tau}}{a(t-\tau)^{\frac{1}{2-\theta}}} \dd \tau \right)^{\frac{2-\theta}{2}} \norm{\rho}^{\theta}_{L^2_t \cM^{\sigma,\delta}} + \eps^2 e^{-\nu t} \\ 
& \lesssim e^{-\nu t} \eps^2\nu^{-(2-\theta)/6}, 
\end{align*}
which is consistent with Proposition \ref{prop:boot} by  $\eps \nu^{-1/3}$ sufficiently small (note $(2-\theta)/6 < 1/3$).  

Turn next to the collisional term $\cC_\mu$. 
Due to the $\grad_v \cdot$, it follows that $\left(\cC_\mu\right)_0$ is average zero, and hence by Theorem \ref{lem:specgap}, there holds 
\begin{align*}
\norm{N_2}_{H^s_m} & \lesssim \int_0^t \frac{e^{-\nu (t-\tau)}}{a(t-\tau)^{1/2}} \norm{\left(\cC_\mu\right)_0}_{H^{s-1}_m} \dd \tau. 
\end{align*}
First, note that (using that $(M_1)_0=0$ by conservation of momentum),
\begin{align*}
\nu \widehat{\left(\cC_{\mu}\right)_0}(t,\eta) & = \nu \widehat{M_T}(t,0) \widehat{\Delta_v \mu}(\eta).  
\end{align*}
Therefore, applying $D_\eta^\alpha \left(\brak{\eta}^{s} \cdot \right)$ and integrating (applying also the product rule), 
\begin{align*} 
\nu \norm{\left(\cC_{\mu}\right)_0}_{H^{s-1}_m} & \lesssim \nu \abs{(M_T)_0(t)}.  
\end{align*}
From this estimate, it follows from the bootstrap hypotheses (and Lemma \eqref{lem:uTctrls})  by an argument similar to that applied on $N_1$ that 
\begin{align*}
\norm{N_2}_{H^s_m} & \lesssim  \nu e^{-\nu t} \eps,
\end{align*}
which is consistent with Proposition \ref{prop:boot}. 

Turn the nonlinear collisional term $\cC_h$. 
As above, the $\grad_v \cdot$, ensures  $\left(\cC_h\right)_0$ is average zero, and hence by Theorem \ref{lem:specgap}, there holds 
\begin{align*}
\norm{N_3}_{H^s_m} & \lesssim \int_0^t \frac{e^{-\nu (t-\tau)}}{a(t-\tau)^{1/2}} \norm{\left(\cC_h\right)_0}_{H^{s-1}_m} \dd \tau. 
\end{align*}
Expanding the collision term as usual, 
\begin{align*}
\nu \widehat{\left(\cC_{h}\right)_0}(t,\eta) & = \nu \sum_{k \in \Integer^n_\ast} \hat{\rho}(t,-k) \left(-\abs{\eta}^2 \widehat{h}(t,k,\eta) + i\eta \cdot \grad_\eta \hat{h}(t,k,\eta) \right) \\ 
& \quad + \nu\sum_{k \in \Integer^n}\left(-\abs{\eta}^2 \widehat{M_T}(t,-k) \hat{h}(t,k,\eta) - \widehat{M_1}(t,-k) \cdot i\eta \hat{h}(t,k,\eta)\right) \\ 
& = \sum_{\ell=1}^2 J_\ell. 
\end{align*}
Consider first $J_1$; as in \eqref{ineq:N1stp1} there holds 
\begin{align}
\brak{\eta}^{s-1}\abs{J_1} & \lesssim  \nu e^{s\nu t} \sum_{k \in \Integer^n_\ast} \left(\brak{k\frac{1-e^{-\nu t}}{\nu}}^{s-1} + \brak{e^{-\nu t}\eta + k\frac{1-e^{-\nu t}}{\nu}}^{s-1} \right) \nonumber \\ & \quad\quad \times \abs{\hat{\rho}(t,-k)}  \left(-\abs{\eta}^2\abs{\widehat{f}(t,k,e^{-\nu t}\eta + k\frac{1-e^{-\nu t}}{\nu})} + \abs{\eta \cdot \grad_\eta \hat{f}(t,k,e^{-\nu t} \eta +k\frac{1-e^{-\nu t}}{\nu} )} \right). \label{ineq:J1bd} 
\end{align}
Using also that
\begin{align*}
\abs{\eta} & \lesssim \brak{t}e^{\nu t}\brak{k\frac{1-e^{-\nu t}}{\nu}} + e^{\nu t} \brak{e^{-\nu t}\eta + k\frac{1-e^{-\nu t}}{\nu}} \\ 
\abs{\eta} \abs{\hat{f}(t,k,e^{-\nu t}\eta + k\frac{1-e^{-\nu t}}{\nu})} & = e^{\nu t}\abs{\widehat{\partial_v^t f}(t,k,e^{-\nu t}\eta + k\frac{1-e^{-\nu t}}{\nu})},
\end{align*}
it follows from \eqref{ineq:J1bd}, Lemma \ref{lem:YoungDisc} (after differentiating with respect to $\eta$), and choosing $\nu$ small to absorb exponential growth, 
\begin{align*}
\norm{J_1}_{H^{s-1}_m} & \lesssim \nu e^{-\delta_1 \nu^{1/3}t} \left(\norm{\rho}_{\cM^{s-1,\delta}}\brak{t}^{2}\norm{f}_{F^{\beta,\delta_1}_{m+1}} + \brak{t}^{-\beta/2}\norm{\rho}_{\cM^{\beta,\delta}}\left(\norm{f}_{\cD^{s,\delta_1}_{m+1}} + \norm{f}_{F^{s,\delta_1}_{m+1}}\right)\right) \\ 
& \lesssim \nu^{1/3} e^{-\frac{1}{2}\delta_1 \nu^{1/3}t}\norm{\rho}_{\cM^{s-1,\delta}} \norm{f}_{F^{\beta,\delta_1}_{m+1}} + \nu^{1/2} e^{-\delta_1 \nu^{1/3}t} \brak{t}^{-\beta/2}\norm{\rho}_{\cM^{\beta,\delta}}\left(\nu^{1/2}\norm{f}_{\cD^{s,\delta_1}_{m+1}} + \nu^{1/2}\norm{f}_{F^{s,\delta_1}_{m+1}}\right). 
\end{align*}
Next, apply the following interpolation for $s = (1-\theta')\beta + \theta'(\sigma+1/2)$, 
\begin{align*}
\norm{f}_{\cD^{s,\delta_1}_{m+1}} & \lesssim \norm{f}_{\cD^{\beta,\delta_1}_{m+1}}^{1-\theta'}  \norm{f}^{\theta'}_{\cD^{\sigma+1/2,\delta_1}_{m+1}} \lesssim e^{(1-\theta')\nu t} \brak{t}^{1-\theta'} \norm{f}_{\cD^{\beta+1,\delta_1}_{m+1}}^{1-\theta'}  \norm{f}^{\theta'}_{\cD^{\sigma+1/2,\delta_1}_{m+1}} \lesssim e^{(1-\theta')\nu t} \brak{t}^{1-\theta'} \eps^{1-\theta'} \norm{f}^{\theta'}_{\cD^{\sigma+1/2,\delta_1}_{m+1}}. 
\end{align*}
Using that $\theta' \in (0,1)$ and that there are extra powers of $t$ available, we see that for $\beta > p+2$, there holds by 
arguments similar to the treatment of $N_1$ above,  
\begin{align*}
\int_0^t \frac{e^{-\nu (t-\tau)}}{a(t-\tau)^{1/2}} \norm{J_1}_{H^{s-1}_m} \dd \tau& \lesssim \eps \int_0^t \frac{e^{-\nu (t-\tau)}}{a(t-\tau)^{1/2}} \nu^{1/3} e^{-\frac{1}{2}\delta_1 \nu^{1/3}\tau}\norm{\rho(\tau)}_{\cM^{s-1,\delta}} \dd \tau \\ 
& \quad + \eps \int_0^t \frac{e^{-\nu (t-\tau)}}{a(t-\tau)^{1/2}} \nu e^{-\delta_1 \nu^{1/3}\tau} \brak{\tau}^{-\beta/2} \left(\norm{f(\tau)}_{\cD^{s,\delta_1}_{m+1}} + \norm{f(\tau)}_{F^{s,\delta_1}_{m+1}}\right) \dd \tau \\
& \lesssim  \eps^2 e^{-\nu t}\nu^{\frac{1}{3}-\frac{(2-\theta)}{6}} + \eps^2 e^{-\nu t} \nu^{\frac{1}{2} - \frac{2-\theta'}{6}},
\end{align*}
which is consistent with Proposition \ref{prop:boot} by choosing $\nu$ and $\eps$ sufficiently small. 
This completes the contribution from $J_1$.

Turn to $J_2$, which is expanded via
\begin{align*}
J_2 & = -\nu\widehat{M_T}(t,0)\abs{\eta}^2\hat{h}(t,0,\eta) + \nu\sum_{k \in \Integer^n_\ast}\left(-\abs{\eta}^2 \widehat{M_T}(t,-k) \hat{h}(t,k,\eta) - \widehat{M_1}(t,-k) \cdot i\eta \hat{h}(t,k,\eta)\right) \\
& = J_{20} + J_{2\neq}. 
\end{align*}
The contribution of $J_{2\neq}$ is similar to $J_1$ and is hence omitted. 
By \eqref{def:hfrelat}, for $\theta \in (0,1)$ such that  $s = \theta(\sigma+1/2) + (1-\theta) \beta$, there holds 
\begin{align*}
\norm{\grad_v h_0}_{H^s_m} \lesssim e^{(s + 1 + n)\nu t}\norm{\partial_v^t f_0}_{H^{s}_m} & \lesssim e^{(s + 1 + n)\nu t}\norm{\partial_v^t f_0}_{H^{\sigma+1/2}_m}^{\theta}\norm{\partial_v^t f_0}_{H^{\beta}_m}^{1-\theta} \\
& \lesssim e^{(s+2+m+n)\nu t}\left(\norm{f_0}_{\cD^{\sigma+1/2}_m}^{\theta} + \norm{f_0}_{F^{\sigma+1/2}_m}^{\theta}\right)\norm{f}_{F^{\beta+1}_m}^{1-\theta}. 
\end{align*}
Therefore, for $J_{20}$, we have by Lemma \ref{lem:uTctrls} and H\"older's inequality (for $\beta$ large relative to $p$), 
\begin{align*}
\int_0^t \frac{e^{-\nu (t-\tau)}}{a(t-\tau)^{1/2}} \norm{J_{20}(\tau)}_{H^{s-1}_m} \dd \tau & \lesssim  \nu e^{-\nu t}\int_0^t \frac{e^{(s+m+n+2)\nu \tau} e^{-\delta \nu^{1/3}\tau}}{a(t-\tau)^{1/2} \brak{\tau}^{\beta}} \left(\brak{\tau}^\beta e^{\delta \nu^{1/3}\tau}\abs{(M_T)_0(\tau)}\right)  \\
& \quad\quad \times \left(\norm{f_0}_{\cD^{\sigma+1/2}_m}^{\theta} + \norm{f_0}_{F^{\sigma+1/2}_m}^{\theta}\right)\norm{f}_{F^{\beta+1}_m}^{1-\theta} \dd \tau \\
& \lesssim \nu e^{-\nu t} \eps \int_0^t \frac{1}{a(t-\tau)^{1/2} \brak{\tau}^{\beta}}\left(\norm{f_0}_{\cD^{\sigma+1/2}_m}^{\theta} + \norm{f_0}_{F^{\sigma+1/2}_m}^{\theta}\right)\norm{f}_{F^{\beta+1}_m}^{1-\theta} \dd \tau \\
& \lesssim e^{-\nu t} \eps^{2-\theta} \left(\nu \int_0^t \brak{\tau}^{-p/2}\norm{f(\tau)}^2_{\cD^{\sigma+1/2}_{m}} + \brak{\tau}^{-p/2-2}\norm{f(\tau)}^2_{F^{\sigma+1/2}_{m}} \dd \tau\right)^{\theta/2}, 
\end{align*}
which is consistent with Proposition \ref{prop:boot}. 
This completes the improvement of estimate \eqref{boot:therm}.

\appendix
\section{Properties of the dissipation multiplier} \label{sec:propDissM}
The properties of the multiplier $M$ are summarized in the next lemma. Properties (a) and (b) are essentially trivial, whereas properties (c) and (d) are not. 
\begin{lemma}[Properties of $M$] \label{lem:propM}
\begin{itemize}
\item[(a)] There is a universal constant $c_m$ (in particular, uniform in $t,k,\eta,$ and $\nu$) such that $c_m < M \leq 1$.
\item[(b)] For $k \neq 0$ there holds, 
\begin{align}
\nu^{1/3} \lesssim \partial_t M(t,k,\eta) + \nu e^{2\nu t} \abs{\eta - k\frac{1-e^{-\nu t}}{\nu}}^2. \label{ineq:dotMMAbd}
\end{align}
\item[(c)] There holds (uniformly in $k,\ell,\eta,\nu$ and $t$)
\begin{align}
\abs{1 - \frac{M(t,\ell,\xi}{M(t,k,\eta)}} & \lesssim \brak{k-\ell,\eta-\xi}\left( \frac{\brak{t}^2}{\nu^{1/3} \max\left(\brak{\eta},\brak{k}\right)} \right). \label{ineq:Mtrans}
\end{align}
\item[(d)] There holds for $\alpha \in \Naturals^n$ such that $\abs{\alpha} \geq 1$, 
\begin{align}
\abs{D_\eta^\alpha M} & \lesssim e^{\abs{\alpha}\nu t} \nu^{\abs{\alpha}/3}. \label{ineq:detaM}
\end{align}
\end{itemize}
\end{lemma}
\begin{remark}
That \eqref{ineq:Mtrans} holds even when $\ell$ or $k$ is zero is crucial to the proof. 
\end{remark}  
\begin{proof}
\textbf{Proof of property (a):} \\ First notice that $M \equiv 1$ if $\nu\eta = k$. If this is not the case, then we make the change of variables $s = \abs{\eta - \frac{k}{\nu}} e^{\nu \tau}$, $ds = \abs{\nu \eta  -k} e^{\nu \tau} d\tau$,  
\begin{align*}
-\log M(t,k,\eta) & = \int_0^t \frac{\nu^{1/3}}{1 + \nu^{2/3} e^{2\nu \tau}\abs{\eta - k \frac{1-e^{-\nu \tau}}{\nu}}^2} \frac{e^{2\nu \tau} \abs{\nu \eta - k}^2}{\brak{e^{\nu \tau} (\nu \eta - k)}^2}  d\tau \\
& = \int_{\abs{\eta - \frac{k}{\nu}}}^{\abs{\eta - \frac{k}{\nu}}  e^{\nu t}} \frac{\nu^{1/3}}{1 + \nu^{2/3} \abs{s \frac{\eta - k \nu^{-1}}{\abs{\eta - k \nu^{-1}}} + \frac{k}{\nu}}^2} \frac{ \nu s  }{\brak{\nu s}^2}  ds  \\
& \lesssim 1. 
\end{align*}
This implies property (a).

\textbf{Proof of property (b):} \\
Turn next to property (b).  
We seperate into two cases.

\noindent
\textbf{The case $\abs{\nu \eta - k} e^{\nu \tau} > \frac{1}{2}$:} This case follows by separately considering  $\nu e^{2\nu \tau}\abs{\eta - k \frac{1-e^{-\nu \tau}}{\nu}}^2 < \nu^{1/3}$ and  $\nu e^{2\nu \tau}\abs{\eta - k \frac{1-e^{-\nu \tau}}{\nu}}^2 \geq \nu^{1/3}$. 

\noindent
\textbf{The case $\abs{\nu \eta - k} e^{\nu \tau} \leq \frac{1}{2}$:} In this case, note that
\begin{align}
\nu e^{2\nu \tau}\abs{\eta - k\frac{1- e^{-\nu \tau}}{\nu} }^2 = \frac{1}{\nu}\abs{e^{\nu \tau}(\nu \eta - k) + k}^2 \gtrsim \frac{\abs{k}^2}{\nu},  
\end{align}
and hence the dissipation term dominates. 

\textbf{Proof of property (c):} \\ 
Finally, turn next to \eqref{ineq:Mtrans}, the most difficult cases.

\noindent
\textbf{The case $k = 0$ or $\ell = 0$:} By definition of $M$, \eqref{ineq:Mtrans} is trivial if $k = \ell = 0$.
The two remaining cases are essentially equivalent, hence without loss of generality, assume that $k = 0$ and $\ell \neq 0$. 
By part (a) and $\abs{e^x - 1}\leq x e^x$,
\begin{align*}
\abs{1 - M(t,\ell,\xi)} \lesssim \int_0^t \frac{e^{2\nu \tau} \abs{\nu \xi - \ell }^2}{\brak{e^{\nu \tau} (\nu \xi - \ell)}^2} \frac{\nu^{1/3}}{1 + \nu^{2/3} e^{2\nu \tau} \abs{ \xi - \ell \frac{1- e^{-\nu \tau}}{\nu}}^2}  d\tau.  
\end{align*}
First, if $\abs{\xi} > 2\abs{\ell t}$, then by $(1-e^{-\nu \tau})/\nu \leq \tau \leq t$ then,
\begin{align*}
\abs{1 - M(t,\ell,\xi)} & \lesssim \int_0^t \frac{\nu^{1/3}}{1 + \nu^{2/3} e^{2\nu \tau} \abs{\xi}^2} d\tau  \lesssim \int_0^t \frac{1}{\nu^{1/3}\brak{\xi}^2} d\tau \lesssim \frac{\brak{t}^2}{\nu^{1/3}\brak{\xi}\brak{\ell}}, 
\end{align*}
which suffices in this case.
In the case $\abs{\xi} \leq 2\abs{\ell t}$, we make the change of variables $s = \nu^{1/3} \abs{\xi - \frac{\ell}{\nu}} e^{\nu \tau}$
yielding (as above, in the case $\xi = \frac{\ell}{\nu}$ the estimate trivializes)
\begin{align*}
  \abs{1 - M(t,\ell,\xi)} & \lesssim \int_{\nu^{-2/3} \abs{\nu \xi - \ell}}^{\nu^{-2/3} \abs{\nu \xi - \ell} e^{\nu t}} \frac{\nu^{2/3}s}{\brak{\nu^{2/3}s}^2} \left(\frac{1}{1 + \abs{s\frac{\nu \xi - \ell}{\abs{\nu \xi - \ell}}  + \frac{\ell}{\nu^{2/3}}}^2}\right) ds.   
\end{align*}
Subdivide the integral into the region where $\frac{\abs{\ell}}{2\nu^{2/3}} < s < \frac{2\abs{\ell}}{\nu^{2/3}}$ and its complement. Each integral can then be bounded by at least $\abs{\ell}^{-1}$, hence,
\begin{align*}
\abs{1 - M(t,\ell,\xi)} & \lesssim \frac{1}{\brak{\ell}} \lesssim \frac{\brak{t}}{\max(\brak{\ell},\brak{\xi})}, 
\end{align*}
where we also used the assumption $\abs{\xi} \lesssim \abs{\ell t}$. 
This completes the proof in the case $k = 0$ and $\ell \neq 0$.

\noindent
\textbf{The case $k\neq 0$ and $\ell \neq 0$:} In this case, we observe that again using part (a) and $\abs{e^x - 1} \leq x e^x$, we have 
\begin{align*}
\abs{1 - \frac{M(t,\ell,\xi)}{M(t,k,\eta)}} & \lesssim \int_0^t \frac{e^{2\nu \tau} \abs{\nu \xi - \ell }^2}{\brak{e^{\nu \tau} (\nu \xi - \ell)}^2} \frac{\nu^{1/3}}{1 + \nu^{2/3} e^{2\nu \tau} \abs{ \xi - \ell \frac{1- e^{-\nu \tau}}{\nu}}^2}  d\tau \\
& \quad + \int_0^t \frac{e^{2\nu \tau} \abs{\nu \eta - k }^2}{\brak{e^{\nu \tau} (\nu \eta - k)}^2} \frac{\nu^{1/3}}{1 + \nu^{2/3} e^{2\nu \tau} \abs{ \eta - k \frac{1- e^{-\nu \tau}}{\nu}}^2}  d\tau.    
\end{align*}
By the proof above and the triangle inequality, we have 
\begin{align*}
\abs{1 - \frac{M(t,\ell,\xi)}{M(t,k,\eta)}} & \lesssim \frac{\brak{t}^2}{\nu^{1/3}\brak{\ell,\xi}} + \frac{\brak{t}^2}{\nu^{1/3}\brak{k,\eta}} \lesssim \brak{k-\ell,\eta-\xi}\frac{\brak{t}^2}{\nu^{1/3} \max(\abs{k},\abs{\eta})}. 
\end{align*}
This completes the proof of \eqref{ineq:Mtrans}. 

Next, turn to the derivative estimates in part (d).  
Consider next the differentiation of $M$ with respect to $\eta$.
Compute a single derivative first: 
\begin{align*}
\abs{\grad_\eta M(t,k,\eta)} & \lesssim \int_0^t\frac{\nu  e^{2\nu \tau}\abs{\eta - k \frac{1-e^{-\nu \tau}}{\nu}} }{\left(1 + \nu^{2/3} e^{2\nu \tau}\abs{\eta - k \frac{1-e^{-\nu \tau}}{\nu}}^2\right)^2} \frac{e^{2 \nu \tau}\abs{\nu \eta - k}^2}{\brak{e^{\nu \tau}(\nu \eta - k) }^2}  d \tau \\
& \quad + \int_0^t \frac{\nu^{1/3}}{1 + \nu^{2/3} e^{2\nu \tau}\abs{\eta - k \frac{1-e^{-\nu \tau}}{\nu}}^2} \frac{\nu e^{2\nu \tau} \abs{\nu \eta - k}}{\brak{e^{\nu \tau} (\nu \eta - k)}^2}  d\tau %\\
%& \lesssim  \int_0^t\frac{\nu^{2/3}  e^{\nu \tau} }{\left(1 + \nu^{2/3} e^{2\nu \tau}\abs{\eta - k\frac{1-e^{-\nu \tau}}{\nu}}^2\right)^{3/2}}  \dd \tau. 
\end{align*}
As above the integrals vanish when $\nu\eta = k$, which implies the gradient vanishes at this point.
Making the change of variables $s = \nu^{1/3} e^{\nu \tau} \abs{\eta - \frac{k}{\nu}}$, we have
\begin{align*}
  \abs{\grad_\eta M(t,k,\eta)} & \lesssim e^{\nu t}\int_0^\infty\frac{\nu^{1/3} \abs{s \frac{\nu \eta - k}{\abs{\nu \eta -k}} + \frac{k}{\nu^{2/3}}} }{\left(1 + \abs{s \frac{\nu \eta - k}{\abs{\nu \eta -k}} + \frac{k}{\nu^{2/3}}}^2\right)^2} \frac{\nu^{2/3}s}{\brak{\nu^{2/3}s}^2}  d s \\
& \quad + e^{\nu t}\int_0^\infty \frac{\nu^{1/3}}{1 + \abs{s \frac{\nu \eta - k}{\abs{\nu \eta -k}} + \frac{k}{\nu^{2/3}}}^2} \frac{\nu^{2/3} }{\brak{\nu^{2/3}s}^2}  d s \\
& \lesssim \nu^{1/3} e^{\nu t}. 
\end{align*}
Iterating the above argument gives for all $\abs{\alpha} \geq 1$ (note that the multiplier was chosen to be infinitely differentiable, 
\begin{align*}
\abs{D_\eta^\alpha M} & \lesssim_{\abs{\alpha}} e^{\abs{\alpha}\nu t} \nu^{\abs{\alpha}/3}. 
\end{align*} 
\end{proof}

\section*{Acknowledgments}
The author would like to thank Siming He for finding an error in the original version of this manuscript. 
The author would like to thank Amitava Bhattacharjee, Michele Coti Zelati, Greg Hammett, Andrew Majda, and Toan Nguyen for helpful discussions on the problem and Amitava Bhattacharjee for suggesting that I pursue this work.  The author was partially supported by NSF CAREER grant DMS-1552826, NSF DMS-1413177, and a Sloan research fellowship. Additionally, the research was supported in part by NSF RNMS \#1107444 (Ki-Net).

\bibliographystyle{abbrv} \bibliography{eulereqns}

\end{document}